\documentclass[11pt,letterpaper]{amsart}
\usepackage[utf8]{inputenc}
\usepackage[titletoc,toc,title]{appendix}
\usepackage[left=2.4cm,right=2.4cm,top=2.5cm,bottom=2.5cm]{geometry}
\usepackage{amssymb,latexsym,amsmath,amsthm,amscd}
\usepackage{graphicx}
\usepackage{dsfont}
\usepackage{marginnote}
\usepackage[all]{xy}
\usepackage{mathrsfs}
\usepackage{color}
\definecolor{Blue}{rgb}{0.3,0.3,0.9}

\usepackage{tikz-cd}
\usepackage{hyperref}
\hypersetup{colorlinks, linkcolor=blue, citecolor=black, urlcolor=black}

\usepackage[T2A,OT1]{fontenc}
\DeclareSymbolFont{cyrillic}{T2A}{cmr}{m}{n}
\DeclareMathSymbol{\Sha}{\mathalpha}{cyrillic}{216}
\newtheorem{thm}{Theorem}[section]
\newtheorem{thmintro}{Theorem}

\newtheorem{corintro}[thmintro]{Corollary}
\newtheorem{def-thm}[thm]{Definition-Theorem}
\newtheorem{cor}[thm]{Corollary}
\newtheorem{lem}[thm]{Lemma}
\newtheorem{def-lem}[thm]{Definition-Lemma}
\newtheorem{prop}[thm]{Proposition}
\newtheorem{conj}[thm]{Conjecture}

\theoremstyle{definition}
\newtheorem{defn}[thm]{Definition}
\theoremstyle{remark}
\newtheorem{rem}[thm]{Remark}
\numberwithin{thm}{section}
\numberwithin{equation}{section}

\newcommand{\cO}{\mathcal{O}}

\newcommand{\rH}{{\mathrm{H}}}

\newcommand{\bQ}{\mathbf{Q}}
\newcommand{\bR}{\mathbf{R}}
\newcommand{\bZ}{\mathbf{Z}}

\newcommand{\bC}{\mathbf{C}}

\newcommand{\scO}{{\mathscr{O}}}

\newcommand{\hooklongrightarrow}{\lhook\joinrel\longrightarrow}

\begin{document}

%\title[Kato's main conjecture for nonordinary modular forms]{Kato's main conjecture for nonordinary modular forms and Rankin--Eisenstein classes}

\title{Kato's main conjecture for 
% modular forms at supersingular primes}
nonordinary modular forms} %in the Fontaine--Laffaille range}

\author{Francesc Castella}
\address{F. C.: University of California, Santa Barbara, CA, United States}
\email{\href{mailto:castella@math.ucsb.edu}{castella@math.ucsb.edu}}

\author{Zheng Liu}
\address{Z. L.:University of California, Santa Barbara, CA, United States}
\email{\href{mailto:zliu@math.ucsb.edu}{zliu@math.ucsb.edu}}

\author{Xin Wan}
\address{X. W.: Academy of Mathematics and Systems Science, Chinese Academy of Sciences and University of Chinese Academy of Sciences, Haidian District, Beijing, China}

\email{\href{mailto:xwan@math.ac.cn}{xwan@math.ac.cn}}

%\thanks{This material is based upon work supported by the National Science Foundation under grant agreements DMS-1801385 (F.C.), DMS-1128155 and  DMS-1352598 (M.C.), and DMS-1301842 and DMS-1501064 and by the Simons Investigator Grant \#376203 (C.S.).  
%This project has received funding from the European Research Council under the European Union's Horizon 2020 research and innovation programme under grant agreement \#682152 (F.S.).}

%\keywords{}

\date{\today}

%\dedicatory{}
%\commby{}

% ----------------------------------------------------------------

\begin{abstract}
%In this paper, 
We prove Kato's main conjecture for modular forms at nonordinary primes %$p>2$ 
for weights in the Fontaine--Laffaille range. %In particular, as a consequence of our main result in the case of weight $2$, we prove the $p$-part of the Birch--Swinnerton-Dyer formula in analytic rank zero for abelian varieties $A/\bQ$ of ${\rm GL}_2$-type.  
%Our proof is based on a reformulation of the conjecture in terms of signed Selmer groups due to Lei--Loeffler--Zerbes, and follows a strategy introduced by the third author %(and recently carried out by Burungale--Skinner--Tian--Wan \cite{BSTW})
%(not intended for publication) 
%in the case of rational elliptic curves with $a_p=0$. 
%
The key ingredients in the proof are a reformulation of the  conjecture in terms of signed Selmer groups due to Lei--Loeffler--Zerbes, certain $p$-adic families of 
Rankin--Eisenstein classes arising from the work of Lei--Loeffler--Zerbes and Kings--Loeffler--Zerbes,  %and their explicit reciprocity laws,  
%\cite{LZ-Coleman}, 
and the lower bound divisibility in an Iwasawa--Greenberg main conjecture for Rankin--Selberg $p$-adic $L$-functions  obtained in our earlier work \cite{CLW}. %for weight $2$ and squarefree level, 
%and significally generalized in this paper.
%
%Our proofs are based on an approach outlined by the third-named author in \cite{wan-ss} for elliptic curves $E/\bQ$ with $a_p=0$ and developed in full detail in recent work of Burungale--Skinner--Tian--Wan \cite{BSTW}, which we extend to non-semistable $E/\bQ$ (allowing in particular $a_3\neq 0$, which is already new in the semistable case), rational abelian varieties of ${\rm GL}_2$-type, and higher weight modular forms.
\end{abstract}

\maketitle

\setcounter{tocdepth}{2}
\tableofcontents

%{\bf Points that are still unclear / not well-written / require attention}
%\begin{enumerate}
%    \item Taking $\eta\in\hat{\Delta}$-isotypic components in some of the 2-variable constructions (BF classes, Coleman maps).
%    \item Relations between $\mathscr{F}^\pm\mathbf{D}_{\rm rig}(T_{\mathbf{h}_v})$ and $\mathscr{F}^\pm T_{\mathbf{h}_v}$.
%    \item Relation between $\eta_{\mathbf{h}_v}/c\cdot\omega_{\mathbf{h}_v}$ and $h_K\cdot\mathscr{L}_v^{{\rm Katz},-}$ (\cite[Lem.~4.36]{fouquet-wan} claims that a divisibility in $\bZ_p^{\rm nr}[[T_v]]$ (up to powers of $T_v$) follows from \cite[Prop.~8.3]{wanIMC}.
%    \item Address the unexpected irregularity $\alpha=\beta$ (see \cite[\S{4.7}]{fouquet-wan} using Betina--Williams).
%    \item Clean up/complete application to $p$-part of BSD formula. 
%    \item ...
%    \item Others that I'm forgetting.
%\end{enumerate}
%\newpage

\section{Introduction}

%\subsection{Iwasawa main conjecture for non-ordinary modular forms}

In %\cite[\S{3}]{Kato1553} (see also 
\cite[\S{4}]{kato-kodai},  
Kato formulated a vast generalisation of Iwasawa's main conjecture, conditional on a strong form of Beilinson's conjectures \cite{beilinson-higher}. For the motive attached to a newform $f\in S_k(\Gamma_0(N))$ of weight $k\geq 2$ and the cyclotomic extension  $\bQ(\mu_{p^\infty})/\bQ$, the conjecture is unconditional, see \cite[Conj.~12.10]{Kato295}.

The main result of this paper is a proof of Kato's main conjecture for newforms of weight $2\leq k< p$ at nonordinary primes $p>2$ under mild hypotheses. 

\subsection{Main results}

Let us prepare some notation for the precise statement of our main results. Let $f=\sum_{n=1}^\infty a_nq^n\in S_k(\Gamma_0(N))$ be a newform of (even) weight $k\geq 2$. %For simplicity, we assume that $f$ has trivial nebentypus. 
Fix an odd prime $p$, and put
\[
\widetilde{\Gamma}={\rm Gal}(\bQ(\mu_{p^\infty})/\bQ)\cong\Delta\times\Gamma,
\]
where $\Gamma$ %={\rm Gal}(\bQ_\infty/\bQ)$ 
is the Galois group of the cyclotomic $\bZ_p$-extension $\bQ_\infty/\bQ$, and $\Delta$ %={\rm Gal}(\bQ(\mu_p)/\bQ)$ 
is cyclic of order $p-1$. Let $F=\bQ(\{a_n\}_{n})\subset\bC$ be the number field generated by the Fourier coefficients of $f$. Fix an embedding %$\iota_\infty:\overline{\bQ}\hookrightarrow\bC$,
$\iota_p:\overline{\bQ}\hookrightarrow\overline{\bQ}_p$, where $\overline{\bQ}$ denotes the algebraic closure of $\bQ$ inside $\bC$, and let $L=F_\mathfrak{P}$ be the completion of $F$ at the prime $\mathfrak{P}$ above $p$ determined by $\iota_p$. Denote by $\scO$ the integer ring of $F$. Let $V_f\cong L^2$ be the $p$-adic Galois representation attached to $f$ by Deligne \cite{deligne-ell-adic}, and for any $G_{\bQ}$-stable $\scO$-lattice $T_f\subset V_f$, let
\begin{align*}
{\rm H}^1_{\rm Iw}(\bQ(\mu_{p^\infty}),T_f)&:=\varprojlim_r\rH^1_{}(\bQ(\mu_{p^r}),T_f),\\%\quad\quad
{\rm H}^1_{\rm Iw}(\bQ(\mu_{p^\infty}),V_f)&:={\rm H}_{\rm Iw}^1(\bQ(\mu_{p^\infty}),T_f)\otimes_{\bZ_p}\bQ_p
\end{align*}
be the Iwasawa cohomology groups for  $\bQ(\mu_{p^\infty})/\bQ$. %(as is well-known, the latter is independent of the choice of $T_f$). 
Let 
\[
\Lambda_{\scO}(\widetilde{\Gamma})=\scO[\![\widetilde{\Gamma}]\!]\cong\scO[\Delta][\![\Gamma]\!]
\]
be the cyclotomic Iwasawa algebra. In his influential work  \cite{Kato295}, Kato constructed a collection of cohomology classes $\mathbf{z}_\gamma^{(p)}\in{\rm H}^1_{\rm Iw}(\bQ(\mu_{p^\infty}),V_f)$, depending $L$-linearly on $\gamma\in V_f$, whose image under the Bloch--Kato dual exponential map is related to the special values $L(f,\chi,r)$, for $1\leq r\leq k-1$, of the complex $L$-function of $f$ twisted by finite order characters  $\chi:\widetilde{\Gamma}\rightarrow\bC^\times$ (see [\emph{op.\,cit.}, Thm.~12.5]). 

%***** MAYBE DON'T USE THIS VIEWPOINT

For $\gamma\in T_f\subset V_f$, Kato showed that $\mathbf{z}_\gamma^{(p)}$ lands in ${\rm H}_{\rm Iw}^1(\bQ(\mu_{p^\infty}),T_f)$ provided  $\rho_f:G_\bQ\rightarrow{\rm Aut}_L(V_f)\cong{\rm GL}_2(L)$ 
has ``large image''; %and is integral relative to the geometric lattice $T_f$ defined in \cite[\S{8.3}]{Kato295} in general; 
assume from now on that this is the case, let $\mathbf{z}_f^{\rm int}\in\rH^1_{\rm Iw}(\bQ(\mu_{p^\infty}),T_f)$ be the resulting integral normalisation of $\mathbf{z}_\gamma^{(p)}$ as in $\S\ref{subsec:Kato-ES}$, and denote by $(\mathbf{z}_f^{\rm int})$ the $\Lambda_{\scO}(\widetilde{\Gamma})$-span of $\mathbf{z}_f^{\rm int}$. 

Let $T_f^\vee(1)={\rm Hom}(T_f,\mu_{p^\infty})$ be the Cartier dual of $T_f$, and
\[
%{\rm H}_{{\rm Iw},{\rm str}}^2(\bQ(\mu_{p^\infty}),T_f):={\rm ker}\biggl\{{\rm H}_{\rm Iw}^2(\bQ(\mu_{p^\infty}),T_f)\longrightarrow\prod_{\eta\mid N}\rH_{\rm Iw}^2(\bQ(\mu_{p^\infty})_\eta,T_f)\biggr\},
{\rm Sel}_{{\rm str}}(f/\bQ(\mu_{p^\infty})):={\rm ker}\biggl\{{\rm H}^1(\bQ(\mu_{p^\infty}),T_f^\vee(1))\longrightarrow\prod_{\eta}\rH^1(\bQ(\mu_{p^\infty})_\eta,T_f^\vee(1))\biggr\}
\]
be the $p$-primary \emph{strict Selmer group}, %of $f$ over $\bQ(\mu_{p^\infty})$, 
where $\eta$ runs over the places of $\bQ(\mu_{p^\infty})$. In this setting,  
Kato formulated the following generalised Iwasawa main conjecture. %for $f$ (cf. \cite[Conj.~12.10]{Kato295}). 

\begin{conj}[Kato's main conjecture]\label{introconj:kato}
Let $p>2$ be a prime. For every character $\eta:\Delta\rightarrow\bZ_p^\times$, we have the equality
\[
{\rm char}_{\Lambda_{\scO}({\Gamma})}\biggl(\frac{{\rm H}^1_{\rm Iw}(\bQ(\mu_{p^\infty}),T_f)^\eta}{(\mathbf{z}_f^{\rm int})^\eta}\biggr)={\rm char}_{\Lambda_{\scO}({\Gamma})}\bigl({\rm Sel}_{{\rm str}}(f/\bQ(\mu_{p^\infty}))^{\eta,\vee}\bigr)
\]
as characteristic ideals of torsion $\Lambda_{\scO}(\Gamma)$-modules, where $M^{\eta,\vee}={\rm Hom}(M^\eta,\bQ_p/\bZ_p)$ denotes the Pontryagin dual of the the $\eta$-isotypic component of $M$.
\end{conj}

%By Kolyvagin's method applied to an Euler system extending the classes $\mathbf{z}_\gamma^{(p)}$, 

When $p$ is an odd prime of good ordinary reduction for $f$, i.e. $p\nmid N$ and $a_p\notin\mathfrak{P}$,  Conjecture~\ref{introconj:kato} is known to be equivalent to Greenberg's main conjecture \cite{greenberg-iwasawa} for the $p$-adic $L$-function of Mazur--Tate--Teitelbaum associated to the $p$-adic unit root of 
\begin{equation}\label{eq:Hecke-p}
x^2-a_px+p^{k-1}=(x-\alpha)(x-\beta).
\end{equation}
%with $\epsilon$ the nebentypus character of $f$. 
In this case, the conjectures are known under mild hypotheses by the work of several mathematicians, including Kato, Skinner--Urban, and Wan (see \cite{Kato295,SU,wan-HMF,BCS}). 

For the statement of our main result on Conjecture~\ref{introconj:kato}, let $\pi_f=\otimes_{\ell}\pi_{f,\ell}$ be the cuspidal automorphic representation of ${\rm GL}_2(\mathbb{A})$ generated by $f$, and consider the following  \emph{root number condition}:
\begin{equation}\label{eq:rtn}
%\textrm{either }
%\begin{cases}
%\textrm{$N$ is divisible by at least two primes; or}\\[0.2em]
\textrm{there exists a prime $\ell\Vert N$ such that $\pi_\ell\cong\sigma_\ell\otimes\xi_\ell$},\tag{rtn}
%\end{cases}\tag{rtn}
\end{equation}
where $\sigma_\ell$ denotes the special representation of ${\rm GL}_2(\bQ_\ell)$, and $\xi_\ell$ is the unique unramified character of $\bQ_\ell^\times$ sending $\ell$ to $-\ell^{k/2-1}$. Let $\omega:\Delta\rightarrow\bZ_p^\times$ be the Teichm\"uller character. 
 %, which gives a generator of the character groups of $\Delta$. 
%We say that $f$ is \emph{$p$-regular} if the roots $\alpha,\beta\in\overline{\bQ}_p$ are distinct. %Let $\alpha,\beta$ denote the roots of \eqref{eq:Hecke-p}. 

%\begin{thmintro}\label{thmintro:kato}
%Let $f\in S_k(\Gamma_0(N))$ be a newform of squarefree conductor $N$ and trivial character, and let $p$ be an odd prime of good non-ordinary reduction for $f$. Assume:
%\begin{itemize}
%\item $2\leq k\leq p$; 
%\item $\bar\rho_f\vert_{G_{\bQ_p}}$ is absolutely irreducible;
%\item Condition~\eqref{eq:rtn}.
%\end{itemize}
%Then 
%$X_{\rm str}(f)$ is $\Lambda$-torsion, with
%\begin{equation}\label{eq:Kato-IMC}
%{\rm char}_\Lambda\bigl(X_{\rm str}(f)\bigr)={\rm char}_\Lambda\biggl(\frac{\mathbb{H}_{\rm Iw}^1(T_f)}{\Lambda\cdot\mathbf{z}_f}\biggr)\nonumber
%\end{equation}
%as ideals in $\Lambda$, 
%\otimes\bQ_p$. If in addition $\bar{\rho}_f$ is ramified at some prime $\ell\Vert N$, then  equality \eqref{eq:Kato-IMC} holds in $\Lambda$, 
%and hence Conjecture~\ref{conj:kato} holds true.
%Conjecture~\ref{introconj:kato} holds.
%\end{thmintro}

\begin{thmintro}\label{thmintro:kato}
Let $f\in S_k(\Gamma_0(N))$ be a newform and $p>2$ be a prime  of good nonordinary reduction. %for $f$. 
Suppose
\begin{itemize}
%\item[(i)] $f$ has good nonordinary reduction at $p$;
\item $2\leq k< p$; 
\item $N$ is squarefree;
\item $f$ is \emph{$p$-regular}, i.e. $\alpha\neq\beta$;
%the roots $\alpha,\beta\in\overline{\bQ}_p$ of \eqref{eq:Hecke-p} are distinct. 
%i.e., $p$-$\alpha_p\neq\beta_p$.
%\eqref{eq:Hecke-p} has no multiple roots.
%\item $f$ has trivial nebentypus.
%\item $\bar\rho_f\vert_{G_{\bQ_p}}$ is absolutely irreducible;
\item either $\begin{cases}
\textrm{$k=2$}, or \\ %and $a_p=0$}, or\\
\textrm{Condition~\eqref{eq:rtn} holds}.
\end{cases}$
\end{itemize}
Then 
%$X_{\rm str}(f)$ is $\Lambda$-torsion, with
%\begin{equation}\label{eq:Kato-IMC}
%{\rm char}_\Lambda\bigl(X_{\rm str}(f)\bigr)={\rm char}_\Lambda\biggl(\frac{\mathbb{H}_{\rm Iw}^1(T_f)}{\Lambda\cdot\mathbf{z}_f}\biggr)\nonumber
%\end{equation}
%as ideals in $\Lambda$, 
%\otimes\bQ_p$. If in addition $\bar{\rho}_f$ is ramified at some prime $\ell\Vert N$, then  equality \eqref{eq:Kato-IMC} holds in $\Lambda$, 
%and hence Conjecture~\ref{conj:kato} holds true.
Conjecture~\ref{introconj:kato} for $\eta=\omega^{k/2}$ holds. 
%where $\omega:\Delta\rightarrow\bZ_p^\times$ is the Teichm\"uller character.
\end{thmintro}

%\begin{rem}
%remarks on the hypotheses
%\end{rem}

As explained in more detail in Section~\ref{subsec:sketch} below, two key ingredients in the proof of Theorem~\ref{thmintro:kato} are certain Rankin--Eisenstein classes arising from the work of Lei--Loeffler--Zerbes \cite{LLZ} and Kings--Loeffler--Zerbes \cite{KLZ1,KLZ2,LZ-Coleman},  %and their explicit reciprocity laws,  
%\cite{LZ-Coleman}, 
and the ``lower bound'' divisibility in an Iwasawa--Greenberg main conjecture for Rankin--Selberg $p$-adic $L$-functions obtained in our earlier work \cite{CLW}. 

%We also note that a different proof of Theorem~\ref{thmintro:kato} (still building on these ingredients) was first announced in \cite{wan-nonord} (no longer intended for publication) and in the work of the third-named author with Fouquet (cf. \cite[\S{4}]{fouquet-wan}).

\subsection{Strategy of the proof}\label{subsec:sketch}

%To motivate our approach\footnote{which is inspired by the outline given by the third-named author in \cite{wan-ss} and recently carried out in \cite{BSTW} of a proof of Kobayashi's main conjecture}, we begin by noting that
%
Our starting point for the proof of Theorem~\ref{thmintro:kato} 
is an equivalent form of Conjecture~\ref{introconj:kato} 
discovered by Lei--Loeffler--Zerbes \cite{LLZ-AJM}  building on integral $p$-adic Hodge theory, and more specifically,  Berger's theory of Wach modules \cite{berger-limit}. 
To briefly explain this, recall that Perrin-Riou's $p$-adic regulator map \cite{PR115}  
%\[
%\mathscr{L}_{V_f}:\mathbb{H}_{\rm loc}^1(V_f)\rightarrow\mathcal{H}\otimes_{\bQ_p}D(V_f)
%\]
gives a way of attaching to Kato's $\mathbf{z}_f^{\rm int}$ an element
\begin{equation}\label{eq:intro-LPR}
\mathscr{L}^{\rm PR}({\rm res}_p(\mathbf{z}_{f}^{\rm int}))\in\mathcal{H}_L(\widetilde{\Gamma})\otimes_{}\mathbb{D}_{\rm cris}(V_f)
\end{equation}
where $\mathcal{H}_L(\widetilde{\Gamma})$ %\subset\bQ_p[[\widetilde{\Gamma}]]$ 
is the algebra of tempered $L$-valued distributions on $\widetilde{\Gamma}$ and $\mathbb{D}_{\rm cris}(V_f)$ is the Dieudonn\'e--Fontaine module associated to $f$. By virtue of Kato's explicit reciprocity law, pairing \eqref{eq:intro-LPR} against any  vector $\nu\in\mathbb{D}_{\rm cris}(V_f^*(1))$ via the de Rham pairing on $\mathbb{D}_{\rm cris}(V_f)\times\mathbb{D}_{\rm cris}(V_f^*(1))$, where $V_f^*={\rm Hom}_L(V_f,L)$ is the dual of $V_f$, gives an element $\mathscr{L}^{\rm PR}_\nu({\rm res}_p(\mathbf{z}_f^{\rm int}))\in\mathcal{H}_L(\widetilde{\Gamma})$ whose values at certain characters $\chi$ of $\widetilde{\Gamma}$ is related to the critical values of $L(f,\chi,s)$. In particular, the $p$-adic $L$-functions $\mathscr{L}_{p,\xi}^{\rm MTT}$ of Mazur--Tate--Teitelbaum \cite{mtt} associated to any root $\xi\in\{\alpha,\beta\}$ of \eqref{eq:Hecke-p} with $v_p(\xi)<k-1$ are recovered in this manner. In the ordinary case, taking $\xi=\alpha$ to be the $p$-adic unit root of \eqref{eq:Hecke-p} and $\nu=\nu_\alpha$ a certain Frobenius eigenvector with eigenvalue $\alpha$, one thus obtains an equality 
\begin{equation}\label{eq:intro-ERL}
\mathscr{L}_{\nu_\alpha}^{\rm PR}({\rm res}_p(\mathbf{z}_f^{\rm int}))=\mathscr{L}_{p,\alpha}^{\rm MTT}
\end{equation}
as elements in $\Lambda_{\scO}(\widetilde{\Gamma})$ (cf. \cite[Thm.~16.6]{Kato295}). With the help of \eqref{eq:intro-ERL}, Kato was able to deduce from his results on Conjecture~\ref{introconj:kato}, a proof of a divisibility in the Iwasawa main conjecture for these bounded $p$-adic $L$-functions. A proof of the opposite divisibility---and hence a full proof of Conjecture~\ref{introconj:kato}---in the ordinary case under mild hypotheses, was then obtained in the work of Skinner--Urban \cite{SU} on the Iwasawa main conjecture for certain Rankin--Selberg $p$-adic $L$-functions related to $\mathscr{L}_{p,\alpha}^{\rm MTT}$ using Eisenstein congruences on ${\rm GU}(2,2)$. 

For the proof of Theorem~\ref{thmintro:kato} in the nonordinary case, our approach builds on Kato's work and our earlier main result \cite{CLW} on the Iwasawa main conjecture for Rankin--Selberg $p$-adic $L$-functions using Eisenstein congruences on ${\rm GU}(3,1)$.  %but a fundamental difficulty is that in this case, the image of the map $\mathscr{L}_{\nu}^{\rm PR}:\rH^1_{\rm Iw}(\bQ_p(\mu_{p^\infty}),T_f)\rightarrow\mathcal{H}$ is not contained in $\Lambda$. 
However, %in this case 
compared to the ordinary case, 
two fundamental new difficulties arise:
\begin{itemize}
\item[(1)] For any of the roots $\xi\in\{\alpha,\beta\}$ of \eqref{eq:Hecke-p}, the image of the map 
\[
\mathscr{L}_{\nu_\xi}^{\rm PR}:\rH^1_{\rm Iw}(\bQ_p(\mu_{p^\infty}),T_f)\longrightarrow\mathcal{H}_L(\widetilde{\Gamma})
\]
is not contained in $\Lambda_{\scO}(\widetilde{\Gamma})$.
\item[(2)] The Rankin--Selberg $p$-adic $L$-function
\begin{equation}\label{eq:CLW-Lp}
\mathcal{L}_{\overline{v}}(f/K)^\eta
\end{equation}
studied in \cite{CLW} has very little relation with the $p$-adic distributions  $\mathscr{L}_{p,\xi}^{\rm MTT}$ (in particular, their ranges of $p$-adic interpolation are disjoint).
\end{itemize}

To address (1), we begin by recalling that in the case $a_p=0$ (so $\beta=-\alpha$), %=\pm\sqrt{-p}^{k-1}$), %\footnote{the ``most supersingular case'', in the terminology of \cite{pollack}}, 
Lei \cite{Lei-PhD} showed that certain linear combinations of $\mathscr{L}_{\nu_\alpha}^{\rm PR}$ and $\mathscr{L}_{\nu_{-\alpha}}^{\rm PR}$ give rise to $\Lambda_{\scO}(\widetilde{\Gamma})$-linear \emph{(signed) Coleman maps}
\begin{equation}\label{eq:intro-Col}
\mathscr{C}^\pm:\rH^1_{\rm Iw}(\bQ_p(\mu_{p^\infty}),T_f)\longrightarrow\Lambda_{\scO}(\widetilde{\Gamma}) %\otimes_{\bZ_p}\bQ_p
\end{equation}
recovering Kobayashi's construction in the case of rational elliptic curves \cite{kobayashi-ss}; in particular, with the property that $\mathscr{C}^\pm({\rm res}_p(\mathbf{z}_f^{\rm int}))$ recovers Pollack's signed $p$-adic $L$-functions \cite{pollack} attached to $f$. This allowed him, as originally done in \cite{kobayashi-ss} for rational elliptic curves, to formulate (and prove results towards) signed main conjectures equivalent to Kato's main conjecture.

Abstracting this idea, we note that given a $\Lambda_{\scO}(\widetilde{\Gamma})$-linear map $\mathscr{C}$ as in \eqref{eq:intro-Col}, one can define a dual Selmer group
\begin{equation}\label{eq:conj-LLZ}
{\rm Sel}_{\mathscr{C}}(\bQ(\mu_{p^\infty}),T_f^\vee(1))\subset\rH^1(\bQ(\mu_{p^\infty}),T_f^\vee(1))
\end{equation}
%for $T_f^\vee={\rm Hom}_{\bZ_p}(T_f,\mu_{p^\infty})$ 
taking the orthogonal complement ${\rm ker}(\mathscr{C})^\perp\subset\rH^1(\bQ_p(\mu_{p^\infty}),T_f^\vee(1))$ under local Tate duality as the local condition cutting ${\rm Sel}_{\mathscr{C}}(\bQ(\mu_{p^\infty}),T_f^\vee(1))$ at the prime above $p$. In favorable circumstances, %(esp. with a good understanding of the image of $\mathscr{C}$), 
Poitou--Tate duality then gives rise to an exact sequence
\begin{align*}
0\longrightarrow\frac{{\rm H}_{\rm Iw}^1(\bQ(\mu_{p^\infty}),T_f)}{(\mathbf{z}^{\rm int}_f)}\longrightarrow&\frac{\rH^1_{\rm Iw}(\bQ_p(\mu_{p^\infty}),T_f)}{({\rm res}_p(\mathbf{z}_f^{\rm int}))}\\
&\quad\longrightarrow{\rm Sel}_{\mathscr{C}}(\bQ(\mu_{p^\infty}),T_f^\vee(1))^\vee\longrightarrow{\rm Sel}_{\rm str}(\bQ(\mu_{p^\infty}),T_f^\vee(1))^\vee\longrightarrow 0,
\end{align*}
whereby Conjecture~\ref{introconj:kato} amounts to the following statement.

\begin{conj}[$\mathscr{C}$-main conjecture]\label{introconj:LLZ}
Suppose $\eta\in\Delta\rightarrow\bZ_p^\times$ is such that 
\[
\mathscr{C}^\eta:\rH^1_{\rm Iw}(\bQ_p(\mu_{p^\infty}),T_f)^\eta\longrightarrow\Lambda_{\scO}(\Gamma) %\otimes_{\bZ_p}\bQ_p
\]
has $\Lambda_{\scO}(\Gamma)$-torsion cokernel with characteristic ideal generated by $\xi_\eta\in\Lambda_{\scO}(\Gamma)$, and 
\[
\mathscr{C}^\eta({\rm res}_p(\mathbf{z}_f^{{\rm int},\eta}))\neq 0,
\]
where $\mathbf{z}_f^{{\rm int},\eta}$ is the projection of $\mathbf{z}_f^{\rm int}$ to the $\eta$-isotytic component of $\rH^1_{\rm Iw}(\bQ(\mu_{p^\infty}),T_f)$. 
Then the Pontryagin dual %$X_\mathscr{C}(f/\bQ_\infty)$ of ${\rm Sel}_{\mathscr{C}}(\bQ(\mu_{p^\infty}),T_f^\vee(1))^\eta$ 
${\rm Sel}_{\mathscr{C}}(\bQ(\mu_{p^\infty}),T_f^\vee(1))^{\eta,\vee}$
is $\Lambda_{\scO}(\Gamma)$-torsion, with
\[
(\xi_\eta)\cdot{\rm char}_{\Lambda_{\scO}(\Gamma)}\bigl({\rm Sel}_{\mathscr{C}}(\bQ(\mu_{p^\infty}),T_f^\vee(1))^{\eta,\vee}\bigr)=\bigl(\mathscr{C}^\eta({\rm res}_p(\mathbf{z}_f^{{\rm int},\eta}))\bigr).
\]
%as principal ideals in $\Lambda_{\scO}(\Gamma)$.
\end{conj}

%Rather than attempting to deduce from \cite{CLW} a proof   of an Iwasawa main conjecture\footnote{however, see \cite{wan-nonord} and \cite{FW} for an approach along these lines, building on Potthart's theory of analytic Iwasawa cohomology.} for $\mathscr{L}_{p,\alpha}^{\rm MTT}$, 
%Building on this idea, 
Our proof of Theorem~\ref{thmintro:kato} is deduced from a proof of   Conjecture~\ref{introconj:LLZ} for one of the Coleman maps $\mathscr{C}^{i,\eta}_{\rm LLZ}$ ($i\in\{1,2\}$) 
constructed by Lei--Loeffler--Zerbes \cite{LLZ-AJM}. 
%based on an understanding the $\Lambda$-module structure of Berger's Wach module $\mathbb{N}(T_f)$. 
Bypassing the use of Honda theory of formal groups in Kobayashi's original construction of signed Coleman maps, the much more general (albeit less explicit) construction of $\mathscr{C}_{\rm LLZ}^{i,\eta}$ is based on an understanding of the $\Lambda_{\scO}(\widetilde{\Gamma})$-module structure of Berger's Wach module $\mathbb{N}(T_f)$. 

As for (2), to deduce from the divisibility in \cite{CLW} a corresponding divisibility in Conjecture~\ref{introconj:LLZ} for $\mathscr{C}_{\rm LLZ}$, we use certain Beilinson--Flach classes (denoted $\mathcal{BF}^{j,\mathbf{h}_v,\eta}_{\rm int}$ in the body of the text) 
\begin{equation}\label{eq:BF-signed}
\mathcal{BF}^{j,\eta}\in\rH^1_{\rm Iw}(K_\infty,T_f(\eta)),\quad j\in\{1,2\}
\end{equation}
attached to the base-change of $f$ to an auxiliary imaginary quadratic field $K/\bQ$ in which $p=v\overline{v}$ splits, where $K_\infty$ denotes the $\bZ_p^2$-extension of $K$. The construction of \eqref{eq:BF-signed} is deduced from the works of Kings--Loeffler--Zerbes \cite{KLZ2}, Loeffler--Zerbes \cite{LZ-Coleman}, and B\"uy\"ukboduk--Lei \cite{BL-non-ord}, together with the study in \cite{BSTW} of lattices in the  big Galois representations attached to a certain Hida family with CM by $K$. 
%As in the approach introduced by the third-named author in \cite{wan-ss} in the case of rational elliptic curves with $a_p=0$ (no longer intended for publication), 
Building on the explicit reciprocity laws satisfied by these classes, for $i\in\{1,2\}$ and $w\in\{v,\overline{v}\}$ we then construct two-variable signed Coleman and logarithm maps
\begin{equation}\label{eq:Col-Log}
\begin{aligned}
\mathscr{C}_{i,w}^\eta:\prod_{\tilde{w}\vert w}\rH^1_{\rm Iw}(K_{\infty,\tilde{w}},T(\eta))&\longrightarrow\Lambda_{\scO}(\Gamma_K),\quad\quad
\mathscr{L}_{i,w}^\eta:{\rm ker}(\mathscr{C}_{i,w})\longrightarrow\Lambda_{\scO^{\rm ur}}(\Gamma_K),
\end{aligned}
\end{equation}
where $\mathscr{O}^{\rm ur}=\mathscr{O}\hat\otimes_{\bZ_p}W(\overline{\mathbf{F}_p})$, connecting: 
\begin{itemize}
\item ${\rm res}_v(\mathcal{BF}^{i,\eta})$ (upon restriction to the cyclotomic $\bZ_p$-extension of $K$) to the product 
\[
\mathscr{C}^{i,\eta}_{\rm LLZ}({\rm res}_p(\mathbf{z}_f^{{\rm int},\eta}))\cdot\mathscr{C}^{i,\eta}_{\rm LLZ}({\rm res}_p(\mathbf{z}_{f\otimes\epsilon_K}^{{\rm int},\eta})),
\]
for $f\otimes\epsilon_K$ the twist of $f$ by the quadratic character attached to $K$; %(upon restriction to the cyclotomic $\bZ_p$-extension of $K$); 
and 
\item ${\rm res}_{\overline{v}}(\mathcal{BF}^{j,\eta})$ for $j\neq i$ to the $p$-adic $L$-function 
\[
\mathcal{L}_{\overline{v}}(f/K)^\eta,
\]
%in \eqref{eq:CLW-Lp},
%studied in \cite{CLW},
\end{itemize}
respectively. From here, the desired divisibility transfer from \cite{CLW} follows easily by global duality. 
%Thus the approach adapts the strategy introduced by the third-named author in \cite{wan-ss} towards Kobayashi's main conjectures, although significant difficulties arise owing to the different construction of the signed Coleman maps and the inexplicit nature of their interpolation properties.
Together with Kato's opposite divisibility, we thus arrive at the proof of Theorem~\ref{thmintro:kato}.

\subsection{Applications} 

As a consequence of Theorem~\ref{thmintro:kato}, we deduce cases of the $p$-part of the Tamagawa Number Conjecture \cite{BK,BF-doc-math}  for modular forms in analytic rank zero. 

\begin{thmintro}\label{thmintro:TNC}
%Let $f\in S_{k}(\Gamma_0(N))$ and $p>2$ be as in Theorem~\ref{thmintro:kato}, and suppose $L(f,k/2)\neq 0$.
Let $f\in S_k(\Gamma_0(N))$ be a newform and $p>2$ be a prime  of good nonordinary reduction. %for $f$. 
Suppose
\begin{itemize}
%\item[(i)] $f$ has good nonordinary reduction at $p$;
\item $2\leq k< p$; 
\item $N$ is squarefree;
\item $f$ is \emph{$p$-regular}, i.e. $\alpha\neq\beta$;
\item either $\begin{cases}
\textrm{$k=2$}, or \\ %and $a_p=0$}, or\\
\textrm{Condition~\eqref{eq:rtn} holds}.
\end{cases}$
\end{itemize}
If $L(f,k/2)\neq 0$ then the Bloch--Kato Selmer group ${\rm Sel}(\bQ,T_f^\vee(1-k/2))$ is finite, and 
%\[
%{\rm ord}_{\varpi}\biggl(\frac{L(f,k/2)}{\Omega_{\omega}^{\pm}}\biggr)={\rm length}_{\scO}\bigl({\rm Sel}(\bQ,T_f^\vee(1-k/2))\bigr)+\sum_{\ell\mid N}{\rm length}_{\scO}\bigl(\rH^1_{\rm unr}(\bQ_\ell,T_f(k/2))\bigr)
%\]
\[
\left[\mathscr{O}\colon\frac{L(f,k/2)}{(-2\pi i)^{k/2}\cdot\Omega_{\omega_f}^\pm}\right]=\#{\rm Sel}(\bQ,T_f^\vee(1-k/2))\cdot\prod_{\ell\mid N}{\rm Tam}_\ell(T_f(k/2)).
\]
where $\pm=(-1)^{k/2-1}$, $\Omega_{\omega_f}^\pm\in\bC^\times$ is the period %(determined up to a unit in $\scO^\times$) 
of Definition~\ref{def:p-periods}, and
\[
{\rm Tam}_\ell(T_f(k/2))=\left[\rH^1_f(\bQ_\ell,T_f(k/2)):\rH^1_{\rm unr}(\bQ_\ell,T_f(k/2))\right]
%\#\biggl(\frac{\rH^1_f(\bQ_\ell,T_f(k/2))}{\rH^1_{\rm unr}(\bQ_\ell,T_f(k/2))}\biggr)
\]
is the Tamagawa factor of $T_f(k/2)$ at $\ell$.
\end{thmintro}

%Specialised to the case of weight $2$, Theorem~\ref{thmintro:TNC} yields the following result on the Birch--Swinnerton-Dyer conjecture for abelian varieties.
In particular, specialised to the case where $f$ is of weight $k=2$ %\footnote{In which case the $p$-regularity of $f$ is known by \cite{coleman-edixhoven}.} 
with  rational Fourier coefficients, Theorem~\ref{thmintro:TNC} recovers the main result of \cite{BSTW} on the $p$-part of the Birch--Swinnerton-Dyer formula in analytic rank zero,
and extends it to the case $a_p\neq 0$:

\begin{corintro}\label{corintro:BSD-E}
Let $E/\bQ$ be a semistable elliptic curve of conductor $N$, and $p>2$ be a  prime of good nonordinary reduction for $E$. If $L(E,1)\neq 0$ then $E(\bQ)$ and  $\Sha(E/\bQ)[p^\infty]$ are both finite, and
\[
{\rm ord}_p\biggl(\frac{L(E,1)}{\Omega_E}\biggr)={\rm ord}_p\biggl(\#\Sha(E/\bQ)[p^\infty]\cdot\prod_{\ell\mid N}c_\ell(E)\biggr),
\]
where $\Omega_E\in\mathbb{R}_{>0}$ is the positive N\'{e}ron period. Hence the $p$-part of the Birch--Swinnerton-Dyer formula holds for $E$.
\end{corintro}

%\begin{proof}
%Let $f\in S_2(\Gamma_0(N))$ be the newform associated to $E$ by \cite{BCDT}. In this case, the $p$-regularity hypothesis is known to hold by \cite{coleman-edixhoven}. Since the period $(-2\pi i)\cdot\Omega_{\omega_f}^+$ in Theorem~\ref{thmintro:TNC} associated to $f$ 
%is known to agree up to a $p$-adic unit with %the positive N\'{e}ron period  $\Omega_E\in\mathbb{R}_{>0}$, 
%$\Omega_E$ (see for example the discussion in \cite[\S{3.3}]{skinner-mult}), the result follows.
%\end{proof}

%Indeed, this follows from the fact that the period $\Omega_{\omega}^+$ in Corollary~\ref{corintro:BSD-A} can be taken to agree (up to a $p$-adic unit) with the positive N\'{e}ron period  $\Omega_E\in\mathbb{R}_{>0}$.

%\begin{proof}
%Immediate from Theorem~\ref{thmintro:BSD}, since in this case we can take $\Omega_{\gamma,\omega}$ to agree (up to a $p$-adic unit) with $\Omega_E$.
%\end{proof}

\begin{rem}
%Building on a different formulation of signed Iwasawa main conjectures, 
An earlier result towards the Birch--Swinnerton-Dyer conjecture for elliptic curves $E/\bQ$ at nonordinary primes $p>2$ with $a_p\neq 0$ was obtained in %\cite[Thm.~C]{CCSS} and 
\cite[Thm.~5.3]{sprung-IMC} conditional %on a conjecture about 
on the existence of signed Beilinson--Flach classes and   reciprocity laws relating them to certain $p$-adic $L$-functions (see Conjecture 3.33 in \emph{op.\,cit.}) akin to some of the results 
in $\S\ref{sec:BF}$ below.
\end{rem}

\begin{rem}
It should be possible to deduce from Theorem~\ref{thmintro:TNC} an extension of Corollary~\ref{corintro:BSD-E} to abelian varieties $A/\bQ$ of ${\rm GL}_2$-type (cf. \cite[Thm.~C]{CCSS}). 
%We leave the details to the interested reader.
\end{rem}

We also note that by a method introduced by Wei~Zhang in the ordinary case \cite{zhang-Kolyvagin}, Theorem~\ref{thmintro:TNC} also has applications to variants of Kolyvagin's conjecture \cite{kolyvagin-Sel}  
for nonordinary primes and higher weight modular forms as in the work of N.~Sweeting \cite{sweeting} and  E.~Da Ronche \cite{daronche}; %among others; 
and together with Theorem~\ref{thmintro:kato}, it should have applications to higher weight analogues of Kurihara's nonvanishing conjectures \cite{kurihara-TATA} and their refinements, as studied in \cite{BCGS,kim-AJM,kim-refined-GL2,c-sano}.

\subsection{Relation to prior work}

In \cite{wan-ss}, the third-named author outlined a proof of Kobayashi's signed main conjectures for elliptic curves $E/\mathbf{Q}$ 
at supersingular primes $p>2$ with $a_p=0$. 
%\begin{equation}\label{eq:ap=0}
%a_p:=p+1-\#E(\bF_p)=0.%\tag{most-SS}
%\end{equation} 
%(Note that this condition is automatic if $p>3$ as a consequence of the Hasse bounds.) 
At the time, the approach was based on a combination of the following three ingredients:
\begin{enumerate}
\item The ``upper bound'' divisibility in the signed main conjectures proved by Kobayashi \cite{kobayashi-ss}; %building on Kato's work \cite{Kato295};
\item The ``lower bound'' divisibility in an Iwasawa--Greenberg main conjecture for certain Rankin--Selberg convolutions obtained in the companion paper \cite{wan-nonord-GU31} via Eisenstein congruences on ${\rm GU}(3,1)$;
\item The construction of ``signed Beilinson--Flach elements'' attached to $E$ and an auxiliary imaginary quadratic field $K$, deduced from the work of Kings--Loeffler--Zerbes \cite{KLZ2} and Loeffler--Zerbes \cite{LZ-Coleman} and their explicit reciprocity laws, 
%relating them to the $p$-adic $L$-functions in (1) and (2), thus 
allowing to transfer the divisibility in (2) to a corresponding divisibility in (1). %thereby concluding the proof of the equality of characteristic ideals predicted by the main conjectures.
\end{enumerate}

Since then, %the Hida theory of semi-ordinary forms and (a slightly modified version of)
the approach to the lower bound divisibility in \cite{wan-nonord-GU31} (in particular, the foundational Hida theory for semi-ordinary forms) has been  developed in \cite{CLW}, while the construction of signed Beilinson--Flach elements for rational elliptic curves with $a_p=0$ %with the desired properties 
has been established in  \cite{BSTW} building on related work of B\"{u}y\"{u}kboduk--Lei \cite{BL-non-ord} and the detailed analysis of the geometry of the Coleman--Mazur  eigencurve at certain weight one points by Betina--Dimitrov--Pozzi \cite{betina-dimitrov-pozzi}. 
%thereby yielding a proof of the following theorem anticipated in \cite[Thm.~1.3]{wan-ss}.

%\begin{thmintro}[Burungale--Skinner--Tian--%Wan]\label{introthm:BSTW}
%Let $E/\bQ$ be an elliptic curve, and $p>2$ a prime of good supersingular reduction for $E$. Assume that $E$ is semistable, and if $p=3$, assume also that $a_3=0$. Then Kobayashi's signed main conjectures \cite[$\S{4}$]{kobayashi-ss} for the trivial character $\eta$ of $\Delta$ hold.
%\end{thmintro}

%Here $\Lambda=\bZ_p[\![\Gamma]\!]$ denotes the cyclotomic Iwasawa algebra, with $\Gamma={\rm Gal}(\bQ_\infty/\bQ)$ the Galois group of the cyclotomic $\bZ_p$-extension of $\bQ$, $\mathcal{L}^\circ(E)\in\Lambda$ denotes the signed $p$-adic $L$-function constructed by Pollack \cite{pollack}, and $\eta$ stands for a character of the torsion part of ${\rm Gal}(\bQ(\mu_{p^\infty})/\bQ)$. 

%******** 
Kobayashi's formulation of Iwasawa main conjectures in the supersingular setting, and his proof of an upper bound  divisibility building on Kato's work, has been extended to other settings, including:
\begin{itemize}
\item[(1)] Rational elliptic curves with good supersingular reduction at $p>2$ (allowing nonzero $a_p$ for $p=3$) in work of Sprung \cite{sprung-JNT};
\item[(2)] Elliptic newforms $f\in S_k(\Gamma_0(N))$ of weight $k\geq 2$ with $a_p=0$ in work of Lei \cite{Lei-PhD}; 
\item[(3)] Elliptic newforms $f\in S_k(\Gamma_0(N))$ of weight $k\geq 2$ with nonordinary reduction at $p>2$ (with $a_p$ not necessarily zero) in work Lei--Loeffler--Zerbes \cite{LLZ-AJM}.
\end{itemize} 

%In this paper we prove an analogue of Theorem~\ref{introthm:BSTW} in these more general settings. In fact, we will 
As mentioned above, to arrive at Theorem~\ref{thmintro:kato} we prove one of the signed main conjectures of Lei--Loeffler--Zerbes \cite{LLZ-AJM}; %for non-ordinary modular forms formulated in \cite{LLZ-AJM} 
since these are known to specialise to conjectures equivalent to Lei's %signed main conjectures 
in the case $a_p=0$, and to Kobayashi's and Sprung's %signed main conjectures 
in the case of rational elliptic curves, Theorem~\ref{thmintro:kato} also yields a proof of the result anticipated in \cite[Thm.~1.3]{wan-ss}\footnote{See also \cite{BSTW} for a recent detailed proof in this case building on \cite{CLW}.}  (not intended for publication) and its corresponding extension to the settings (1)--(3) above. Note however that the approach outlined in \emph{op.\,cit.} (and extended by Sprung \cite{sprung-IMC} to setting of rational elliptic curves at supersingular primes $p>2$ as in \cite{sprung-JNT}) differs markedly from the approach carried out in this paper. In particular, our construction of the signed logarithm maps \eqref{eq:Col-Log} and the proof of the associated explicit reciprocity law connecting ${\rm res}_{\overline{v}}(\mathcal{BF}^{j,\eta})$ to $\mathcal{L}_{\overline{v}}(f/K)^\eta$ is new (cf.  $\S\S\ref{subsec:ERL}$ and \ref{subsec:transfer}).

Finally, %since the main result of \cite{CLW} applies for general good primes, 
we note that it should be possible to adapt the methods in this paper to give a new proof %(under slightly different hypotheses) 
of the main result of \cite{SU} towards the two-variable Iwasawa main conjecture for modular forms in the $p$-ordinary case 
%under slightly different hypotheses 
(cf. \cite{CW-adv,fouquet-wan}). %(indeed, the main result of \cite{CLW} applies for general good primes . 
%We plan to study this in a future work.

%By the known equivalence between the signed main conjectures of \cite{LLZ-AJM} and Kato's main conjecture \cite[Conj.~12.10]{Kato295}, our results yield a proof of the latter in great generality when combined with the deformation-theoretic arguments developed in \cite{fouquet-wan}. 
%(in particular allowing primes $p>2$ of additive reduction).

\begin{rem}
In \cite{wan-nonord} (incorporated as part of \cite{fouquet-wan}), the third-named author outlined a different proof of Theorem~\ref{introconj:kato}. %In contrast to our method in this paper, %the strategy in 
Contrary to the approach via signed Iwasawa theory carried out in this paper, the approach in \emph{op.\,cit.} 
%%\emph{loc.\,cit.} is to deduce from \cite{Kato295}, \cite{CLW}, and 
directly builds on the \emph{unbounded} Beilinson--Flach classes of Loeffler--Zerbes \cite{LZ-Coleman} and a so-called ``unramified Iwasawa theory'' along an appropriately chosen line in a $2$-variable $p$-adic deformation of $T_f$. %The method here is more elegant and essential conceptually. %which is known to be equivalent to Conjecture~\ref{introconj:kato}. %by \cite{pott-cyc}). 
%
%The approach in this paper is 
%inspired by the outline given in \cite{wan-ss} (and recently carried out in \cite{BSTW}) for the proof of Kobayashi's main conjecture, and by 
% inspired by the work of the first-named author with \c{C}iperiani, Skinner, and Sprung \cite{CCSS} in the case of weight $2$.
\end{rem}

\subsection{Remarks on the hypotheses}

We give some comments on the need for the hypotheses in our Theorem~\ref{thmintro:kato} (and its applications).
\begin{enumerate}
\item As already noted, it should be possible to extend Theorem~\ref{thmintro:kato} to good ordinary primes $p>2$ by similar methods. (In some sense, this case is simpler, since there is no `signed decomposition' to be taken, and one can directly use the explicit reciprocity laws from  \cite{KLZ2}.) 
\item The weight restriction $k<p$ arises from our use of integral $p$-adic Hodge theory, in particular, the choice of a good basis of the Wach module $\mathbb{N}(T_f)$ using Fontaine--Laffaille theory. This condition (together with $p$-nonordinarity) implies irreducibility of the residual representation $\bar{\rho}_f\vert_{G_{\bQ_p}}$, a condition that should be added in general. It might be possible to go beyond the Fontaine--Laffaille range by building on very recent advances in integral $p$-adic Hodge theory\footnote{We are grateful to Matthias Flach for informative remarks concerning this point.}, a direction that we intend to explore in future work.
\item The hypothesis that $N$ be squarefree is carried over from \cite{CLW}, and it should be removable with some more work.
\item The $p$-regularity hypothesis on $f$ is carried over from \cite{LZ-Coleman}, but it is also used in some of our arguments here. It might be possible to remove it building on a different construction of $p$-adic $L$-functions interpolating Rankin--Selberg $L$-values, as explained in  \cite[\S{4.7}]{fouquet-wan}.
\item The restriction to the $\eta=\omega^{k/2}$-isotypic component for the action of $\Delta$ arises from our use of anticyclotomic $p$-adic $L$-functions interpolating (twists of) the central $L$-values $L(f/K,k/2)$, and seems essential to our method.
\item Finally, the last hypothesis in Theorem~\ref{thmintro:kato} could be significantly relaxed by a suitable application of the anticyclotomic $\mu=0$ results of \cite{ChHs1} in higher weight (cf. $\S\ref{subsec:ac-descent-def}$). 
%extending the application of these results in the weight $2$ (adding to our application of these result for $k=2$ in Section~\ref{subsec:ac-def}).
%similar results in \cite{hsieh} and \cite{vatsal-special}, respectively).
\end{enumerate}

\subsection{Acknowledgements}
We would like to thank Yukako Kezuka, Masato Kurihara, and Takamichi Sano for helpful discussions, especially in relation to the application of our results to the Tamagawa number conjecture. We would also like to thank Matthias Flach, Chan-Ho Kim, and Antonio Lei for useful comments on an earlier draft. 

During the preparation of this paper, F.C. was partially supported by the NSF grant DMS-2401321 and the 2024--2025 AMS Centennial Research Fellowship; Z.L. was partially supported by the NSF grant DMS-2501507; X.W. was partially supported by the NSFC grants 12288201, 11621061, CAS Project for
Young Scientists in Basic Research grant no. YSBR-033, and the Strategic Priority Research Program of Chinese Academy of Sciences
under Grant XDA0480503.

%\newpage

%\input{1intro-c}

%\section{Preliminaries}\label{sec:prelim}
%\input{2prelim-c}

%\section{Two-variable Iwasawa theory}
%\input{3ERL-c}

%\section{Main results}
%\input{4proofs-c}

%\newpage

\section{Preliminaries}\label{sec:prelim}

The goal of this section is to give a precise formulation of Conjectures~\ref{introconj:kato} and \ref{introconj:LLZ} in our setting and explain their equivalence.  
%The results in this section are essentially contained in the work of Lei--Loeffler--Zerbes, see esp. \cite{LLZ-AJM,LLZ-ANT,LLZ-Sha}.
%
%Our proof of Kato's main conjecture for supersingular primes will rely on a reformulation due to Lei--Loeffler--Zerbes \cite{LLZ-AJM} in terms of signed Selmer groups defined using Wach modules. In this section, we briefly review the aspects of their work that will be needed for our later arguments. %later in the paper.
%
%
Throughout we fix a prime $p>2$, and complex and $p$-adic embeddings $\iota_\infty:\overline{\bQ}\hookrightarrow\bC$, $\iota_p:\overline{\bQ}\hookrightarrow\overline{\bQ}_p$.
%while setting up our notations and conventions along the way.

\subsection{Modular forms}\label{subsec:modforms}

Let $f=\sum_{n=1}^\infty a_nq^n\in S_k(\Gamma_0(N))$ be a  newform of weight $k\geq 2$, level $N$ with $p\nmid N$, and trivial nebentypus. Throughout we assume that
\begin{equation}\label{eq:FL}
k<p.\tag{FL}
\end{equation}

Let $F=\bQ(\{a_n\}_{n})$ be the Hecke field of $f$, viewed inside $\overline{\bQ}$ via $\iota_\infty$, and let $L$ denote the completion of the image of ${F}$ under $\iota_p$. Let $\scO$ be the ring of integers of $L$, and assume that  $f$ is \emph{nonordinary} at $p$, meaning that 
\begin{equation}\label{eq:nonord}
\vert\iota_p(a_p)\vert<1.\nonumber
\end{equation}

Let $V_f$ be the $L$-linear Galois representation associated to $f$ by Deligne. % \cite{deligne-ell-adic}. 
We adopt the convention that the $p$-adic cyclotomic character $\varepsilon:G_{\bQ_p}\rightarrow\bZ_p^\times$ 
has Hodge--Tate weight $+1$, so the Hodge--Tate weights of  $V_f\vert_{G_{\bQ_p}}$ are $0$ and $1-k$.  Let $T_f$ denote the canonical $G_\bQ$-stable $\scO$-lattice in $V_f$ defined by Kato in \cite[\S{8.3}]{Kato295}.

\subsection{Iwasawa algebras}

Let $\widetilde{\Gamma}$ be the Galois group of the cyclotomic $\bZ_p^\times$-extension $\bQ(\mu_{p^\infty})/\bQ$; this decomposes as a direct product
\[
\widetilde{\Gamma}\cong\Delta\times\Gamma,
\]
where $\Gamma$ denotes the Galois group of the cyclotomic $\bZ_p$-extension $\bQ_\infty/\bQ$, and $\Delta={\rm Gal}(\bQ(\mu_p)/\bQ)$ is cyclic of order $p-1$. %\cong\bZ/(p-1)\bZ$. 
Put
\[
%\widetilde{\Lambda}=
\Lambda_{\scO}(\widetilde{\Gamma}):=\scO[\![\widetilde{\Gamma}]\!]\cong\scO[\Delta][\![\Gamma]\!],\quad\quad
%\Lambda=
\Lambda_{\scO}(\Gamma):=\scO[\![{\Gamma}]\!]
\]
for the cyclotomic Iwasawa algebras. 
%and set  $\widetilde{\Lambda}=\Lambda_{\cO}(\widetilde{\Gamma})$ and $\Lambda=\Lambda_{\cO}(\Gamma)$ for the ease of notation.  
Upon the choice of a topological generator $\gamma\in\Gamma$, we shall often identify $\Lambda$ with the formal  power series ring $\scO[\![X]\!]$ via $X=\gamma-1$.

Let $\mathcal{H}_L\subset L[\Delta][\![X]\!]$ be the subring of power  series convergent on the $p$-adic unit disk $\vert X\vert<1$ and put
\[
\mathcal{H}_L(\Gamma)=\{h(\gamma-1)\;\vert\;h\in\mathcal{H}_L\};
\]
this is identified with the ring of locally analytic $L$-valued distributions on $\Gamma$, and it naturally contains $\Lambda_{\scO}(\Gamma)$ as the subring of $\scO$-valued measures on $\Gamma$. Put also 
\[
\Lambda_L(\Gamma):=L\otimes_{\scO}\Lambda_\scO(\Gamma),
\] 
which is similarly identified with the subring of bounded elements in $\mathcal{H}_L(\Gamma)$, and 
\[
\Lambda_L(\widetilde{\Gamma}):=L[\Delta]\otimes_{L}\Lambda_L(\Gamma),\quad\quad\mathcal{H}_L(\widetilde{\Gamma}):=L[\Delta]\otimes_L\mathcal{H}_L(\Gamma).
\]

\subsection{Signed Coleman maps}

In this section we recall the definition of signed Coleman maps originating from the work of Lei--Loeffler--Zerbes \cite{LLZ-AJM}, where such maps are constructed from suitable bases of the Wach module of $T_f$ exploiting its $\Lambda_{\scO}(\widetilde{\Gamma})$-module structure through the Mellin transform.

%Following the work of Lei--Loeffler--Zerbes \cite{LLZ-AJM}, a construction of ``signed'' Coleman maps for $f$ arises from a choice of basis of the Wach module of $V_f$. These Coleman maps are used both to decompose the unbounded $p$-adic $L$-function of $f$ and to define bounded signed Selmer groups over the cyclotomic $\bZ_p$-extension. In this section we recall the construction and a good choice of basis for which one has a tight control on the image of the resulting Coleman maps. 

\subsubsection{Wach modules}

Put $\mathbb{A}_{\bQ_p}^+=\scO[\![\pi]\!]$ for the power series ring\footnote{Note that this ring is most often denoted $\scO\otimes_{\bZ_p}\mathbb{A}_{\bQ_p}^+$ with $\mathbb{A}_{\bQ_p}^+=\bZ_p[\![\pi]\!]$, but since our coefficient ring will be fixed to be $\scO$, the extension of scalars will be omitted for the ease of notation.} over $\scO$ in the  variable $\pi$.  
The ring $\mathbb{A}_{\bQ_p}^+$ is equipped with $\scO$-linear actions of $\varphi$ and $\widetilde{\Gamma}$ defined by 
\[
\varphi(\pi)=(1+\pi)^p-1,\quad\quad\sigma(\pi)=(1+\pi)^{\varepsilon(\sigma)}-1 
\]
for $\sigma\in\widetilde{\Gamma}$. 
%where $\chi:G_{\bQ_p}\rightarrow\bZ_p^\times$ is the cyclotomic character. 
Put  $\mathbb{B}_{\bQ_p}^+=\mathbb{A}_{\bQ_p}^+[1/p]\subset L[\![\pi]\!]$, and let $\mathbb{B}_{{\rm rig},\bQ_p}^+$ be the subring of $L[\![\pi]\!]$ of formal power series convergent in the $p$-adic unit disk $\vert\pi\vert<1$. The actions of $\varphi$ and $\widetilde{\Gamma}$ extend to these larger rings, and %in all these rings 
$\varphi$ has a left inverse $\psi$ satisfying
\[
(\varphi\circ\psi)(g)(\pi)=\frac{1}{p}\sum_{\zeta\in\mu_p}g(\zeta(1+\pi)-1).
\]
One easily checks that $\psi(1+\pi)=0$, and that the action of $\widetilde{\Gamma}$ on $1+\pi$ extends to isomorphisms
\begin{equation}\label{eq:mellin}
\Lambda_{\scO}(\widetilde{\Gamma})\cong(\mathbb{A}_{\bQ_p}^+)^{\psi=0},\quad\quad\Lambda_L(\widetilde{\Gamma})\cong(\mathbb{B}_{\bQ_p}^+)^{\psi=0},\quad\quad\mathcal{H}_L(\widetilde{\Gamma})\cong(\mathbb{B}_{{\rm rig},\bQ_p}^+)^{\psi=0}
\end{equation}
of $\Lambda_{\scO}(\widetilde{\Gamma})$-, $\Lambda_L(\widetilde{\Gamma})$-, and $\mathcal{H}_L(\widetilde{\Gamma})$-modules, respectively, all of which will be denoted by $\mathfrak{M}$ and referred to as the \emph{Mellin transform}.

Let $\mathbb{N}(T_f)\subset\mathbb{N}(V_f)$ be the Wach modules attached to the $G_{\bQ_p}$-representations  $T_f\subset V_f$ by \cite{berger-limit}. These are free rank $2$ modules over $\mathbb{A}_{\bQ_p}^+$ and $\mathbb{B}_{\bQ_p}^+$, respectively, with compatible actions of $\varphi$ and $\widetilde{\Gamma}$. By Corollary~III.4.5 in \emph{op.\,cit.} there is an isomorphism of filtered $\varphi$-modules
\begin{equation}\label{eq:Berger-iso-V}
\mathbb{N}(V_f)/\pi\mathbb{N}(V_f)\cong\mathbb{D}_{\rm cris}(V_f),
\end{equation} 
and the image of $\mathbb{N}(T_f)$ under this isomorphism defines an $\scO$-lattice $\mathbb{D}_{\rm cris}(T_f)\subset\mathbb{D}_{\rm cris}(V_f)$.

\begin{lem}\label{lem:good-int-basis}
%Let $f\in S_k(\Gamma_0(N))$ be a newform satisfying
%\begin{equation}\label{eq:FL}
%k< p.\tag{FL}
%\end{equation}
There exists an $\scO$-basis $(v_1,v_2)$ of $\mathbb{D}_{\rm cris}(T_f(k-1))$ with 
\begin{itemize}
\item[(i)] $v_1$ an $\scO$-module generaotor of ${\rm Fil}^0\,\mathbb{D}_{\rm cris}(T_f(k-1))\subset\mathbb{D}_{\rm cris}(T_f(k-1))$;
\item[(ii)] $v_2=\varphi(v_1)$. 
\end{itemize}
In particular, the matrix $A_\varphi$ of $\varphi$ %acting on $\mathbb{D}_{\rm cris}(V_f(k-1))$ 
with respect to such basis $(v_1,v_2)$ is given by
\[
A_\varphi=\left(\begin{array}{cc}0&-1/p^{k-1}\\
1& a_p/p^{k-1}
\end{array}
\right).
\]
\end{lem}

\begin{proof}
The existence of bases $(v_1,v_2)$ as claimed under the Fontaine--Laffaille condition \eqref{eq:FL} is shown in \cite[Lem.~3.1]{LLZ-Sha}. Note that the proof uses the fact that $\mathbb{D}_{\rm cris}(T_f(k-1))$ is then a \emph{strongly divisible} lattice of $\mathbb{D}_{\rm cris}(V_f(k-1))$ in the sense of \cite[\S{7.7}]{FL}. The form of $A_\varphi$ then follows from the relation $\varphi^2=a_pp^{1-k}\varphi-p^{1-k}$.
\end{proof}

Note that by definition %of $\mathbb{N}(T_f(k-1))$ 
we have an isomorphism
\begin{equation}\label{eq:Berger-iso-T}
\mathbb{N}(T_f(k-1))/\pi\mathbb{N}(T_f(k-1))\cong\mathbb{D}_{\rm cris}(T_f(k-1))
\end{equation}
coming from \eqref{eq:Berger-iso-V}.

%Put
%\[
%P_0=\left(\begin{array}{cc}
%0&-\frac{1}{q^{k-1}}\\
%\delta &\frac{a_p}{q^{k-1}}
%\end{array}\right)\in M_{2\times 2}(\mathbb{A}_{\bQ_p}^+),
%\]
%where $\delta=p/(q-\pi^{p-1})\in(\mathbb{A}_{\bQ_p}%^+)^\times$, and $q=\varphi(\pi)/\pi$.

\begin{lem}\label{lem:Wach-basis}
%Let $f\in S_k(\Gamma_0(N))$ be a newform satisfying  \eqref{eq:FL}. Then 
There exists an $\scO$-basis $(v_1,v_2)$ of $\mathbb{D}_{\rm cris}(T_f(k-1))$ as in Lemma~\ref{lem:good-int-basis} and an $\mathbb{A}_{\bQ_p}^+$-basis $(n_1,n_2)$ of $\mathbb{N}(T_f(k-1))$ lifting $(v_1,v_2)$ under \eqref{eq:Berger-iso-T} and  satisfying the following conditions:
\begin{itemize}
\item[(i)] The matrix $P_\varphi$ of $\varphi$ acting on $\mathbb{N}(T_f(k-1))$ with respect to the basis $(n_1,n_2)$ satisfies
\[
P_\varphi\equiv P_0\pmod{\pi},\quad\textrm{where}\quad{P_0=\left(\begin{array}{cc}
0&-1/q^{k-1}\\
\delta^{k-1} & a_p/q^{k-1}
\end{array}\right)}
\]
with $q=\varphi(\pi)/\pi$ and $\delta=p/(q-\pi^{p-1})$; %\in(\mathbb{A}_{\bQ_p}^+)^\times$.
\item[(ii)] $((1+\pi)\varphi(n_1),(1+\pi)\varphi(n_2))$ is a $\Lambda_{\scO}(\widetilde{\Gamma})$-basis of $(\varphi^*\mathbb{N}(T_f(k-1)))^{\psi=0}$.
\end{itemize}
\end{lem}

\begin{proof}
In weight $k\geq 3$, this is shown in \cite[Thm.~3.3]{LLZ-Sha}, in which case one can even take $P_\varphi=P_0$ in part (i). In general, by the proof of \cite[Prop.~V.2.3]{berger-limit} we may choose bases $(v_1,v_2)$ and $(n_1,n_2)$ for which (i) holds with $P_\varphi=P_0$; then by \cite[Thm.~3.5]{LLZ-AJM} we may replace $(n_1,n_2)$ by an $\mathbb{A}_{\bQ_p}^+$-basis $(n_1',n_2')\equiv(n_1,n_2)\pmod{\pi}$ for which (ii) also holds. %whence the result.
\end{proof}

By \cite[Prop.~III.2.1]{berger-limit} we have an inclusion of $\mathbb{B}_{{\rm rig},\bQ_p}^+$-modules
\begin{equation}\label{eq:incl-Brig}
\mathbb{B}_{{\rm rig},\bQ_p}^+\otimes_{\mathbb{A}_{\bQ_p}^+}\mathbb{N}(T_f(k-1))\subset\mathbb{B}_{{\rm rig},\bQ_p}^+\otimes_{\scO}\mathbb{D}_{\rm cris}(T_f(k-1)).
\end{equation}
In particular, for any bases $(v_1,v_2)$ and $(n_1,n_2)$ as in Lemma~\ref{lem:Wach-basis}, which we fix from now on, there is a change of basis matrix
$M\in M_{2\times 2}(\mathbb{B}_{{\rm rig},\bQ_p}^+)$ such that
\begin{equation}\label{eq:change-of-matrix}
(n_1\;,\;n_2)=(v_1\;,\;v_2)\;M.
\end{equation}
In particular, $M\equiv I_2\pmod{\pi}$, where $I_2$ denotes the $2\times 2$ identity matrix. 

\begin{defn}\label{def:log-matrix}
The \emph{logarithm matrix} $M_{\rm log}\in M_{2\times 2}(\mathcal{H}_L(\widetilde{\Gamma}))$ associated to the bases $(v_1,v_2)$, $(n_1,n_2)$ is given by
\[
M_{\rm log}:=\mathfrak{M}^{-1}((1+\pi)A_\varphi\cdot\varphi(M)),
\]
where $M$ is the change of matrix \eqref{eq:change-of-matrix} and $\mathfrak{M}$ is the Mellin transform \eqref{eq:mellin}. 
%(see e.g. \cite[Cor.~B.2.8]{PR-st}), and $\psi$ is the left inverse of $\varphi$.
\end{defn}

\begin{rem}\label{rem:Mlog-Gamma}
One can check that in fact $M_{\rm log}$ is defined over $\mathcal{H}_L(\Gamma)$ (see \cite[\S{3.A}]{LLZ-ANT}).
\end{rem}

\begin{rem}\label{rem:Mlog-basis}
Since $1\mapsto 1+\pi$ under the Mellin transform,  replacing $(n_1,n_2)$ by another basis $(n_1',n_2')$ of $\mathbb{N}(T_f(k-1))$ with $n_i'\equiv n_i\;({\rm mod}\,\pi)$ leaves the resulting logarithm matrix $M_{\rm log}$ unchanged.
\end{rem}

\subsubsection{Signed Coleman maps}\label{subsubsec:signed-Col-cyc}

For any finite dimensional $\bQ_p$-vector space $V$ equipped with a continuous action of $G_{\bQ_p}$ put
\begin{align*}
{\rm H}_{\rm Iw}^1(\bQ_p(\mu_{p^\infty}),T)&:=\varprojlim_n\rH^1(\bQ_p(\mu_{p^n}),T),\\
{\rm H}^1_{\rm Iw}(\bQ_p(\mu_{p^\infty}),V)&:={\rm H}^1_{\rm Iw}(\bQ_p(\mu_{p^\infty}),T)\otimes_{\bZ_p}\bQ_p,
\end{align*}
where $T\subset V$ is any $G_{\bQ_p}$-stable lattice. Put
\[
V=V_f(k-1),\quad\quad T=T_f(k-1). 
\]
%and similarly  $T=T_f(k-1)$.
Note that for any $h\geq 1$ the $G_{\bQ_p}$-representation $V^*(1):={\rm Hom}_L(V,L(1))$ 
%has Hodge--Tate weights $\{2-k,1\}$, and so it 
satisfies ${\rm Fil}^{-h}\mathbb{D}_{\rm cris}(V^*(1))=\mathbb{D}_{\rm cris}(V^*(1))$ (indeed, in our conventions $V^*(1)$ has Hodge--Tate weights $\{2-k,1\}$). Hence by Perrin-Riou's work \cite{PR115}  
for any integer $h\geq 1$ there is a big exponential map
\[
\Omega_{V^*(1),h}:(\mathbb{B}_{{\rm rig},\bQ_p}^+)^{\psi=0}\otimes_{\bQ_p}\mathbb{D}_{\rm cris}(V^*(1))\longrightarrow\mathcal{H}_L(\widetilde{\Gamma})\otimes_{\Lambda_L(\widetilde{\Gamma})}\rH^1_{\rm Iw}(\bQ_p(\mu_{p^\infty}),V^*(1)),
\]
where $(\mathbb{B}_{{\rm rig},\bQ_p}^+)^{\psi=0}$ is being identified with $\mathcal{H}_L(\widetilde{\Gamma})$ via $\mathfrak{M}$, characterised by an interpolation property described in [\emph{op.\,cit.}, \S{3.2.3}]. (The construction of $\Omega_{V^*(1),h}$ depends on the choice of a compatible system $(\zeta_{p^n})_{n\geq 1}$ of $p$-power roots of unity, which we fix once and for all and omit from the notation.) Composing with the $\mathcal{H}_L(\widetilde{\Gamma})$-linear extension of the Perrin-Riou pairing
\[
\langle\,,\,\rangle_{V}:\rH^1_{\rm Iw}(\bQ_p(\mu_{p^\infty}),V)\times\rH^1_{\rm Iw}(\bQ_p(\mu_{p^\infty}),V^*(1))\longrightarrow\Lambda_L(\widetilde{\Gamma})
\]
given by $\langle(x_n)_n,(y_n)_n\rangle_V=\varprojlim_n\sum_{\tau}\langle x_n,y_n^\tau\rangle_{\rm Tate}\cdot\tau$, where $\tau$ runs over the elements in ${\rm Gal}(\bQ_p(\mu_{p^n})/\bQ_p)$ and $\langle\,,\,\rangle_{\rm Tate}:\rH^1(\bQ_p(\mu_{p^n}),T)\times\rH^1(\bQ_p(\mu_{p^n}),T^*(1))\rightarrow\scO$ is the Tate local duality pairing, 
for every $\Xi\in(\mathbb{B}_{{\rm rig},\bQ_p}^+)^{\psi=0}\otimes_{\bQ_p}\mathbb{D}_{\rm cris}(V^*(1))$ %taking $h=1$ 
we obtain a map
\begin{align*}
\mathscr{L}_{V,\Xi}:\rH^1_{\rm Iw}(\bQ_p(\mu_{p^\infty}),V)&\longrightarrow\mathcal{H}_L(\widetilde{\Gamma})\\
\mathbf{z}&\longmapsto\langle\mathbf{z},\Omega_{V^*(1),1}(\Xi)\rangle_V.
\end{align*}
In particular, letting $(v_1^*,v_2^*)$ denote the basis of $\mathbb{D}_{\rm cris}(T^*(1))$ dual to $(v_1,v_2)$ under the natural pairing $[\,,\,]_{\rm dR}:\mathbb{D}_{\rm cris}(T)\times\mathbb{D}_{\rm cris}(T^*(1))\rightarrow\scO$, a key result of Berger \cite[Thm.~II.13]{berger-explicit} (see also \cite[Cor.~3.24]{LLZ-AJM}) is the description of the resulting Perrin-Riou's \emph{big logarithm map}\footnote{Or \emph{big $p$-adic regulator map}.}  
\begin{equation}\label{eq:PR-Qp}
\mathscr{L}_{V}:=(v_1\;,\;v_2)\biggl(\begin{array}{cc}
\mathscr{L}_{V,(1+\pi)\otimes v_1^*}\\
\mathscr{L}_{V,(1+\pi)\otimes v_2^*}
\end{array}\biggr):
\rH^1_{\rm Iw}(\bQ_p(\mu_{p^\infty}),V)\longrightarrow\mathbb{D}_{\rm cris}(V)\otimes_{\bQ_p}\mathcal{H}_L(\widetilde{\Gamma})
\end{equation}
as the composition
\begin{equation}\label{eq:PR-Berger}
\begin{aligned}
\rH^1_{\rm Iw}(\bQ_p(\mu_{p^\infty}),V)\overset{\cong}\longrightarrow\mathbb{N}(V)^{\psi=1}
&\xrightarrow{1-\varphi}(\varphi^*\mathbb{N}(V))^{\psi=0}\\
%&\xrightarrow{{\rm id}\otimes 1}(\varphi^*\mathbb{N}(V))^{\psi=0}\otimes_{\Lambda_L(\widetilde{\Gamma})}\mathbb{B}_{{\rm rig},\bQ_p}^+\\
&\hooklongrightarrow\mathbb{D}_{\rm cris}(V)\otimes_{\bQ_p}(\mathbb{B}_{{\rm rig},\bQ_p}^+)^{\psi=0}\\
&\xrightarrow{\mathfrak{M}^{-1}}\mathbb{D}_{\rm cris}(V)\otimes_{\bQ_p}\mathcal{H}_L(\widetilde{\Gamma}),
\end{aligned}
\end{equation}
where the first arrow is given by Fontaine's isomorphism (see \cite[\S{II.1}]{CC} for a reference), as refined by Berger \cite{berger-explicit} in terms of Wach modules  (see also \cite[Thm.~2.6]{LLZ-Sha}), and the second arrow  is the composite embedding
\[
(\varphi^*\mathbb{N}_{}(V))^{\psi=0}\xrightarrow{{\rm id}\otimes 1}(\varphi^*\mathbb{N}(V))^{\psi=0}\otimes_{\Lambda_L(\widetilde{\Gamma})}\mathbb{B}_{{\rm rig},\bQ_p}^+\cong(\varphi^*\mathbb{N}_{\rm rig}(V))^{\psi=0}\subset\mathbb{D}_{\rm cris}(V)\otimes_{\bQ_p}(\mathbb{B}_{{\rm rig},\bQ_p}^+)^{\psi=0}
\]
explained in \cite[Prop.~2.11]{LLZ-ANT}, where $\varphi^*\mathbb{N}_{\rm rig}(V)$ denotes the $\mathbb{B}_{{\rm rig},\bQ_p}^+$-span of $\varphi(\mathbb{N}_{\rm rig}(V))$ for
\[
\mathbb{N}_{\rm rig}(V):=
\mathbb{N}(V)\otimes_{\mathbb{B}_{\bQ_p}^+}\mathbb{B}_{{\rm rig},\bQ_p}^+. 
\]

The first two maps in \eqref{eq:PR-Berger} admit integral versions, with $T$ in place of $V$. As introduced in \cite{LLZ-AJM}, this leads to the following decomposition of $\mathscr{L}_V$ is terms of \emph{signed Coleman maps} $\widetilde{\mathscr{C}}_i$ in the spirit of Kobayashi \cite{kobayashi-ss}.

\begin{prop}\label{prop:signed-Col-cyc}
Let $(n_1,n_2)$ be a $\mathbb{A}_{\bQ_p}^+$-basis of $\mathbb{N}(T)$ as in Lemma~\ref{lem:Wach-basis}, and for $i\in\{1,2\}$ define 
\[
\widetilde{\mathscr{C}}_i:\rH^1_{\rm Iw}(\bQ_p(\mu_{p^\infty}),T)\longrightarrow\Lambda_{\scO}(\widetilde{\Gamma})
\]
by the composition
\begin{align*}
\widetilde{\mathscr{C}}_i:\rH^1_{\rm Iw}(\bQ_p(\mu_{p^\infty}),T)\overset{\cong}\longrightarrow\mathbb{N}(T)^{\psi=1}
&\xrightarrow{1-\varphi}(\varphi^*\mathbb{N}(T))^{\psi=0}
\xrightarrow{{\rm pr}_i\circ\mathfrak{I}}\Lambda_{\scO}(\widetilde{\Gamma}),
\end{align*}
where the last arrow takes the $i$-th coordinate with respect to the basis $((1+\pi)\varphi(n_1),(1+\pi)\varphi(n_2))$ of $(\varphi^*\mathbb{N}(T))^{\psi=0}$. Then 
%we have the decomposition
\[
\mathscr{L}_{V}=(v_1\;,\;v_2)\cdot M_{{\rm log}}\cdot
\biggl(\begin{array}{cc}
\widetilde{\mathscr{C}}_1\\
\widetilde{\mathscr{C}}_2
\end{array}\biggr)
\]
where $v_i=n_i\;{\rm mod}\,\pi\in\mathbb{D}_{\rm cris}(T)$. %%(so $(v_1,v_2)$ is a $\cO$-basis of $\mathbb{D}_{\rm cris}(T)$ with $v_1\in{\rm Fil}^0\,\mathbb{D}_{\rm cris}(T)$ and $v_2=\varphi(v_1)$).
\end{prop}

\begin{proof}
This is clear from the description of $\mathscr{L}_V$ in  \eqref{eq:PR-Berger} and the definition of $M_{\rm log}$.
\end{proof}

\begin{rem}\label{rem:PR-alpha-beta}
Let $\alpha$ and $\beta$ be the roots of $x^2-a_px+p^{k-1}$, where $a_p$ is the $p$-th Fourier coefficient of $f\in S_k(\Gamma_0(N))$, and suppose  the $p$-regularity hypothesis $\alpha\neq\beta$.  
%\begin{equation}\label{eq:reg}
%\alpha\neq\beta. \nonumber %\tag{reg}
%\end{equation}
(As recalled in the Introduction, this is known to hold for $k=2$ by \cite{coleman-edixhoven}.) Then 
\[
v_\alpha:=\frac{1}{\alpha-\beta}(\alpha v_1-p^{k-1}v_2),\quad\quad v_\beta:=\frac{1}{\alpha-\beta}(-\beta v_1+p^{k-1}v_2)
\]
are $\varphi$-eigenvectors in $\mathbb{D}_{\rm cris}(T)$ with eigenvalues $1/\alpha$ and $1/\beta$, respectively, and for $\xi\in\{\alpha,\beta\}$ letting 
\[
%\mathscr{L}_{V}^\gamma=
\mathscr{L}_{V,(1+\pi)\otimes v_\xi^*}:\rH^1_{\rm Iw}(\bQ_p(\mu_{p^\infty}),V)\longrightarrow\mathcal{H}_L(\widetilde{\Gamma})
\]
denote the coordinate of $\mathscr{L}_V$ with respect to $v_\xi$, so $(v_\alpha^*,v_\beta^*)$ is the basis of $\mathbb{D}_{\rm cris}(V^*(1))$ dual to $(v_\alpha,v_\beta)$, the decomposition in Proposition~\ref{prop:signed-Col-cyc} can be rewritten as
\begin{equation}\label{eq:factor-L}
\biggl(\begin{array}{cc}
\mathscr{L}_{V,(1+\pi)\otimes v_\alpha^*}^{}\\
\mathscr{L}_{V,(1+\pi)\otimes v_\beta^*}^{}
\end{array}\biggr)
=%(v_\alpha\;,\;v_\beta)\cdot 
Q_{\alpha,\beta}^{-1} M_{{\rm log}}\cdot
\biggl(\begin{array}{c}
\widetilde{\mathscr{C}}_1\\
\widetilde{\mathscr{C}}_2
\end{array}\biggr),
\end{equation}
where $Q_{\alpha,\beta}=\frac{1}{\alpha-\beta}\left(\begin{smallmatrix}\alpha & -\beta\\ -p^{k-1}&p^{k-1}\end{smallmatrix}\right)$.
\end{rem}

\subsubsection{Images of the signed Coleman maps}

For every character $\eta:\Delta\rightarrow\bZ_p^\times$, applying the projector $e_\eta=\frac{1}{p-1}\sum_{\delta\in\Delta}\eta^{-1}(\sigma)\sigma\in\bZ_p[\Delta]$ to the $\eta$-isotypic component %for the action of $\widetilde{\Gamma}=\Delta\times\Gamma$, 
we get $\Lambda_\scO(\Gamma)$-module maps 
\begin{equation}\label{eq:Col-eta}
e_\eta\widetilde{\mathscr{C}}_i:\rH^1_{\rm Iw}(\bQ_p(\mu_{p^\infty}),T)^\eta\longrightarrow\Lambda_{\scO}(\widetilde{\Gamma})^\eta\cong\Lambda_{\scO}(\Gamma).
\end{equation}

%For $i\in\{1,2\}$, we let
%\begin{equation}\label{eq:Col-eta=1}
%\mathscr{C}_i%:=e_1\widetilde{\mathscr{C}}_i
%:
%\rH^1_{\rm Iw}(\bQ_{p,\infty},T)\cong e_1\rH^1_{\rm Iw}(\bQ_p(\mu_{p^\infty}),T)\longrightarrow\Lambda_{\scO}(\Gamma)
%\end{equation}
%denote the map $e_\eta\widetilde{\mathscr{C}}_i$ for the trivial character $\eta=1$ of $\Delta$, where $\bQ_{p,\infty}=\bQ_p(\mu_{p^\infty})^\Delta$ is the cyclotomic $\bZ_p$-extension of $\bQ_p$. 

For $m\geq 1$, put
\[
\delta_m:=\prod_{j=0}^{m-1}(u^{-j}\gamma-1)\in\Lambda_{\scO}(\Gamma),
\]
where $\gamma\in\Gamma$ and $u\in 1+p\bZ_p$ are topological generators.

\begin{prop}\label{prop:image-Col}
%Let $\eta$ be a character of $\Delta$ and $i\in\{1,2\}$.  Then then 
Let $i\in\{1,2\}$. For every character $\eta:\Delta\rightarrow\bZ_p^\times$, there
exists a factor $\xi_{\eta,i}$ of $\delta_{k-1}$ such that %we have the inclusion 
\[
{\rm Image}(e_\eta\widetilde{\mathscr{C}}_i)\subset \xi_{\eta,i}\Lambda_{\scO}(\Gamma)
\]
%and this has 
with finite index.
\end{prop}

\begin{proof}
See Corollary~5.3 and Theorem~5.10 in \cite{LLZ-ANT}.
\end{proof}

\begin{rem}\label{rem:xi}
As observed in \cite[Rem.~2.15]{BL-non-ord}, we have $\xi_{\eta,1}=1$ for all $\eta$, while $\xi_{\eta,2}\neq 1$ unless $k=2$ and $\eta=\mathds{1}$ is the trivial character of $\Delta$.
\end{rem}

%In \cite{LLZ-AJM}, the signed Coleman maps ${\rm Col}_i$ (and the corresponding signed Selmer groups ${\rm Sel}_{p^\infty}^i$) are defined using the basis from \cite{BLZ} (which requires $v_p(a_p)\geq\lfloor\frac{k-2}{p-1}\rfloor$).

%****************

%work of \cite{LLZ-AJM} (see also \cite{LLZ-ANT}): signed Coleman maps, signed $p$-adic $L$-functions (using Kato's Euler system), signed IMC and equivalence with Kato's IMC. 

%Write in the style of \cite{BSTW}

\subsection{Kato's Euler system}\label{subsec:Kato-ES}

Let $V_{F}(f)$ and $S(f)$ be the subspaces of de Rham and Betti cohomology attached to the eigenform $f\in S_k(\Gamma_0(N))$ as in \cite[\S{6.3}]{Kato295}; %(where the latter is denoted $S(f)$); 
these are ${F}$-vector spaces of dimension $1$ and $2$, respectively. For any commutative ring extension $A$ of ${F}$, we put $V_{A}(f)=V_{F}(f)\otimes_{{F}}A$ and $S_A(f)=S(f)\otimes_{{F}}A$. Then one has canonical isomorphisms 
%then one has canonical isomorphisms 
\begin{equation}\label{eq:comp-iso}
V_{L}(f)\cong V_f,\quad\quad S_{L}(f)\cong {\rm Fil}^1\mathbb{D}_{\rm cris}(V_f),
\end{equation}
%(see \cite[Lem.~A.3]{nakamura-invmath} for the first isomorphism), 
%by the comparison isomorphisms, 
where we recall that $L=F_\mathfrak{P}$ is the completion of $F$ at the prime $\mathfrak{P}$ above $p$ induced by $\iota_p$.

Let $V_{\bC}(f)^\pm$ denote the $\pm$-eigenspace of $V_{\bC}(f)$ under the action of complex conjugation. For every  $\gamma\in V_{F}(f)$ with $\gamma^\pm\neq 0\in V_{\bC}(f)^\pm$ and an ${F}$-basis $\omega\in S_{}(f)$,  define $\Omega_{\gamma,\omega}^\pm\in\bC^\times$ by
\[
{\rm per}_f(\omega)=\Omega_{\gamma,\omega}^+\cdot\gamma^++\Omega_{\gamma,\omega}^-\cdot\gamma^-,
\]
where ${\rm per}_f:S_{}(f)\rightarrow V_{\mathbf{C}}(f)$ is the period map of \cite[\S{6.3}]{Kato295}.

\subsubsection{Kato's explicit reciprocity law}

Let $q\in\bZ$. Following \cite[\S{12.2}]{Kato295}, for $V$ a finite dimensional $\bQ_p$-vector space equipped with a continuous $G_{\bQ}$-action unramified outside a finite set of primes  and $T\subset V$ any $G_{\bQ}$-stable lattice, we put
\begin{align*}
{\rm H}_{\rm Iw}^q(\bQ(\mu_{p^\infty}),T)&:=\varprojlim_n\rH^q(\bZ[\mu_{p^n},1/p],j_*T),\\
{\rm H}_{\rm Iw}^q(\bQ(\mu_{p^\infty}),V)&:={\rm H}_{\rm Iw}^q(\bQ(\mu_{p^\infty}),T)\otimes_{\bZ_p}\bQ_p,
\end{align*}
where $j:{\rm Spec}(\bQ(\mu_{p^n}))\rightarrow{\rm Spec}(\bZ[\mu_{p^n},1/p])$ is the natural map. These are zero for $q\neq 1,2$, and the latter is independent of the choice of $T$.

%By a twisted variant of \cite[Thm.~12.5]{Kato295} (cf. \cite[Cor.~A.5]{nakamura-invmath}), 
By \cite[Thm.~12.5]{Kato295}, there is an $L$-linear map 
\[
V_f\longrightarrow{\rm H}^1_{\rm Iw}(\bQ(\mu_{p^\infty}),V_f):\;\gamma\longmapsto\mathbf{z}_{\gamma}^{(p)} 
\]
%sending $\gamma\mapsto\mathbf{z}_{V,\gamma}$ 
%(denoted by $\gamma\mapsto\mathbf{z}_\gamma^{(p)}(k-1)$ in \emph{loc.\,cit.}) 
with the property that for $0\leq j\leq k-2$, if $\gamma\in V_{F}(f)\subset V_f$ under the identification \eqref{eq:comp-iso}, then the image 
%${\rm exp}^*({\rm res}_p({\rm pr}_{\bQ}(\mathbf{z}^{(p)}_{\gamma,k-1-j})))$ 
of $\mathbf{z}_{\gamma}^{(p)}$ by the composition%\footnote{using $V_f^*\cong V_f(k-1)=V$ for the last equality}
\begin{align*}
{\rm H}^1_{\rm Iw}(\bQ(\mu_{p^\infty}),V_f)\xrightarrow{\otimes((\zeta_{p^n})_{n\geq 0})^{\otimes(k-1-j)}}
{\rm H}^1_{\rm Iw}(\bQ(\mu_{p^\infty}),V(-j))&\xrightarrow{{\rm pr}_{\bQ}}\rH^1(\bQ,V(-j))\\%=\rH^1(\bQ,V_f(k-1-j))\\
&\xrightarrow{{\rm res}_p}\rH^1(\bQ_p,V(-j))\\
%\rH^1(\bQ_p,V_f(k-1-j))\\
&\xrightarrow{{\rm exp}^*
}{\rm Fil}^0\mathbb{D}_{\rm cris}(V_f^*(-j))={\rm Fil}^1\mathbb{D}_{\rm cris}(V_f)
\end{align*}
is contained in $S_{}(f)\subset{\rm Fil}^1\mathbb{D}_{\rm cirs}(V_f)$ under the identification \eqref{eq:comp-iso}, and satisfies
%\begin{equation}\label{eq:kato-ERL}
%(\textrm{image of $\mathbf{z}_\gamma^{(p)}$})\xrightarrow{{\rm per}_f^\pm}
%{\rm per}_f^\pm({\rm exp}^*({\rm res}_p({\rm pr}_{\bQ}(\mathbf{z}^{(p)}_{\gamma,k-1-j}))))=(2\pi i)^{k-j-2}\cdot%\frac{L_{\{p\}}(f,j+1)}{\Omega_{\gamma,\omega}^\pm}\cdot\omega,
%takes the image of $\mathbf{z}_\gamma^{(p)}$ to 
\begin{equation}\label{eq:kato-ERL}
{\rm per}_f^\pm(\textrm{image of $\mathbf{z}_\gamma^{(p)}$})=%(2\pi i)^{k-2-j}\cdot L_{\{p\}}(f,j+1)\cdot\gamma^\pm,
\frac{L_{\{p\}}(f,j+1)}{(-2\pi i)^{j+1}}\cdot\gamma^\pm,
\end{equation}
where $L_{\{p\}}(f,s)$ is the $L$-function on $f$ with the Euler factor at $p$ removed, the sign is $\pm=(-1)^{j}$, and ${\rm per}_f^\pm:S(f)\rightarrow V_{\bC}(f)^\pm$ is the composition of ${\rm per}_f$ with the projection to the $\pm$-eigenspace. 

\subsubsection{Integral normalisations}\label{subsec:integral-normalisation}

Let $\cO=\cO_{F,\mathfrak{P}}$ denote the localization of the ring of integers of $F$ at the prime $\mathfrak{P}$ above $p$ determined by $\iota_p$. 

Taking cohomology with coefficients in $\cO$ %rather than $F$ 
in the definition of $V_F(f)$ defines a lattice $T_\cO(f)\subset V_F(f)$ contained in $T_f$. In the following we put 
\[
T_{\cO}(f)^{\pm}=T_\cO(f)\cap V_{\bC}(f)^{\pm}
\]
and fix $\cO$-module generators $\gamma_f^\pm\in T_{\cO}(f)^{\pm}$. 

On the other hand, adapting the notation in \cite[\S{2.8}]{KLZ2} let $M_{\cO}(f)$ denote the maximal subspace of
\[
\rH^1_{c}(Y_1(N)(\bC),{\rm Sym}^{k-2}\mathscr{H}_{\bZ}^\vee)\otimes_{\bZ}\cO
\]
on which the Hecke operators $T_n$ act as multiplication by $a_n$, and put
\[
M_{\cO}(f)^\pm:=M_{\cO}(f)\cap\bigl(\rH^1_c(Y_1(N)(\bC),{\rm Sym}^{k-2}\mathscr{H}_{\bZ}^\vee)\otimes_{\bZ}\bC\bigr)^\pm,
\]
and fix $\cO$-module generators $\delta_f^\pm\in M_{\cO}(f)^\pm$. 

From the definitions, we have natural inclusions $M_{\cO}(f)^\pm\hookrightarrow T_{\cO}(f)^\pm$ which become isomorphisms after extending scalars to $F$.

\begin{lem}\label{lem:Gor}
%Suppose the residual representation $\bar{\rho}_f:G_\bQ\rightarrow{\rm GL}_2(\kappa_L)$ is irreducible. 
We can write 
\[
\delta_f^\pm=c^\pm\cdot\gamma_{f}^\pm
\]
with $c^\pm=c_f\in\cO$ a generator of the congruence ideal for $f$.
\end{lem}

\begin{proof}
Let $\kappa_L=\mathscr{O}/\mathfrak{m}_L$ be the residue field of $\mathscr{O}$. By a result of Fontaine  (see e.g.  \cite[Thm.~2.6]{Edi}), the residual representation $
\bar{\rho}_f:G_{\bQ}\rightarrow{\rm GL}_2(\kappa_L)$
attached to $f$ is irreducible (even after restriction to $G_{\bQ_p}$), and so the localisation of the Hecke algebra at the corresponding maximal ideal is known to be Gorenstein %(see \cite[Thm.~1.5]{kilford-wiese} and the references therein)
(see \cite[Thm.~2.1]{faltings-jordan}). Thus the argument in \cite[Lem.~2.5]{BSTW} applies.
\end{proof}

%Note that if follows from \cite{hida-AJM} that the $\cO$-module
%\[
%H^1_k(N_f)_\cO^\pm:=H^1_k(N_f)_{\cO}\cap H^1_k(N_f)_{\bC}^\pm[\mathfrak{p}_f]
%\]
%s free of rank one. 

Following \cite[\S{2.1}]{maksoud} we consider the following $p$-normalisation of periods.

\begin{defn}\label{def:p-periods}
%Let $\gamma_\cO^{\pm}$ be $\cO$-module generators of 
%\[
%\bigl(\rH^1_c(Y_1(N)(\bC),{\rm Sym}^{k-%2}\mathscr{H}_{\bZ}^\vee)\otimes_{\bZ}\cO\bigr)^\pm[\mathfrak{p}_f],
%\]
%respectively, which we identify with their images in $V_F(f)$ 
%(see \cite{KLZ1} for the unexplained notation). 
For any nonzero $\omega\in S(f)$, the \emph{$p$-normalised periods} of $\omega$ are the complex numbers $\Omega_{\omega}^\pm\in\bC^\times$ defined by
\[
{\rm per}_f(\omega)=\Omega_\omega^+\cdot{\delta}_\cO^++\Omega_\omega^-\cdot{\delta}_\cO^-.
\]
In other words, $\Omega_\omega^\pm=\Omega_{\delta_\cO,\omega}^\pm$ where $\delta_\cO:=\delta_\cO^++\delta_\cO^-$.
%
%he \emph{$p$-normalized periods} of $f$ are the complex numbers $\Omega_f^\pm\in\bC^\times$ (well-defined up to a unit in $\cO^\times$) defined by the relation
%\[
%{\rm per}_f(\omega_f)^\pm=\Omega_f^\pm\cdot\nu_f^\pm,
%\]
%where $\nu_f^\pm$ is an $\cO$-basis of $
%H^1_k(N_f)_\cO^\pm$.
\end{defn}

%\begin{rem}
%Viewed in $V_{F}(f)=M_{F}(f)^*$, the class $\Upsilon$ is contained in $M_{\cO}(f)^*$. 
%\end{rem}

%\begin{defn}\label{def:Kato}
%Fix $\cO$-bases $\gamma^\pm$ of $H_k^1(N_f)_{\cO}^\pm$, and put $\gamma:=\gamma^++\gamma^-\in T$. 
%In the following, we let 
On the other hand, we consider the following $p$-integral normalisation of Kato's class:
\[
\mathbf{z}_f^{\rm int}:=\mathbf{z}_{\gamma_\cO}^{(p)}\in{\rm H}^1_{\rm Iw}(\bQ(\mu_{p^\infty}),V_f),
\]
where $\gamma_\cO:=\gamma_\cO^++\gamma_\cO^-\in T_{\cO}(f)\subset T_f$. %is as in Definition~\ref{def:p-periods}.
%be the class $\mathbf{z}_\gamma^{(p)}(k-1)$ of \cite[Thm.~12.5]{Kato295} associated to $\gamma=\Upsilon$.   
%\end{defn}
%
%In particular, letting $z_f\in\rH^1(\bQ,V)$ be the image of $\mathbf{z}_f$ under the natural projection $\mathbb{H}^1(\bQ(\mu_{p^\infty}),V)\rightarrow\rH^1(\bQ,V)$, by Kato’s explicit reciprocity law \cite{Kato295} the Bloch--Kato dual exponential map ${\rm exp}^*_{\omega}:\rH^1(\bQ_p,V)\rightarrow F$ sends the restriction ${\rm res}_p(z_f)$ to
%\begin{equation}\label{eq:kato-ERL}
%{\rm exp}_{\omega}^*({\rm res}_p(z_1))=\frac{L_{\{p\}}(f,1)}{\Omega},
%\end{equation}
%where $L_{\{p\}}(f,s)$ denotes the $L$-function on $f$ with the Euler factor at $p$ removed. ***** NEED OTHER $j\in\{1,\dots,k-1\}$ HERE.
%

\begin{prop}\label{prop:Kato-int}
Then class $\mathbf{z}_{f}^{\rm int}$ is contained in $\rH^1_{\rm Iw}(\bQ(\mu_{p^\infty}),T_f)\subset\rH^1_{\rm Iw}(\bQ(\mu_{p^\infty}),V_f)$.
\end{prop}

\begin{proof}
As noted in the proof of Lemma~\ref{lem:Gor}, under our running assumptions the residual representation $\bar{\rho}_f$ is irreducible, so the result thus follows from (the proof of) \cite[Thm.~12.5(4)]{Kato295}.
\end{proof}

%From now on we assume that the residual representation
%\begin{equation}\label{eq:irred}
%\textrm{$\bar{\rho}_f:G_{\bQ}\longrightarrow{\rm GL}_2(\kappa_L)$ is irreducible}\tag{irred}
%\end{equation}
%(as will be the case in our applications in this paper). Then by \cite[Thm.~12.5]{Kato295} we know that $\mathbf{z}_f^{\rm int}$ is integral, i.e. $\mathbf{z}_f^{\rm int}\in{\rm H}^1_{\rm Iw}(\bQ(\mu_{p^\infty}),T_f)$.

\begin{defn}\label{def:cyc-Lp}
For $\xi\in\{\alpha,\beta\}$ and $i\in\{1,2\}$ put
\[
L_\xi(f):=\mathscr{L}_{V,(1+\pi)\otimes v_\xi^*}({\rm res}_p(\mathbf{z}_{f,k-1}^{\rm int})), %\in\mathcal{H}_L(\Gamma),
\quad\quad L_{i}(f):=\widetilde{\mathscr{C}}_i({\rm res}_p(\mathbf{z}_{f,k-1}^{\rm int})), %\in\Lambda_L(\Gamma).
\]
where $\mathbf{z}_{f,k-1}^{\rm int}$ denotes the image of $\mathbf{z}_f^{\rm int}$ under the twisting map 
\[
{\rm H}^1_{\rm Iw}(\bQ(\mu_{p^\infty}),V_f)\xrightarrow{\otimes((\zeta_{p^n})_{n\geq 0})^{\otimes(k-1)}}{\rm H}^1_{\rm Iw}(\bQ(\mu_{p^\infty}),V).
\]
%given by  $\otimes((\zeta_{p^n})_{n\geq 0})^{\otimes(k-1)}$.
Note that these are elements in $\mathcal{H}_L(\widetilde{\Gamma})$ and $\Lambda_{\scO}(\widetilde{\Gamma})$, respectively. 
\end{defn}

\subsubsection{Interpolation formulas}\label{subsec:interp-MTT}

Fix a topological generator $\gamma\in\Gamma$, and for $\zeta$ a primitive $p^t$-th root of unity, let $\psi_\zeta:G_\bQ\twoheadrightarrow\Gamma\rightarrow\overline{\bQ}_p^\times$ be the character determined by $\psi_\zeta(\gamma)=\zeta$. By abuse of notation, we also let $\psi_\zeta$ the resulting Dirichlet character $(\bZ/p^{t+1}\bZ)^\times\rightarrow\overline{\bQ}_p^\times$. Let 
\[
\varepsilon:\widetilde{\Gamma}\longrightarrow\bZ_p^\times,\quad\quad
\langle\varepsilon\rangle=\varepsilon\omega^{-1}:\Gamma\longrightarrow 1+p\bZ_p,
\]
be the $p$-adic cyclotomic character, where $\omega:\Delta\rightarrow\bZ_p^\times$ is the mod $p$ cyclotomic character. 

\begin{defn}
For every character $\eta:\Delta\rightarrow\bZ_p^\times$, let 
\[
L_\xi(f)^\eta\in\mathcal{H}_L(\Gamma),\quad\quad \textrm{(resp. $L_i(f)^\eta\in\Lambda_{\scO}(\Gamma)$)}
\]
denote the images of $L_\xi(f)$ and $L_i(f)$ under the projections$\mathcal{H}_L(\widetilde{\Gamma})\rightarrow\mathcal{H}_L(\Gamma)$ (resp.  $\Lambda_\scO(\widetilde{\Gamma})\rightarrow\Lambda_{\scO}(\Gamma)$) defined by $e_\eta$. 
\end{defn}

As in \cite[\S{14.22}]{Kato295}, the image of the differential $\omega_f\in S(f)$ attached to $f$ (see e.g. \cite[\S{6.1}]{KLZ1}) under the Faltings--Tsuji comparison isomorphism \eqref{eq:comp-iso} gives an $\scO$-basis of ${\rm Fil}^1\mathbb{D}_{\rm cris}(T_f)$. Thus in the following we shall assume that the basis $(v_1,v_2)$ from Lemma~\ref{lem:Wach-basis} is such that $v_1=\omega_f$. 

%The differential $\omega_f$ %%\in\rH^0(X_1(N)_{/\cO},\omega^k)$ 
%attached to $f$ generates an $\cO$-lattice $S_{\cO}(f):=\cO\omega_f\subset S_{}(f)$ with 
%\[
%_{\cO}(f)\otimes_{\cO}\mathscr{O}\cong{\rm Fil}^1\mathbb{D}_{\rm cris}(T_f)
%\]
%nder the second isomorphism \eqref{eq:comp-iso}; thus we can write 
%\begin{equation}\label{eq:comp-omega}
%v_1=r_f\cdot\omega_f 
%\end{equation}
%with $r_f\in\mathscr{O}$. 

%Let $v_p$ be the $p$-adic valuation of $\overline{\bQ}_p$ with $v_p(p)=1$.

\begin{thm}\label{thm:MTT}
Let $\xi\in\{\alpha,\beta\}$ and $\eta:\Delta\rightarrow\bZ_p^\times$ be a character. Then $L_\xi(f)^\eta$ satisfies the following interpolation property: For all characters $\chi=\langle\varepsilon\rangle^j\psi_\zeta$ of $\Gamma$ with $0\leq j\leq k-2$ and $\zeta$ a primitive $p^t$-th root of unity for some $t\geq 0$, we have
\[
\phi_\chi(L_\xi(f)^\eta)=%\frac{\lambda_f}{\alpha-\beta}
%\frac{1}{r_f}\cdot 
e_{p,\xi,\eta}(f,\chi)\cdot\frac{j!\cdot L(f\otimes\psi_\zeta^{-1}\eta\omega^{j},j+1)}{(-2\pi i)^{j+1}\cdot\mathfrak{g}(\psi_\zeta^{-1}\eta\omega^j)\cdot\Omega_{\omega_f}^{\pm}},
\]
where $\mathfrak{g}(\psi_\zeta^{-1}\eta\omega^j)$ is the Gauss sum, the sign is $\pm=(-1)^j$,  
$\Omega_{\omega_f}^\pm$ is the $p$-normalised period in Definition~\ref{def:p-periods}, and
%\item 
\[
e_{p,\xi,\eta}(f,\chi)=\frac{p^{t'(j+1)}}{\xi^{t'}}\cdot\biggl(1-\frac{\psi_\zeta^{-1}\eta\omega^j(p)p^{k-2-j}}{\xi}\biggr)\biggl(1-\frac{\psi_\zeta\eta^{-1}\omega^{-j}(p)p^j}{\xi}\biggr)
\]
%\begin{cases}\frac{1}{\gamma^{t+1}}\cdot\frac{p^{(j+1)(t+1)}}{\mathfrak{g}(\psi_\zeta^{-1})}&\textrm{if $t\geq 1$, with $\mathfrak{g}(\psi_\zeta^{-1})$ the Gauss sum,}\\[0.6em]
%\Bigl(1-\frac{p^j}{\gamma}\Bigr)\Bigl(1-\frac{p^{k-2-j}}{\gamma}\Bigr)&\textrm{else.}\end{cases}
%\end{itemize}
with $t'=0$ or $t+1$ according to whether or not $\eta=\omega^{-j}$ and $t=0$.
%\[
%t'=
%\begin{cases}
%0&\textrm{if $\eta=\omega^{-j}$ and $t=0$},\\  t+1&\textrm{otherwise}.
%\end{cases}
%\]
%\noindent Moreover, if $v_p(\xi)<k-1$, then $L_{\xi}(f)$ is uniquely characterised by this property.
\end{thm}

\begin{proof}
This is a reformulation of \cite[Thm.~16.6]{Kato295}, and follows from Kato's explicit reciprocity law \eqref{eq:kato-ERL} (and its variant for twists by finite order characters of $\widetilde{\Gamma}$) and the interpolation property of $\mathscr{L}_{V}$. Indeed, for $\xi\in\{\alpha,\beta\}$ Kato considers the differentials $\eta_\xi\in\mathbb{D}_{\rm cris}(V_f)$ attached to $\omega_f$ characterised by the properties $\varphi(\eta_\xi)=\xi\cdot\eta_\xi$ and $[\eta_\xi,\omega_f]_{\rm dR}=1$, and shows that the image 
\[
\eta_{\xi,1}:=\eta_\xi\otimes t^{-1}e_1\in%\mathbb{D}_{\rm cris}(V_f(1))%=\mathbb{D}_{\rm cris}(V_f(k-1)^*(1))
\mathbb{D}_{\rm cris}(V_f)\otimes\mathbb{D}_{\rm dR}(\bQ_p(1))\cong\mathbb{D}_{\rm cris}(V^*(1))
\]
(where $t\in\mathbb{B}_{\rm dR}$ and the basis $e_1\in\bQ_p(1)$ are the ones defined by our chosen  $(\zeta_{p^n})_{n\geq 0}$) is such that
\[
\mathscr{L}_{V,(1+\pi)\otimes\eta_{\xi,1}}({\rm res}_p(\mathbf{z}_{f,k-1}^{\rm int}))\in\mathcal{H}_L(\widetilde{\Gamma})
\]
has the stated interpolation property after applying $e_\eta$. Since a direct calculation shows that
\begin{equation}\label{eq:dual-eta}
\eta_{\xi,1}=%r_f\cdot v_{\xi'}^*,
v_{\xi'}^*
\end{equation}
where $\xi'$ is the element in $\{\alpha,\beta\}$ different from $\xi$, %and $r_f$ is as in \eqref{eq:comp-omega}, 
the result follows. 
%The second claim follows from the fact that $L_\xi(f)\in\mathcal{H}_L(\widetilde{\Gamma})$ has growth $h=v_p(\xi)$ by \cite[\S{3.2.4}]{PR115}. %(see \cite[\S{14}]{mtt}).
\end{proof}

Note that from \eqref{eq:factor-L} we have the decomposition
\begin{equation}\label{eq:factor-Lp}
\biggl(\begin{array}{cc}
L_{\alpha}(f)\\
L_{\beta}(f)
\end{array}\biggr)
=Q_{\alpha,\beta}^{-1} M_{{\rm log}}\cdot
\biggl(\begin{array}{cc}
L_1(f)\\
L_2(f)\end{array}\biggr)\in\mathcal{H}_L(\widetilde{\Gamma})^{\oplus 2}.
\end{equation}
%in $\mathcal{H}_L(\widetilde{\Gamma})^{\oplus{2}}$. 

\begin{prop}\label{prop:L-nonzero}
For every character $\eta:\Delta\rightarrow\bZ_p^\times$, the projection $L_{i}(f)^\eta\in\Lambda_{\scO}(\Gamma)$ is nonzero for some $i\in\{1,2\}$.
\end{prop}

\begin{proof}
By the interpolation property of Theorem~\ref{thm:MTT} and the nonvanishing result of \cite{rohrlich-division} (or just \cite[Prop.~2]{shimuraCPAM} for $k>2$), the left-hand side of \eqref{eq:factor-Lp} 
has nonzero image under $e_\eta$. Since $Q_{\alpha,\beta}^{-1}M_{\rm log}$ has nonzero determinant (see \cite[Cor.~3.2]{LLZ-ANT}), the result follows.
\end{proof}

\begin{rem}\label{rem:cf-ap=0}
Contrary to the case $a_p=0$ studied in \cite[Cor.~5.11]{pollack}, due to the inexplicit nature of $M_{\rm log}$ in general, the interpolation property of Theorem~\ref{thm:MTT} does not translate into a similar property of $L_i(f)^\eta$ in a way allowing us to show that $L_i(f)^\eta\neq 0$ for \emph{both} $i\in\{1,2\}$ from the aforementioned nonvanishing results (cf. \cite[Cor.~3.29]{LLZ-AJM}).
\end{rem}

\subsection{Main conjectures}

%In this section 
We recall the formulation of signed main conjectures from \cite{LLZ-AJM,LLZ-ANT} and their relation with Kato's main conjecture.

\subsubsection{Formulations}

For $i\in\{1,2\}$, define the \emph{$i$-th submodule} of $\rH^1_{\rm Iw}(\bQ_p(\mu_{p^\infty}),T)$ by
\[
\rH^1_{{\rm Iw},i}(\bQ_p(\mu_{p^\infty}),T):={\rm ker}(\widetilde{\mathscr{C}}_i),
\]
where $\widetilde{\mathscr{C}}_i$ is the signed Coleman map of Proposition~\ref{prop:signed-Col-cyc}. Let $V^*(1)={\rm Hom}_L(V,L(1))$ be the Kummer dual of $V$, put $A_f(1)=V_f(1)/T_f(1)$, and using the identification\footnote{Recall that in general $V_f(k-1)^*\cong V_{\bar{f}}$ where $\bar{f}$ is the form obtained by complex conjugating the Fourier coefficients of $f$; however $\bar{f}=f$ for $f$ with trivial nebentypus.}
\begin{equation}\label{eq:dual-kummer}
V^*(1)=V_f(k-1)^*(1)\cong V_f(1).
\end{equation}
let $\rH^1_i(\bQ_p(\mu_{p^\infty}),A_f(1))$ be the orthogonal complement of $\rH^1_{{\rm Iw},i}(\bQ_p(\mu_{p^\infty}),T)$ under  local Tate duality
\[
\rH^1_{\rm Iw}(\bQ_p(\mu_{p^\infty}),T)\times\rH^1(\bQ_p(\mu_{p^\infty}),A_f(1))\longrightarrow\bQ_p/\bZ_p.
\]

\begin{defn}\label{def:signed-Sel}
For $i\in\{1,2\}$ put
\[
{\rm Sel}_i(f/\bQ(\mu_{p^\infty})):={\rm ker}\biggl\{\rH^1(\bQ(\mu_{p^\infty}),A_f(1))\longrightarrow\frac{\rH^1(\bQ_p(\mu_{p^\infty}),A_f(1))}{\rH^1_i(\bQ_p(\mu_{p^\infty}),A_f(1))}\times\prod_{w\nmid p}\rH^1(\bQ(\mu_{p^\infty})_w,A_f(1))\biggr\};
\]
and define the \emph{strict Selmer group} by
\[
{\rm Sel}_{\rm str}(f/\bQ(\mu_{p^\infty}))
:={\rm ker}\bigl\{{\rm Sel}_i(f/\bQ(\mu_{p^\infty}))\longrightarrow\rH^1(\bQ_p(\mu_{p^\infty}),A_f(1))\bigr\}
\]
(which is of course independent of $i$). 
\end{defn}

Let $X_i(f/\bQ(\mu_{p^\infty}))={\rm Hom}_{\bZ_p}({\rm Sel}_{i}(f/\bQ(\mu_{p^\infty})),\bQ_p/\bZ_p)$ be the Pontryagin dual of ${\rm Sel}_i(f/\bQ(\mu_{p^\infty}))$ and similarly put $X_{\rm str}(f/\bQ(\mu_{p^\infty}))={\rm Hom}_{\bZ_p}({\rm Sel}_{\rm str}(f/\bQ(\mu_{p^\infty})),\bQ_p/\bZ_p)$. 

%Following \cite[\S{12.2}]{Kato295}, for $q\in\bZ$ put
%\begin{align*}
%\mathbb{H}^q(\bQ(\mu_{p^\infty}),T)&:=\varprojlim_n\rH^q(\bZ[\mu_{p^n},1/p],j_*T),\\
%\mathbb{H}^q(\bQ(\mu_{p^\infty}),V)&:=\mathbb{H}^q(\bQ(\mu_{p^\infty}),T)\otimes_{\bZ_p}\bQ_p,
%\end{align*}
%where $j:{\rm Spec}(\bQ(\mu_{p^n}))\rightarrow{\rm Spec}(\bZ[\mu_{p^n},1/p])$ is the natural map. These are zero for $q\neq 1,2$, and the latter is independent of the choice of $T$. 

\subsubsection{Kato's main conjecture}

As we explain below, the following is a reformulation of Kato's main conjecture (cf. \cite[Conj.~12.10]{Kato295}).

\begin{conj}[Kato's main conjecture]\label{conj:kato}
%For $p>2$, the class $\mathbf{z}_{f}^{\rm int}=\mathbf{z}_{\gamma_{\cO}}^{(p)}\in\rH^1_{\rm Iw}(\bQ(\mu_{p^\infty}),V_f)$ is contained in $\rH^1_{\rm Iw}(\bQ(\mu_{p^\infty}),T_f)$, and 
For every character $\eta:\Delta\rightarrow\bZ_p^\times$ we have
\[
{\rm char}_{\Lambda_\scO(\Gamma)}\biggl(\frac{{\rm H}_{\rm Iw}^1(\bQ(\mu_{p^\infty}),T)^\eta}{(\mathbf{z}_{f,k-1}^{\rm int})^\eta}\biggr)={\rm char}_{\Lambda_\scO(\Gamma)}\bigl(X_{\rm str}(f/\bQ(\mu_{p^\infty}))^\eta\bigr)
\]
as characteristic ideals of torsion $\Lambda_{\scO}(\Gamma)$-modules.
\end{conj}

\begin{rem}\label{rem:equiv-KatoIMC}
Let $S$ be a set of primes containing $\infty$ and the primes dividing $Np$, and let $\bQ^S$ denote the maximal extension of $\bQ$ unramified outside $S$. As shown in \cite[Lemma~4.3]{kurihara-invmath} (see also \cite[\S{6.2}]{Lei-PhD}), Poitou--Tate duality gives
\begin{equation}\label{eq:PT-H2}
X_{\rm str}(f/\bQ(\mu_{p^\infty}))\cong{\rm ker}\biggl\{{\rm H}^2_{{\rm Iw},S}(\bQ(\mu_{p^\infty}),T)\longrightarrow\prod_{\ell\in S}\prod_{w\mid\ell}\rH^2_{\rm Iw}(\bQ(\mu_{p^\infty})_w,T)\biggr\},
\end{equation}
where ${\rm H}^2_{{\rm Iw},S}(\bQ(\mu_{p^\infty}),T)=\varprojlim_n\rH^2({\rm Gal}(\bQ^S/\bQ(\mu_{p^n})),T)$. On the other hand, from the localisation sequence in \'{e}tale cohomology we have
\begin{equation}\label{eq:loc-H2}
\rH^2_{\rm Iw}(\bQ(\mu_{p^\infty}),T)\cong{\rm ker}\biggl\{{\rm H}^2_{{\rm Iw},S}(\bQ(\mu_{p^\infty}),T)\longrightarrow\prod_{\ell\in S\smallsetminus\{p\}}\prod_{w\mid\ell}\rH^2_{\rm Iw}(\bQ(\mu_{p^\infty})_w,T)\biggr\}
\end{equation}
(see \cite[\S{6}]{kurihara-invmath}). Since by \cite[Lem.~4.4]{Lei-PhD}\footnote{Whose proof, building on \cite[Prop.~4.1.4]{BLZ}, does not require $a_p=0$.} we have   $\rH^0(\bQ_p(\mu_{p^\infty}),A_f(1))=0$, %for $w_p$ the prime of $\bQ(\mu_{p^\infty})$ above $p$, 
comparing \eqref{eq:PT-H2} and \eqref{eq:loc-H2} and using local Tate duality it follows that 
\[
X_{\rm str}(f/\bQ(\mu_{p^\infty}))\cong\rH^2_{\rm Iw}(\bQ(\mu_{p^\infty}),T).
\]
Thus Conjecture~\ref{conj:kato} (for  $\eta:\Delta\rightarrow\bZ_p^\times$) is equivalent to Conjecture~\ref{introconj:kato} (for $\eta\omega^{k-1}$) in the Introduction.
\end{rem}

\subsubsection{Signed main conjectures}

%Extending work of Kobayashi and Lei in the case $a_p=0$ and of Sprung \cite{sprung-IMC} in the case of elliptic curves $E/\bQ$, Lei--Loeffler--Zerbes \cite{LLZ-AJM} gave the following formulation of signed Iwasawa main conjecture in the nonordinary setting. 

The next recall the formulation of signed Iwasawa main conjectures due to Lei--Loeffler--Zerbes (cf. \cite[Conj.~1.1]{LLZ-ANT}).

%For every character $\eta:\Delta\rightarrow\bZ_p^\times$, let $L_i(f)^\eta\in\Lambda_{\scO}({\Gamma})$ denote the image of $L_i(f)$ in Definition~\ref{def:cyc-Lp} under the projection $e_\eta:\Lambda_{\scO}(\widetilde{\Gamma})\rightarrow\Lambda_{\scO}({\Gamma})$.

\begin{conj}[Lei--Loeffler--Zerbes signed main conjectures]\label{conj:signed}
Let $i\in\{1,2\}$. For every character $\eta:\Delta\rightarrow\bZ_p^\times$ we have 
%$X_i(f)$ is $\Lambda_\scO(\Gamma)$-torsion with
\[
{\rm char}_{\Lambda_\scO(\Gamma)}\bigl(X_i(f/\bQ(\mu_{p^\infty}))^\eta\bigr)={\rm char}_{\Lambda_\scO(\Gamma)}\biggl(\frac{{\rm Image}(\widetilde{\mathscr{C}}_i)^\eta}{(L_{i}(f)^\eta)}\biggr);
\]
%where $L_i(f)^\eta$ is the image of $L_i(f)$ 
%in Definition~\ref{def:cyc-Lp} 
%under the projection $e_\eta:\Lambda_{\scO}(\widetilde{\Gamma})\rightarrow\Lambda_{\scO}({\Gamma})$;
equivalently, we have 
\[
(\xi_{\eta,i})\cdot{\rm char}_{\Lambda_\scO(\Gamma)}\bigl(X_i(f/\bQ(\mu_{p^\infty}))^\eta\bigr)=\bigl(L_i(f)^\eta\bigr),
\]
where $\xi_{\eta,i}$ is as in Proposition~\ref{prop:image-Col}.
\end{conj}

\subsubsection{Equivalence}

Our results on Kato's main conjecture for $f$ %for nonordinary prime 
will build on the following reformulation discovered by Lei--Loeffler--Zerbes \cite{LLZ-AJM}. %\cite[Thm.\,6.5]{LLZ-AJM}.

%\begin{rem}
%Implicit in the statement of Conjecture~\ref{conj:signed} is the prediction that $\mathscr{L}_p^\circ(f)$ is nonzero for each choice of $\circ\in\{+,-\}$.
%\end{rem}

\begin{prop}\label{prop:equiv-kato}
Let $i\in\{1,2\}$. 
%and suppose $\bar{\rho}_f$ irreducible as $G_{\bQ}$-representation. 
If $\eta:\Delta\rightarrow\bZ_p^\times$ is such that $L_i(f)^\eta\neq 0$, then the following are equivalent:
\begin{enumerate}
\item[(i)] $X_i(f/\bQ(\mu_{p^\infty}))^\eta$ is $\Lambda_\scO(\Gamma)$-torsion with
\[
(\xi_{\eta,i})\cdot{\rm char}_{\Lambda_\scO(\Gamma)}\bigl(X_i(f/\bQ(\mu_{p^\infty}))^\eta\bigr)\subset\bigl(L_i(f)^\eta\bigr);
\]
\item[(ii)] $X_{\rm str}(f/\bQ(\mu_{p^\infty}))^\eta$ is $\Lambda_{\scO}(\Gamma)$-torsion with
\[
{\rm char}_{\Lambda_\scO(\Gamma)}\bigl(X_{\rm str}(f/\bQ(\mu_{p^\infty}))^\eta\bigr)\subset{\rm char}_{\Lambda_\scO(\Gamma)}\biggl(\frac{{\rm H}_{\rm Iw}^1(\bQ(\mu_{p^\infty}),T)^\eta}{(\mathbf{z}_{f,k-1}^{\rm int})^\eta}\biggr);
\]
\end{enumerate}
and the same holds for the respective opposite divisibilities. In particular, Conjecture~\ref{conj:signed} for $L_i(f)^\eta$ and Kato's main Conjecture~\ref{conj:kato} for $\eta$ are equivalent.
\end{prop}

\begin{proof}
This is essentially shown in \cite[Thm.~6.5]{LLZ-AJM} (cf. \cite[\S{7}]{kobayashi-ss}), but we provide the details for the convenience of the reader. As noted in the proof of Proposition~\ref{prop:Kato-int}, under our assumptions the residual representation $\bar\rho_f$ is irreducible, and so by \cite[Thm.~12.4]{Kato295} we know that 
$\rH^1_{\rm Iw}(\bQ(\mu_{p^\infty}),T)$ is a free $\Lambda_{\scO}(\widetilde{\Gamma})$-module of rank $1$. From Poitou--Tate duality we have the exact sequence
\begin{equation}\label{eq:PT-1}
{\rm H}^1_{\rm Iw}(\bQ(\mu_{p^\infty}),T)\longrightarrow\frac{{\rm H}^1_{\rm Iw}(\bQ_p(\mu_{p^\infty}),T)}{{\rm H}^1_{{\rm Iw},i}(\bQ_p(\mu_{p^\infty}),T)}\longrightarrow X_i(f/\bQ(\mu_{p^\infty}))\longrightarrow X_{\rm str}(f/\bQ(\mu_{p^\infty}))\longrightarrow 0,
\end{equation}
and by definition and the hypothesis on $\eta$, the composition of the left-most map in \eqref{eq:PT-1} with $\widetilde{\mathscr{C}_i}$ is a nonzero map $\rH^1_{\rm Iw}(\bQ(\mu_{p^\infty}),T)\rightarrow\Lambda_{\scO}(\widetilde{\Gamma})$ taking $\mathbf{z}_{f,k-1}^{\rm int}$ to $L_i(f)$. Thus the latter map is an injection, and quotienting out by the image of $\mathbf{z}_{f,k-1}^{\rm int}$ and taking $\eta$-isotypic components, \eqref{eq:PT-1} yields the short exact sequence
\begin{equation}\label{eq:PT-quot}
0\longrightarrow\frac{{\rm H}^1_{\rm Iw}(\bQ(\mu_{p^\infty}),T)^\eta}{(\mathbf{z}_{f,k-1}^{\rm int})^\eta}\longrightarrow\frac{{\rm Image}(\widetilde{\mathscr{C}}_i)^\eta}{(L_i(f)^\eta)}\longrightarrow X_i(f/\bQ(\mu_{p^\infty}))^\eta\longrightarrow X_{\rm str}(f/\bQ(\mu_{p^\infty}))^\eta\longrightarrow 0.\nonumber
\end{equation}
Using Proposition~\ref{prop:image-Col} and  multiplicativity of characteristic ideals, this yields the result. %from \eqref{eq:PT-quot}.
\end{proof}

%\subsection{Anticyclotomic descent}\label{subsec:ac-descent}

%For an ideal $I\subset\Lambda_K$, let $I^-$ (resp. $I^+$) the image of $I$ under the natural projection $\Lambda_K\rightarrow\Lambda^-$ (resp. $\Lambda_K\rightarrow\Lambda$).

%\begin{prop}\label{prop:cyc-restr}
%We have the divisibility in $\Lambda^-$
%\[
%{\rm char}_{\Lambda_K}\bigl(X_{\rm Gr}(f/K_\infty)\bigr)
%\otimes_{\Lambda_K}\Lambda^-
%\;\supset\;{\rm char}_{\Lambda^-}\bigr(X_{\rm Gr}(f/K_\infty^-)\bigr).
%\]
%\end{prop}

%\begin{proof}
%\end{proof}

%\newpage

\section{Beilinson--Flach elements}\label{sec:BF}

Let $K/\bQ$ be an imaginary quadratic field of discriminant $-D_K<0$ such that
\begin{equation}\label{eq:spl}
\textrm{$p=v\overline{v}$\quad splits in $K$.}\tag{spl}
\end{equation}
For every prime $w$ of $K$ above $p$, let $K_\infty^w$ denote the $\bZ_p$-extension unramified of $K$ outside $w$, and put $\Gamma^w={\rm Gal}(K_\infty^w/K)$. Let ${\rm rec}_K:\mathbb{A}_K^\times\rightarrow G_{K}^{\rm ab}$ be the geometrically normalised reciprocity map of class field theory, and let ${\rm rec}_w:K_w^\times\rightarrow G_{K_w}^{\rm ab}$ denote the restriction of ${\rm rec}_K$ to $K_w^\times$. Then ${\rm rec}_{w}(1+p\bZ_p)$ injects into $\Gamma^w$ (here we identified $1+p\bZ_p$ with the group of $1$-units in $\cO_{K_w}$), and so
\[
[\Gamma^w:{\rm rec}_{w}(1+p\bZ_p)]=p^{h_p}
\]
for some $h_p\geq 0$ (with $h_p=0$ if the class number $h_K$ of $K$ is coprime to $p$). 

As usual, upon the choice of a topological generator $\gamma_w\in\Gamma^w$, we shall often identify $\Lambda_{\scO}(\Gamma^w)$ with the one-variable power series ring $\scO[\![T_w]\!]$ via $\gamma_w\mapsto 1+T_w$.

\subsection{Canonical CM family}\label{subsec:CM}

Following \cite[\S{4.1.6}]{BSTW}, define  $\mathbf{h}_v\in\Lambda_{\bZ_p}(\Gamma^v)[\![q]\!]$ by %the $q$-series 
%canonical CM family %with CM by $K$ 
%as defined in \cite[\S{4.1.6}]{BSTW}, i.e. 
%(\emph{cf.} \cite[\S{5.2}]{jsw}, 
\[
\mathbf{h}_v(q)=\sum_{n=1}^\infty\mathbf{b}_nq^n,\quad
\textrm{}\quad 
\mathbf{b}_n=\sum_{\substack{\mathfrak{a}\subset\cO_K,\mathbf{N}(\mathfrak{a})=n\\ v\nmid\mathfrak{a}}}\Theta_v(x_\mathfrak{a}),
\]
with $\Theta_v:\mathbb{A}_K^\times\rightarrow\Gamma^v$ the composition of ${\rm rec}_K:\mathbb{A}_K^\times\rightarrow G_K^{\rm ab}$ with the projection $G_K^{\rm ab}\rightarrow\Gamma^v$, and where for each ideal $\mathfrak{a}\subset\cO_K$, we let $x_\mathfrak{a}=(x_{\mathfrak{a},w})_w$ be any finite id\`ele with ${\rm ord}_w(x_{\mathfrak{a},w})={\rm ord}_w(\mathfrak{a})$ for all finite places $w$ of $K$ with ${\rm ord}_w(\mathfrak{a})\neq 0$ and $x_{\mathfrak{a},w}=1$ for all other places $w$.

The $q$-series $\mathbf{h}_v(q)$ is a \emph{CM Hida family} of tame level $D_K$ and tame character given by the quadratic Dirichlet character $\epsilon_K$ corresponding to $K$, in the sense that if $\phi:\Lambda_{\bZ_p}(\Gamma^v)\rightarrow\bQ_p$ is the $\bZ_p$-algebra homomorphism defined by $\phi(\gamma_v^{h_p})=(1+p)^m$ for some $m\geq 0$, %where $\gamma_v\in\Gamma^v$ is a topological generator, 
the specialisation
\[
\phi(\mathbf{h}_v)(q):=\sum_{n=1}^\infty\phi(\mathbf{b}_n)q^n
\]
is the $q$-expansion of an ordinary $p$-stabilized newform $\phi(\mathbf{h}_v)\in S_{m+1}(\Gamma_0(D_Kp),\epsilon_K\omega^{-m})$ with complex multiplication by $K$.

\subsubsection{Galois representations attached to CM eigenforms}

Let 
\[
\psi:K^\times\backslash\mathbb{A}_K^\times\longrightarrow\bC^\times
\]
be an algebraic Hecke character.  %of conductor $\mathfrak{c}_\psi$. 
We say that $\psi$ has \emph{infinity type} $(a,b)\in\bZ^2$ is $\psi(z)=z^a\overline{z}^b$ for all $z\in(K\otimes_{\bQ}\bR)^\times\cong\bC^\times$, where the identification is via $\iota_\infty$. The restriction of $\psi$ to the finite id\`{e}les $\mathbb{A}_{K,{\rm f}}^\times$ takes values in $\overline{\bQ}^\times$. The \emph{$p$-adic avatar} of $\psi$ is the character 
$\hat\psi:G_K\rightarrow\overline{\bQ}_p^\times$ defined by 
\[
\hat\psi(\sigma\vert_{K^{\rm ab}})=\iota_p\iota_\infty^{-1}(\psi(x))x_v^ax_{\overline{v}}^b,
\]
where $x\in\mathbb{A}_{K,{\rm f}}^{\times}$ is such that ${\rm rec}_K(x)=\sigma\vert_{K^{\rm ab}}$. Thus, letting $\mathfrak{c}_\psi$ denote the conductor of $\psi$, we have 
\[
\hat\psi({\rm Frob}_w)=\psi(w)
\]
for all finite primes $w$ of $K$ with $w\nmid p\mathfrak{c}_\psi$,  where ${\rm Frob}_w$ denotes a geometric Frobenius at $w$.

Suppose the infinity type of $\psi$ is of the form $(1-k,0)$ for some $k\geq 1$, and let $g=\theta_\psi$ be the theta series of weight $k$ associated to $\psi$. 
Let $L$ be the finite extension of $\bQ_p$ generated by the image under $\iota_p$ of the field generated over $\mathbb{Q}$ by $\iota_\infty^{-1}(\psi(\mathbb{A}_{K,{\rm f}}^\times))$. 
%the values of $\psi$ on $\mathbb{A}_{K,{\rm f}}^\times$. 
Then the $L$-valued Galois representation $V_g$ associated to $g$ by Deligne \cite{deligne-ell-adic} (see also in \cite[\S{8.3}]{Kato295}) is such that
\[
V_g\cong{\rm Ind}_K^\bQ(\hat{\psi}).
\]
%as $G_{\bQ}$-modules.
%where $V_g^*={\rm Hom}_L(V_g,L)$ denotes the contragredient of $g$.

\subsubsection{Galois representation attached to $\mathbf{h}_v$}

Put $\Lambda_D=\bZ_p[\![\bZ_p^\times]\!]$. For integers  $M>0$ prime to $p$, let
\[
\mathfrak{h}_{Mp^\infty}^{\rm ord}:=\varprojlim_r e_{\rm ord}'\mathfrak{h}_{\Gamma_1(Mp^r)},\quad\quad\mathfrak{H}_{Mp^\infty}^{\rm ord}:=\varprojlim_r e_{\rm ord}'\mathfrak{H}_{\Gamma_1(Mp^r)},
\]
be Hida's anti-ordinary Hecke algebras acting on $\Lambda_D$-adic modular forms and cusp forms, respectively (see \cite[\S{7.4}]{KLZ2}), and note that $\mathfrak{H}_{Mp^\infty}^{\rm ord}$ surjects onto $\mathfrak{h}_{Mp^\infty}^{\rm ord}$ by restriction. Put
\begin{equation}\label{eq:big-H}
\begin{aligned}
H^1_{\rm ord}(Mp^\infty)&:=\varprojlim_re_{\rm ord}'\rH^1_{\rm\acute{e}t}(X_1(Mp^r)_{\overline{\bQ}},\bZ_p(1)),\\
\widetilde{H}^1_{\rm ord}(Mp^\infty)&:=\varprojlim_re_{\rm ord}'\rH^1_{\rm\acute{e}t}(Y_1(Mp^r)_{\overline{\bQ}},\bZ_p(1)),
\end{aligned}
\end{equation}
equipped with their natural $\mathfrak{h}_{Mp^\infty}^{\rm ord}[G_\bQ]$- and $\mathfrak{H}_{Mp^\infty}^{\rm ord}[G_\bQ]$-module structures, respectively. 
%
%\begin{defn}\label{def:CM-Tate-lattice}

Let $\varphi_v:\mathfrak{h}_{D_Kp^\infty}^{\rm ord}\rightarrow\Lambda_{\bZ_p}(\Gamma^v)$ be the $\Lambda_D$-algebra homomorphism corresponding to $\mathbf{h}_v$, where $\Lambda_{\bZ_p}(\Gamma^v)$ is viewed as a $\Lambda_D$-module via the $\bZ_p$-algebra homomorphism $\Lambda_D\rightarrow\Lambda_{\bZ_p}(\Gamma^{v})$ determined by $[\langle\varepsilon(\tilde{\gamma})\rangle]\mapsto\varepsilon(\tilde{\gamma})^{-1}\gamma_v$ for all $\tilde{\gamma}\in G_K$ via the cyclotomic character $\varepsilon:G_K\rightarrow\bZ_p^\times$, and put
\begin{equation}\label{def:CM-Tate-lattice}
\begin{aligned}
\mathbb{T}_{\mathbf{h}_v}&:=H^1_{\rm ord}(D_Kp^\infty)\otimes_{\mathfrak{h}_{D_Kp^\infty}^{\rm ord}}\Lambda_{\bZ_p}(\Gamma^v),\\%\quad\quad
\widetilde{\mathbb{T}}_{\mathbf{h}_v}&:=\widetilde{H}^1_{\rm ord}(D_Kp^\infty)\otimes_{\mathfrak{H}_{D_Kp^\infty}^{\rm ord}}\Lambda_{\bZ_p}(\Gamma^v),
\end{aligned}
\end{equation}
where the tensor products are with respect to $\varphi_v$ and its composition with the projection $\mathfrak{H}_{D_Kp^\infty}^{\rm ord}\rightarrow\mathfrak{h}_{D_Kp^\infty}^{\rm ord}$, respectively.
%\end{defn}

Let $\Psi^v:G_K\rightarrow\Lambda_{\bZ_p}(\Gamma^v)^\times$ denote the canonical projection $G_K\rightarrow\Gamma^v$ composed with the inclusion $\Gamma^v\hookrightarrow\Lambda_{\bZ_p}(\Gamma^v)^\times$, %via group-like elements, 
and write $\Lambda_{\bZ_p}(\Gamma^v)^\iota$ for the free $\Lambda_{\bZ_p}(\Gamma^{v})$-module of rank one on which $G_K$ acts via the inverse of $\Psi^v$. 
%The construction of two-variable Beilinson--Flach elements for the base change $f/K$ in the next section will be based on the constructions in \cite{LZ-Coleman} and \cite{BL-non-ord} and the following key result from \cite{BSTW}.

\begin{thm}\label{thm:Ind}
Let $\mathbb{T}_{\mathbf{h}_v}^\square$ denote the quotient of $\mathbb{T}_{\mathbf{h}_v}$ by its torsion submodule. Then 
\[
%%\mathbb{T}_1
\mathbb{T}_{\mathbf{h}_v}^{\square}\cong{\rm Ind}_K^{\bQ}\,\Lambda_{\bZ_p}(\Gamma^v)^\iota
\]
as $\Lambda_{\bZ_p}(\Gamma^v)[G_\bQ]$-modules.
\end{thm}

\begin{proof}
This is shown in Theorem~5.19 of \cite{BSTW}.
\end{proof}

\subsection{Zeta elements over $K$}\label{subsec:zeta}

Let $\alpha$ and $\beta$ be the roots of the Hecke polynomial $x^2-a_px+p^{k-1}$, and for $\xi\in\{\alpha,\beta\}$ let $f^\xi\in S_k(\Gamma_0(Np))$ be the $p$-stabilization of $f$ with $U_p$-eigenvalue $\xi$. Upon enlarging $\scO$ if necessary, we assume that $\alpha,\beta\in\scO$. Recall from \S\ref{subsec:modforms} the $\scO$-lattice $T_f\subset V_f$. Working with level $Np$, we define $T_{f^\xi}\subset V_{f^\xi}$ similarly as before, and write $T_f^*={\rm Hom}_{\scO}(T_f,\scO)$ and likewise define $T_{f^\xi}^*$.

Replacing $\bZ_p(1)$ by the $p$-adic \'{e}tale sheaf ${\rm TSym}^{k-2}\mathscr{H}_r(1)$ on $X_1(Mp^r)$ and $Y_1(Mp^r)$ introduced in \cite[\S{2.3}]{KLZ2}, the inverse limits  \eqref{eq:big-H} define the $\mathfrak{h}^{\rm ord}_{Mp^\infty}[G_\bQ]$- and $\mathfrak{H}^{\rm ord}_{Mp^\infty}[G_\bQ]$-modules $H_{\rm ord}^1(Mp^{\infty})^{[k-2]}$ and $\widetilde{H}_{\rm ord}^1(Mp^\infty)^{[k-2]}$. It follows from \cite[Thm.~4.5.1, Rem.~4.5.3]{KLZ2} %(see also [\emph{loc.\,cit.}, Rem.~4.5.3]), 
that these are identified with the twists of $H_{\rm ord}^1(Mp^{\infty})$ and $\widetilde{H}_{\rm ord}^1(Mp^\infty)$, respectively, by the $\bZ_p$-linear twisting map 
\begin{equation}\label{eq:moment}
m_{k-2}:\Lambda_D\longrightarrow\Lambda_D
\end{equation}
given by $z\mapsto z^{2-k}[z]$ for $z\in\bZ_p^\times$. From these, we define $\mathbb{T}_{\mathbf{h}_v}^{[k-2]}$ and $\widetilde{\mathbb{T}}_{\mathbf{h}_v}^{[k-2]}$ by tensoring with $\varphi_v$ as in \eqref{def:CM-Tate-lattice}, and let $\mathbb{T}_{\mathbf{h}_v}^{\square,[k-2]}$ and $\widetilde{\mathbb{T}}_{\mathbf{h}_v}^{\square,[k-2]}$ denote the quotient by the torsion submodules.

Put $M=ND_K$, fix an integer $c>1$ with $(c,6pM)=1$, and let
\begin{equation}\label{eq:BF-unb}
_c\mathcal{BF}^{\xi,\mathbf{h}_v}\in \rH^1(\bZ[1/p],T_{f^\xi}^*\hat\otimes\bigl(\widetilde{\mathbb{T}}_{\mathbf{h}_v}^{\square,[k-2]}\bigr)\hat\otimes\mathcal{H}_{L}(\widetilde{\Gamma}))
\end{equation}
be the Beilinson--Flach element ${}_c{\rm BF}_{m}^{\xi,\mathbf{g}}$ of \cite[Thm.~3.2]{BL-non-ord} for $m=1$ and $\mathbf{g}=\mathbf{h}_v$ the canonical CM family of $\S\ref{subsec:CM}$; by construction, this interpolates for varying $n\geq 1$ and $0\leq j\leq k-2$  the image of the modified Rankin--Eisenstein classes
\[
\frac{(-1)^j\bigl((U_p')^{-n},(U_p')^{-n}\bigr)}{j!\binom{k-2}{j}^2}{}_{c}\mathcal{RI}_{p^n,p^{n+1}M,1}^{[j]}
\]
of \cite[Def.~5.1.5]{KLZ2} (as well as \cite[Def.~3.2.1]{LZ-Coleman}) under the composite map
\begin{equation}\label{eq:BF-def}
\begin{aligned}
\rH^3_{\rm\acute{e}t}(Y(p^n,p^{n+1}M)^2,\Lambda(\mathscr{H}_{\bZ_p}\langle t_{p^{n+1}M}\rangle)&^{[j,j]}(2-j)\otimes L)\\
&\longrightarrow\rH^1(\bQ(\mu_{p^n}),T_{f^\xi}^*\hat\otimes\widetilde{H}^1_{\rm ord}(Mp^\infty)^{[k-2]}(-j)\otimes L)\\
&\longrightarrow\rH^1(\bQ(\mu_{p^n}),T_{f^\xi}^*\hat\otimes\widetilde{\mathbb{T}}_{\mathbf{h}_v}^{\square,[k-2]}(-j)\otimes L)
\end{aligned}
\end{equation}
of \cite[p.\,10685]{BL-non-ord}, where the second arrow in \eqref{eq:BF-def} is induced by the projection 
\[
\widetilde{H}^1_{\rm ord}(Mp^\infty)^{[k-2]}\longrightarrow\widetilde{\mathbb{T}}_{\mathbf{h}_v}^{[k-2]}\longrightarrow\widetilde{\mathbb{T}}_{\mathbf{h}_v}^{\square,[k-2]}.
\]

By \cite[Prop.~7.3.1]{KLZ2} (see also \cite[Prop.~4.3.6]{LLZ-K}), there is an isomorphism $({\rm Pr}^\xi)_*:T_{f^\xi}^*\xrightarrow{\sim}T_f^*$ induced by a  $p$-stabilisation map ${\rm Pr}^\xi$; in the following we shall thus view the Beilinson--Flach elements \eqref{eq:BF-unb} in $\rH^1(\bZ[1/p],T_{f}^*\hat\otimes(\widetilde{\mathbb{T}}_{\mathbf{h}_v}^{\square,[k-2]})\hat\otimes\mathcal{H}_{L}(\widetilde{\Gamma})^\iota)$. Moreover, as observed in \cite[\S{6.4}]{BSTW}, we may find $c$ as above such that 
\[
\mathbf{c}:=(c^2-\epsilon_K(c)\otimes\langle c\rangle\gamma_v^{-r_c})\otimes\gamma_{\rm cyc}^{2r_c}
\] 
is invertible, where $r_c=\log_p(c)$. Upon the choice of such $c$ we then set
\[
\mathcal{BF}^{\xi,\mathbf{h}_v}:=\mathbf{c}^{-1}\cdot{}_c\mathcal{BF}^{\xi,\mathbf{h}_v}\in\rH^1(\bZ[1/p],T\hat\otimes(\widetilde{\mathbb{T}}_{\mathbf{h}_v}^{\square,[k-2]})\hat\otimes\mathcal{H}_{L}(\widetilde{\Gamma})^\iota),
\]
which is independent of $c$, and where we used the identification $T_f^*\cong T$ (since $f$ has trivial nebentypus, and so the complex conjugate $\bar{f}$  is $f$  itself).

\begin{thm}%[B\"uy\"ukboduk--Lei]
\label{thm:BF-factor}
There exist classes  
\[
{\mathcal{BF}}^{1,\mathbf{h}_v},\;{\mathcal{BF}}^{2,\mathbf{h}_v}
\in 
\rH^1(\bZ[1/p],T\hat\otimes(\widetilde{\mathbb{T}}_{\mathbf{h}_v}^{\square,[k-2]})\hat\otimes\Lambda_{\bZ_p}(\widetilde{\Gamma})^\iota)\otimes_{}L
\] 
such that
\begin{equation}
\left(\begin{array}{cc}
\mathcal{BF}^{\alpha,\mathbf{h}_v}\\
\mathcal{BF}^{\beta,\mathbf{h}_v}
\end{array}\right)
=Q_{\alpha,\beta}^{-1}M_{{\rm log}}\cdot
\left(\begin{array}{cc}
{\mathcal{BF}}^{1,\mathbf{h}_v}\\
{\mathcal{BF}}^{2,\mathbf{h}_v}
\end{array}\right),\nonumber
\end{equation}
where $Q_{\alpha,\beta}=\frac{1}{\alpha-\beta}\bigl(\begin{smallmatrix}\alpha & -\beta\\ -p^{k-1}&p^{k-1}\end{smallmatrix}\bigr)$ and $M_{{\rm log}}\in M_{2\times 2}(\mathcal{H}_L(\Gamma))$ is the logarithm matrix of Definition~\ref{def:log-matrix}.
\end{thm}

\begin{proof}
This follows from \cite[Thm.~3.7]{BL-non-ord}.
\end{proof}

Fix some $s\geq 0$ such that for both $i\in\{1,2\}$ the classes
\begin{equation}\label{eq:varpi-s}
\mathcal{BF}_{\rm int}^{i,\mathbf{h}_v}:=\varpi^s\cdot\mathcal{BF}^{i,\mathbf{h}_v}
\end{equation}
land in $\rH^1(\bZ[1/p],T\hat\otimes(\widetilde{\mathbb{T}}_{\mathbf{h}_v}^{\square,[k-2]})\hat\otimes\Lambda_{\bZ_p}(\widetilde{\Gamma})^\iota)$ (the particular choice of $s$ will be irrelevant  
for our arguments later). 
In the following, twisting by the inverse of the character $m_{k-2}$ of \eqref{eq:moment}, we shall view $\mathcal{BF}_{\rm int}^{i,\mathbf{h}_v}$ as living in $\rH^1(\bZ[1/p],T\hat\otimes(\widetilde{\mathbb{T}}_{\mathbf{h}_v}^{\square})\hat\otimes\Lambda_{\bZ_p}(\widetilde{\Gamma})^\iota)$.

\subsection{Two-variable Coleman maps}

Let $F_\infty$ be the unramified $\bZ_p$-extension of $\bQ_p$, and put
\[
U:={\rm Gal}(F_\infty/\bQ_p),\quad\quad G:={\rm Gal}(F_\infty(\mu_{p^\infty})/\bQ_p)\cong\widetilde{\Gamma}\times U.
\]
We begin by recalling the construction by Loeffler--Zerbes \cite{LZ2} of a two-variable big logarithm map $\mathscr{L}_{V,F_\infty}$ interpolating Perrin-Riou's
\[
\mathscr{L}_{V,F}:\rH^1_{\rm Iw}(F(\mu_{p^\infty}),V)\longrightarrow\mathbb{D}_{\rm cris}(V)\otimes_{\bQ_p}\mathcal{H}_L(\widetilde{\Gamma})
\]
(cf. \eqref{eq:PR-Qp}) over the finite extensions $F/\bQ_p$ contained in $F_\infty$.

For any finite unramified extension $F/\bQ_p$, write $\mathbb{N}_F(T)$ for the Wach module attached to $T\vert_{G_F}$. By \cite[Prop.~4.5]{LZ2}, the Fontaine--Berger isomorphism $\rH^1_{\rm Iw}(F(\mu_{p^\infty}),T)\cong\mathbb{N}_F(T)^{\psi=1}$ extends to
\[
\rH^1_{\rm Iw}(F_\infty(\mu_{p^\infty}),T)\cong\mathbb{N}_{F_\infty}(T)^{\psi=1},
\]
where $\mathbb{N}_{F_\infty}(T)=\varprojlim_F\mathbb{N}_F(T)$ with inverse limit relative to the trace maps, and with $F$ running over the finite extensions of $\bQ_p$ contained in $F_\infty$; and by Proposition~3.12 in \emph{op.\,cit.} we  have
\[
(\varphi^*\mathbb{N}_{F_\infty}(T))^{\psi=0}=(\varphi^*\mathbb{N}_{}(T))^{\psi=0}\hat\otimes_{\bZ_p}S_{F_\infty/\bQ_p},
\]
where $S_{F_\infty/\bQ_p}\subset\Lambda_{\widehat{\scO}_{F_\infty}}(U)$ %--- the \emph{Yager module} --- 
is a free $\Lambda_{\bZ_p}(U)$-module of rank one. Following  \cite[Def.~4.6]{LZ2} (see also \cite[\S{2.3}]{BL-non-ord}), letting $\Omega_{\bQ_p}$ be a $\Lambda_{\bZ_p}(U)$-basis of $S_{F_\infty/\bQ_p}$, 
similarly as in \eqref{eq:PR-2var} we define 
\begin{equation}\label{eq:PR-2var}
\mathscr{L}_{V,F_\infty}:\rH^1_{\rm Iw}(F_\infty(\mu_{p^\infty}),T)\longrightarrow\Omega_{\bQ_p}\cdot\mathcal{H}_L(G)
\otimes_{\bQ_p}\mathbb{D}_{\rm cris}(V)\nonumber
\end{equation}
by the composition
\begin{equation}\label{eq:PR-2var-def}
\begin{aligned}
\rH^1_{\rm Iw}(F_\infty(\mu_{p^\infty}),T)\overset{\cong}\longrightarrow\mathbb{N}_{F_\infty}(T)^{\psi=1}
&\xrightarrow{1-\varphi}
(\varphi^*\mathbb{N}_{}(T))^{\psi=0}\hat\otimes_{\bZ_p}S_{F_\infty/\bQ_p}\\
&\hooklongrightarrow\mathbb{D}_{\rm cris}(V)\otimes_{\bQ_p}(\mathbb{B}_{{\rm rig},\bQ_p}^+)^{\psi=0}\hat\otimes_{\bZ_p}S_{F_\infty/\bQ_p}\\
&\xrightarrow{\mathfrak{M}^{-1}}\mathbb{D}_{\rm cris}(V)\otimes_{\bQ_p}(\mathcal{H}_L(\widetilde{\Gamma})\hat\otimes_{\bZ_p}S_{F_\infty/\bQ_p})\\
&\hooklongrightarrow\Omega_{\bQ_p}\cdot\mathcal{H}_L(G)
\otimes_{\bQ_p}\mathbb{D}_{\rm cris}(V),
\end{aligned}
\end{equation}
using the embedding of \cite[Prop.~2.11]{LLZ-ANT} and the inclusion $\mathcal{H}_L(\widetilde{\Gamma})\hat\otimes_{\bZ_p}\Lambda_{\bZ_p}(U)\subset\mathcal{H}_L(\widetilde{\Gamma})\hat\otimes_{\bZ_p}\mathcal{H}(U)=\mathcal{H}_L(G)$, respectively, for the unlabeled arrows. 
%In the same manner as in Proposition~\ref{prop:signed-Col-cyc}, we then have the following.

\begin{prop}\label{prop:signed-Col-2var}
Let $(n_1,n_2)$ be a $\mathbb{A}_{\bQ_p}^+$-basis of $\mathbb{N}(T)$ as in Lemma~\ref{lem:Wach-basis}. Then 
we can write
%$\mathscr{L}_{V,F_\infty}$ %of \eqref{eq:PR-2var} 
%decomposes as
\begin{equation}\label{eq:dec-L-2var}
\mathscr{L}_{V,F_\infty}=(v_1\;,\;v_2)\cdot M_{{\rm log}}\cdot
\biggl(\begin{array}{cc}
\widetilde{\mathscr{C}}_{1,F_\infty}\\
\widetilde{\mathscr{C}}_{2,F_\infty}
\end{array}\biggr),
\end{equation}
where $\widetilde{\mathscr{C}}_{i,F_\infty}$
%:\rH^1_{\rm Iw}(F_\infty(\mu_{p^\infty}),T)\longrightarrow\Omega_{\bQ_p}\cdot\Lambda_{\cO}(G)$ 
is defined as the composition
\begin{align*}
\widetilde{\mathscr{C}}_{i,F_\infty}:\rH^1_{\rm Iw}(F_\infty(\mu_{p^\infty}),T)\overset{\cong}\longrightarrow\mathbb{N}_{F_\infty}(T)^{\psi=1}
&\xrightarrow{1-\varphi}(\varphi^*\mathbb{N}(T))^{\psi=0}\hat\otimes_{\bZ_p}S_{F_\infty/\bQ_p}\\
&\xrightarrow{{\rm pr}_i\circ\mathfrak{I}}\Lambda_{\scO}(\widetilde{\Gamma})\hat\otimes_{\bZ_p}S_{F_\infty/\bQ_p}\\
&=\Lambda_{\scO}(\widetilde{\Gamma})\hat\otimes_{\bZ_p}(\Omega_{\bQ_p}\cdot\Lambda_{\bZ_p}(U))\\&=\Omega_{\bQ_p}\cdot\Lambda_{\scO}(G),
\end{align*}
where the second arrow is given by taking the $i$-th coordinate with respect to the basis $(n_1,n_2)$. 
\end{prop}

\begin{proof}
In the same way as in Proposition~\ref{prop:signed-Col-cyc}, this follows from the description of $\mathscr{L}_{V,F_\infty}$ in \eqref{eq:PR-2var-def} and the definition of $M_{\rm log}$.
\end{proof}

\begin{prop}\label{prop:image-Col-2var}
Let $i\in\{1,2\}$. For every character $\eta:\Delta\rightarrow\bZ_p^\times$, we have an inclusion %we have the inclusion 
\[
{\rm Image}(e_\eta\widetilde{\mathscr{C}}_{i,F_\infty})\subset \Omega_{\bQ_p}\cdot\xi_{\eta,i}\Lambda_{\scO}(G)
\]
%and this has 
with finite index, where  $\xi_{\eta,i}$ is the same factor of $\delta_{k-1}$ as in Proposition~\ref{prop:image-Col}.
\end{prop}

\begin{proof}
This follows readily from the construction of $\widetilde{\mathscr{C}}_{i,F_\infty}$ and Proposition~\ref{prop:image-Col} (see \cite[Lem.~2.16]{BL-non-ord} for the details).
\end{proof}

Similarly as in $\S\ref{subsubsec:signed-Col-cyc}$, letting 
\[
\mathscr{L}_{V,\xi,F_\infty}=\mathscr{L}_{V,(1+\pi)v_\xi^*,F_\infty}:\rH^1_{\rm Iw}(F_\infty(\mu_{p^\infty}),T)\longrightarrow\Omega_{\bQ_p}\cdot\mathcal{H}_L(G)
\]
be the coordinate of $\mathscr{L}_{V,F_\infty}$ with respect to $v_\xi$, %so that $\mathscr{L}_{V,F_\infty}=\mathscr{L}_{V,\alpha,F_\infty}\cdot v_\alpha+\mathscr{L}_{V,\beta,F_\infty}\cdot v_\beta$, 
the decomposition \eqref{eq:dec-L-2var} %in Proposition~\ref{prop:signed-Col-2var} 
can be rewritten as
\begin{equation}\label{eq:factor-L-2var}
\biggl(\begin{array}{cc}
\mathscr{L}_{V,\alpha,F_\infty}\\
\mathscr{L}_{V,\beta,F_\infty}
\end{array}\biggr)
=Q_{\alpha,\beta}^{-1} M_{{\rm log}}\cdot
\biggl(\begin{array}{cc}
\widetilde{\mathscr{C}}_{1,F_\infty}\\
\widetilde{\mathscr{C}}_{2,F_\infty}
\end{array}\biggr),
\end{equation}
where $Q_{\alpha,\beta}=\frac{1}{\alpha-\beta}\bigl(\begin{smallmatrix}\alpha & -\beta\\ -p^{k-1}&p^{k-1}\end{smallmatrix}\bigr)$.

\subsubsection{Semi-local signed Coleman maps} 

Let 
\[
\Gamma_K={\rm Gal}(K_\infty/K)\cong\bZ_p^2
\]
be the Galois group of the $\bZ_p^2$-extension  $K_\infty/K$. For every $w\in\{v,\overline{v}\}$, fix a prime $\widetilde{w}_0$ of $K_\infty$ above $w$, and let $\Gamma_{\widetilde{w}_0}\subset\Gamma_K$ denote the  associated decomposition group at $w_0$. For every character $\eta:\Delta\rightarrow\bZ_p^\times$, the 
%$K_{\infty,\widetilde{w}}$ is an extension of $K_w$ of the form $F_\infty(\mu_{p^\infty})^\Delta$ for some infinite unramified extension $F_\infty/\bQ_p$, where $\Delta={\rm Gal}(K_w(\mu_\infty)/K_w)_{\rm tor}$, and so 
constructions of the preceding paragraphs give 
%(after taking $\eta$-isotypic components) 
maps
\begin{equation}\label{eq:L-Col-wtilde}
\begin{aligned}
e_\eta\mathscr{L}_{V,\xi,K_{\infty,\widetilde{w}_0}}:\rH^1_{\rm Iw}(K_{\infty,\widetilde{w}_0},V(\eta))&\longrightarrow\Omega_{\bQ_p}\cdot\mathcal{H}_L(\Gamma_{\widetilde{w}_0}),\\
e_\eta\widetilde{\mathscr{C}}_{i,K_{\infty,\widetilde{w}_0}}:\rH^1_{\rm Iw}(K_{\infty,\widetilde{w}_0},T(\eta))&\longrightarrow\Omega_{\bQ_p}\cdot\Lambda_{\scO}(\Gamma_{\widetilde{w}_0}),
\end{aligned}
\end{equation}
where $V(\eta)$ denotes the twist $V\otimes\eta$, and likewise for $T(\eta)$. Using the Shapiro's lemma isomorphism 
\begin{align*}
\rH^1(K_w,T(\eta)\otimes\Lambda_{\scO}(\Gamma_K)^\iota)&
\cong\bigoplus_{\widetilde{w}\mid w}\rH^1_{\rm Iw}(K_{\infty,\widetilde{w}},T(\eta))\\
&\cong\Lambda_{\scO}(\Gamma_K)\otimes_{\Lambda_{\scO}(\Gamma_{\widetilde{w}_0})}\rH^1_{\rm Iw}(K_{\infty,\widetilde{w}_0},T(\eta)),
\end{align*} 
where $\widetilde{w}$ runs over the primes of $K_\infty$ above $w$, we define
\begin{equation}\label{eq:def-2var}
\begin{aligned}
e_\eta\mathscr{L}_{V,\xi,w}:\rH^1(K_w,V(\eta)\otimes\Lambda_{\scO}(\Gamma_K)^\iota)&\longrightarrow\Omega_{\bQ_p}\cdot\mathcal{H}_L(\Gamma_K),\\
e_\eta\widetilde{\mathscr{C}}_{i,w}:\rH^1(K_w,T(\eta)\otimes\Lambda_{\scO}(\Gamma_K)^\iota)&\longrightarrow\Omega_{\bQ_p}\cdot\Lambda_\scO(\Gamma_K)
\end{aligned}
\end{equation}
by tensoring \eqref{eq:L-Col-wtilde} with $\Lambda_{\scO}(\Gamma_K)$ over $\Lambda_{\scO}(\Gamma_{\widetilde{w}_0})$. Then from \eqref{eq:factor-L-2var} we obtain the relations
\begin{equation}\label{eq:factor-L-2var-w}
\biggl(\begin{array}{cc}
e_\eta\mathscr{L}_{V,\alpha,w}\\
e_\eta\mathscr{L}_{V,\beta,w}
\end{array}\biggr)
=Q_{\alpha,\beta}^{-1} M_{{\rm log},w}\cdot
\biggl(\begin{array}{cc}
e_\eta\widetilde{\mathscr{C}}_{1,w}\\
e_\eta\widetilde{\mathscr{C}}_{2,w}
\end{array}\biggr),
\end{equation}
where $M_{{\rm log},w}\in M_{2\times 2}(\mathcal{H}_L(\Gamma^w))$ is the image of the logarithm matrix $M_{\rm log}$ of Definition~\ref{def:log-matrix} under the inverse of the isomorphism given by the composition $\Gamma^w\hookrightarrow\Gamma_K\twoheadrightarrow\Gamma$.

\subsection{Two-variable Selmer groups}\label{subsec:2var-Sel}

For every character $\eta:\Delta\rightarrow\bZ_p^\times$, consider the $G_K$-module
\[
\mathbf{T}_{K}^\eta:=T(\eta)\otimes\Lambda_{\scO}(\Gamma_K)^\iota,\quad\quad\mathbf{A}_K^\eta:=T(\eta)^\vee(1)\otimes\Lambda_{\scO}(\Gamma_K),
\]
where we recall that $T=T_f(k-1)$. For $\bullet\in\{1,2,{\rm rel},{\rm str}\}$ and $w\in\{v,\overline{v}\}$ a prime of $K$ above $p$, put
\[
\rH^1_\bullet(K_w,\mathbf{T}_K^\eta):=\begin{cases}
{\rm ker}(e_\eta\widetilde{\mathscr{C}}_{\bullet,w})&\textrm{if $\bullet\in\{1,2\}$,}\\[0.2em]
\rH^1(K_w,\mathbf{T}_K^\eta)&\textrm{if $\bullet={\rm rel}$,}\\[0.2em]
0&\textrm{if $\bullet={\rm str}$},
\end{cases}
\]
and let $\rH^1_{\bullet}(K_w,\mathbf{A}_K^\eta)$ be the orthogonal complement of $\rH^1_{\bullet}(K_w,\mathbf{A}_K^\eta)$ under the local Tate duality
\[
\rH^1(K_w,\mathbf{T}_K^\eta)\times\rH^1(K_w,\mathbf{A}_K^\eta)\longrightarrow\bQ_p/\bZ_p.
\]

%\begin{defn}\label{def:2var-Sel}
For $(\bullet,\circ)\in\{1,2,{\rm rel},{\rm str}\}^{\oplus 2}$, define the (compact) Selmer group $\rH^1_{\bullet,\circ}(K,\mathbf{T}_K^\eta)$ by setting
\begin{equation}\label{eq:def-2var-Sel}
\rH^1_{\bullet,\circ}(K,\mathbf{T}_K^\eta):={\rm ker}\Biggl\{\rH^1(K,\mathbf{T}_K^\eta)\longrightarrow\frac{\rH^1(K_v,\mathbf{T}_K^\eta)}{\rH^1_\bullet(K_v,\mathbf{T}_K^\eta)}\times\frac{\rH^1(K_{\overline{v}},\mathbf{T}_K^\eta)}{\rH^1_\circ(K_{\overline{v}},\mathbf{T}_K^\eta)}\times\prod_{w\nmid p}\rH^1(K_w^{\rm unr},\mathbf{T}_K^\eta)\Biggr\};\nonumber
\end{equation}
%where $\rH^1_{\rm unr}(K_w,\mathbf{T}_K^\eta)={\rm ker}\{\rH^1(K_w,\mathbf{T}_K^\eta)\rightarrow\rH^1(K_w^{\rm unr},\mathbf{T}_K^\eta)\}$ is the unramified submodule; 
and the \emph{dual (discrete) Selmer group} ${\rm Sel}_{\bullet,\circ}(f_\eta/K_\infty)\subset\rH^1(K,\mathbf{A}_K^\eta)$ by replacing $\mathbf{T}_K^\eta$ with $\mathbf{A}_K^\eta$ in the above definition. Finally, write $X_{\bullet,\circ}(f_\eta/K_\infty)$  
%$={\rm Sel}_{\bullet,\circ}(f/K_\infty)^{\eta,\vee}$ 
for the Pontryagin dual of ${\rm Sel}_{\bullet,\circ}(f_\eta/K_\infty)$.

%\eqref{eq:def-2var-Sel}. 

\subsection{Explicit reciprocity laws}\label{subsec:ERL}

In this section we explain the relation between the classes $\mathcal{BF}^{i,\mathbf{h}_v}$ of Theorem~\ref{thm:BF-factor} and two different types of %special values of 
Rankin--Selberg $p$-adic $L$-functions arising from the explicit reciprocity laws established in \cite{LZ-Coleman}.

%\subsection{Zeta elements}
\subsubsection{Filtrations}

By Ohta's work \cite{OhtaII}, the modules $H^1_{\rm ord}(Mp^\infty)$ and $\widetilde{H}^1_{\rm ord}(Mp^\infty)$ introduced in $\S\ref{subsec:CM}$ 
are equipped with $\mathfrak{H}_{Mp^\infty}^{\rm ord}[G_{\bQ_p}]$-stable filtrations 
\begin{align*}
\mathscr{F}^+H^1_{\rm ord}(Mp^\infty)&\subset\mathscr{F}^+H^1_{\rm ord}(Mp^\infty),\\%\quad\quad
\mathscr{F}^+\widetilde{H}^1_{\rm ord}(Mp^\infty)&\subset\mathscr{F}^+\widetilde{H}^1_{\rm ord}(Mp^\infty)
\end{align*}
with $\mathscr{F}^+H^1_{\rm ord}(Mp^\infty)=\mathscr{F}^+\widetilde{H}^1_{\rm ord}(Mp^\infty)$ isomorphic to $\mathfrak{h}_{Mp^\infty}^{\rm ord}$ as $\mathfrak{H}^{\rm ord}_{Mp^\infty}$-modules. Put 
\begin{align*}
\mathbb{T}^+_{\mathbf{h}_v}&:=\mathscr{F}^+H^1_{\rm ord}(D_Kp^\infty)\otimes_{\mathfrak{h}_{D_Kp^\infty}^{\rm ord}}\Lambda_{\bZ_p}(\Gamma^v),\\
\widetilde{\mathbb{T}}^+_{\mathbf{h}_v}&:=\mathscr{F}^+\widetilde{H}^1_{\rm ord}(D_Kp^\infty)\otimes_{\mathfrak{h}_{D_Kp^\infty}^{\rm ord}}\Lambda_{\bZ_p}(\Gamma^v),
\end{align*}
%\mathbb{T}^+_{\mathbf{h}_v}&:=\mathscr{F}^+H^1_{\rm ord}(D_Kp^\infty)\otimes_{\mathfrak{h}_{D_Kp^\infty}^{\rm ord}}\Lambda_{\bZ_p}(\Gamma^v),
%\quad\quad
%&&\widetilde{\mathbb{T}}^+_{\mathbf{h}_v}:=\mathscr{F}^+\widetilde{H}^1_{\rm ord}(D_Kp^\infty)\otimes_{\mathfrak{h}_{D_Kp^\infty}^{\rm ord}}\Lambda_{\bZ_p}(\Gamma^v),\\
%\mathbb{T}^-_{\mathbf{h}_v}&:=\mathbb{T}_{\mathbf{h}_v}/\mathbb{T}^+_{\mathbf{h}_v},\quad\quad
%&&\widetilde{\mathbb{T}}^-_{\mathbf{h}_v}:=\widetilde{\mathbb{T}}_{\mathbf{h}_v}/\widetilde{\mathbb{T}}^+_{\mathbf{h}_v},
%\end{align*}
%and $\mathbb{T}^-_{\mathbf{h}_v}:=\mathbb{T}_{\mathbf{h}_v}/\mathbb{T}^+_{\mathbf{h}_v}$, $\mathbb{H}^-_{\mathbf{h}_v}:=\mathbb{H}_{\mathbf{h}_v}/\mathbb{H}^+_{\mathbf{h}_v}$.
%and $\mathbb{T}^-_{\mathbf{h}_v}:=\mathbb{T}_{\mathbf{h}_v}/\mathbb{T}^+_{\mathbf{h}_v}$, $\widetilde{\mathbb{T}}^-_{\mathbf{h}_v}:=\widetilde{\mathbb{T}}_{\mathbf{h}_v}/\widetilde{\mathbb{T}}^+_{\mathbf{h}_v}$
so we have a commutative diagram with exact rows
\begin{equation}\label{eq:fil-Ohta}
\begin{aligned}
\xymatrix{
0\ar[r]&\mathbb{T}_{\mathbf{h}_v}^+\ar[r]\ar@{=}[d]&\mathbb{T}_{\mathbf{h}_v}^{\square}\ar[r]\ar@{^{(}->}[d]&\mathbb{T}_{\mathbf{h}_v}^-:=\mathbb{T}_{\mathbf{h}_v}^{\square}/\mathbb{T}^+_{\mathbf{h}_v}\ar[r]\ar@{^{(}->}[d]&0\\
0\ar[r]&\widetilde{\mathbb{T}}_{\mathbf{h}_v}^+\ar[r]&\widetilde{\mathbb{T}}_{\mathbf{h}_v}^{\square}\ar[r]&\widetilde{\mathbb{T}}_{\mathbf{h}_v}^-:=\widetilde{\mathbb{T}}_{\mathbf{h}_v}^{\square}/\widetilde{\mathbb{T}}^+_{\mathbf{h}_v}\ar[r]&0.
}
\end{aligned}
\end{equation}

\subsubsection{Triangulations} 
Let $\mathscr{R}$ denote %\footnote{consistently with the notation used in $\S\ref{subsec:Kato-ES}$} 
the Robba ring, consisting of Laurent power series in $L(\!(\pi)\!)$ %in the variable $\pi$ 
convergent in some annulus inside $\{\vert\pi\vert<1\}$, and let $\mathbb{D}_{\rm rig}^\dagger$ be Fontaine's functor from $p$-adic $G_{\bQ_p}$-representations to $(\varphi,\widetilde{\Gamma})$-modules over $\mathscr{R}$, as extended by Berger--Colmez \cite{berger-colmez-familles} to $p$-adic families. 

Put $\mathscr{F}^{\pm}\mathbf{D}_{\rm rig}^\dagger(\widetilde{\mathbb{T}}^\square_{\mathbf{h}_v})=\mathbf{D}_{\rm rig}^\dagger(\mathscr{F}^\pm\widetilde{\mathbb{T}}_{\mathbf{h}_v}^\square)$, and similarly with $\mathbb{T}^\square_{\mathbf{h}_v}$; and for $\xi\in\{\alpha,\beta\}$ and $\circ,\bullet\in\{+,-,\emptyset\}$ put
\[
\mathscr{F}_{\xi,{\rm ord}}^{\circ,\bullet}\mathbb{D}_{\rm rig}^\dagger(V\otimes\widetilde{\mathbb{T}}_{\mathbf{h}_v}^{\square}):=\mathscr{F}_\xi^\circ\mathbb{D}_{\rm rig}^\dagger(V)\otimes\mathscr{F}^\bullet\mathbb{D}_{\rm rig}^\dagger(\widetilde{\mathbb{T}}_{\mathbf{h}_v}^{\square}),
\]
and similarly with $\mathbb{T}_{\mathbf{h}_v}^{\square}$ in place of $\widetilde{\mathbb{T}}_{\mathbf{h}_v}^{\square}$, where $\mathscr{F}_\xi^\pm$ denotes the filtered pieces in the triangulation 
\[
0\longrightarrow\mathscr{F}_{\xi}^+\mathbb{D}_{\rm rig}^\dagger(V)\longrightarrow\mathbb{D}_{\rm rig}^\dagger(V)\longrightarrow\mathscr{F}_\xi^-\mathbb{D}_{\rm rig}^\dagger(V)\longrightarrow 0
\]
arising from \cite{kisin-FM} and its extension to Coleman families in \cite{liu-triangulation}.

The regulator maps $e_\eta\mathscr{L}_{V,\xi,w}$ of \eqref{eq:def-2var} extend by linearity to
\begin{equation}\label{eq:def-2var-H}
e_\eta\mathscr{L}_{V,\xi,w}:\rH^1(K_w,T(\eta)\otimes\Lambda_{\scO}(\Gamma_K)^\iota\otimes_{\Lambda_{\scO}(\Gamma)}\mathcal{H}_L(\Gamma))\longrightarrow\Omega_{\bQ_p}\cdot\mathcal{H}_L(\Gamma_K)
\end{equation}
using the isomorphism (see \cite[Prop.~2.4.2]{LZ-Coleman})
\[
\rH^1\bigl(K_w,V(\eta)\otimes\Lambda_{\scO}(\Gamma_K)^\iota\bigr)\otimes_{\Lambda_{\scO}(\Gamma)}\mathcal{H}_L(\Gamma)\cong\rH^1\bigl(K_w,V(\eta)\otimes\Lambda_{\scO}(\Gamma_K)^\iota\otimes_{\Lambda_{\scO}(\Gamma)}\mathcal{H}_L(\Gamma)\bigr).
\]

\subsubsection{Explicit reciprocity law for $w=v$}

As shown in \cite[Thm.~7.1.2]{LZ-Coleman}, the image of ${\rm res}_p(\mathcal{BF}^{\xi,\mathbf{h}_v})$ under the composite map
\begin{align*}
\rH^1(\bQ_p,V\otimes(\widetilde{\mathbb{T}}_{\mathbf{h}_v}^{\square})\otimes\Lambda_{\scO}(\widetilde{\Gamma})^\iota)\otimes_{\Lambda_{\scO}(\widetilde{\Gamma})}\mathcal{H}_L(\widetilde{\Gamma})&\cong\rH^1_{\rm Iw}(\bQ_p(\mu_{p^\infty}),\mathbb{D}_{\rm rig}^\dagger(V\otimes\widetilde{\mathbb{T}}_{\mathbf{h}_v}^{\square}))\\
&\longrightarrow\rH^1_{\rm Iw}(\bQ_p(\mu_{p^\infty}),\mathscr{F}_{\xi,{\rm ord}}^{-,\emptyset}\mathbb{D}_{\rm rig}^\dagger(V\otimes\widetilde{\mathbb{T}}_{\mathbf{h}_v}^{\square}))
\end{align*}
is contained in the image of the injection
\[
\rH^1_{\rm Iw}(\bQ_p(\mu_{p^\infty}),\mathscr{F}_{\xi,{\rm ord}}^{-,+}\mathbb{D}_{\rm rig}^\dagger(V\otimes\widetilde{\mathbb{T}}_{\mathbf{h}_v}^{\square}))\longrightarrow\rH^1_{\rm Iw}(\bQ_p(\mu_{p^\infty}),\mathscr{F}_{\xi,{\rm ord}}^{-,\emptyset}\mathbb{D}_{\rm rig}^\dagger(V\otimes\widetilde{\mathbb{T}}_{\mathbf{h}_v}^{\square})).
\]
On the other hand, under the isomorphism
\[
\rH^1(K_v,V(\eta)\otimes\Lambda_{\bZ_p}(\Gamma^v)\hat\otimes\Lambda_{\bZ_p}(\Gamma)^\iota)\otimes_{\Lambda_{\scO}(\Gamma)}\mathcal{H}_L({\Gamma})\cong
\rH^1_{\rm Iw}(\bQ_p(\mu_{p^\infty}),\mathscr{F}_{\xi,{\rm ord}}^{\emptyset,+}\mathbb{D}_{\rm rig}^\dagger(V\otimes\widetilde{\mathbb{T}}_{\mathbf{h}_v}^{\square}))^\eta
\]
induced by the identification $K_v=\bQ_p$ and the isomorphism %$\mathbb{T}_{\mathbf{h}_v}^+=
\[
\widetilde{\mathbb{T}}_{\mathbf{h}_v}^+=\mathbb{T}^+_{\mathbf{h}_v}\cong\Lambda_{\bZ_p}(\Gamma^v)^\iota
\]
as $G_{\bQ_p}$-modules (cf. \eqref{eq:fil-Ohta} and Theorem~\ref{thm:BSTW}) defined by the $\Lambda$-adic Ohta's differential  $\omega_{\mathbf{h}_v}$ attached to $\mathbf{h}_v$ as in \cite[Prop.~10.1.1]{KLZ2}, the extended regulator map $e_\eta\mathscr{L}_{V,\xi,v}$ of \eqref{eq:def-2var-H} factors through the natural map
\[
\rH^1_{\rm Iw}(\bQ_p(\mu_{p^\infty}),\mathscr{F}_{\xi,{\rm ord}}^{\emptyset,+}\mathbb{D}_{\rm rig}^\dagger(V\otimes\widetilde{\mathbb{T}}_{\mathbf{h}_v}^{\square}))^\eta\longrightarrow\rH^1_{\rm Iw}(\bQ_p(\mu_{p^\infty}),\mathscr{F}_{\xi,{\rm ord}}^{-,+}\mathbb{D}_{\rm rig}^\dagger(V\otimes\widetilde{\mathbb{T}}_{\mathbf{h}_v}^{\square}))^\eta,
\]
and so $e_\eta\mathscr{L}_{V,\xi,v}$ can naturally be evaluated at ${\rm res}_p(\mathcal{BF}^{\xi,\mathbf{h}_v,\eta})$ for the  $\eta$-isotypic component $\mathcal{BF}^{\xi,\mathbf{h}_v,\eta}$ of $\mathcal{BF}^{\xi,\mathbf{h}_v}$.

\begin{defn}\label{def:2var-unb}
For $\xi\in\{\alpha,\beta\}$ and a character  $\eta:\Delta\rightarrow\bZ_p^\times$, put
\[
L_{\xi,\xi}^{\rm (I)}(f,\mathbf{h}_v)^\eta:=e_\eta\mathscr{L}_{V,\xi,v}({\rm res}_p(\mathcal{BF}^{\xi,\mathbf{h}_v,\eta})).
\]
%which yields an element in $\mathcal{H}_L(\Gamma_K)$.
\end{defn}

%Note that $L_{\xi,\xi}(f,\mathbf{h}_v)^\eta$ is an element of $\mathcal{H}_L(\Gamma_K)$ with growth $h=2v_p(\xi)$ in the cyclotomic direction. 
Put
\[
\Xi^{\rm (I)}:=\bigl\{(1,\chi_2):\Gamma^v\hat\otimes\Gamma\rightarrow\overline{\bQ}_p^\times\;|\;\chi_2=\langle\varepsilon\rangle^j\psi_\zeta,\;\,0\leq j\leq k-2, \;\,\zeta\in\mu_{p^\infty}\bigr\},
\]
where $\psi_\zeta$ is the character associated to $\zeta$ as in $\S\ref{subsec:interp-MTT}$.

%We shall only need the following interpolation property satisfied by $L_{\xi_1,\xi_2}(f,\mathbf{h}_v)$ in the case $\xi_1=\xi_2$, so in the following to restrict to this case. %**** ( BETTER TO PUT THIS AFTER SPLITTING OF $L_p$ TO AVOID CONFUSION )

%***** (  RECALL THE ${\rm Tw}_{k/2-1}$ from Corollary~\ref{cor:signed-BF} )

\begin{thm}[Reciprocity Law I]\label{thm:ERL-I}
Let $\xi\in\{\alpha,\beta\}$ and $\eta:\Delta\rightarrow\bZ_p^\times$ be a character. Then $L_{\xi,\xi}^{\rm (I)}(f,\mathbf{h}_v)^\eta$ 
satisfies the following interpolation property: For all $\chi=(1,\chi_2)\in\Xi^{\rm (I)}$ 
with $\chi_2=\langle\varepsilon\rangle^j\psi_\zeta$ for $0\leq j\leq k-2$ and $\zeta$ a primitive $p^t$-th root of unity for some $t\geq 0$,  we have 
\[
\phi_\chi\bigl(L^{\rm (I)}_{\xi,\xi}(f,\mathbf{h}_v)^\eta\bigr)= %\frac{1}{r_f}\cdot 
e_{p,\xi,\eta}(f,\chi)^2\cdot\frac{(j!)^2}{2^{2j+k+1}\cdot(-i)^{k-1}}\cdot\frac{L(f\otimes\omega^j\eta\psi_\zeta^{-1},\mathbf{h}_{v,1}^\circ,j+1)}{\mathfrak{g}(\omega^j\eta\psi_\zeta^{-1})^2\cdot\pi^{2j+2}\cdot\langle f,f\rangle_N},
\]
where  $e_{p,\xi,\eta}(f,\chi)$, $\mathfrak{g}(\omega^j\eta\psi_\zeta^{-1})$ are as in Theorem~\ref{thm:MTT}, 
%\[
%\mathcal{E}_{p,\xi,\eta}(f,\chi)=\begin{cases}\frac{1}{\xi^{2(t+1)}}\cdot\frac{p^{2(j+1)(t+1)}}{\mathfrak{g}(\omega^j\eta\psi_\zeta^{-1})^2}&\textrm{if $t\geq 1$}
%\\[0.6em]
%\Bigl(1-\frac{p^j}{\xi}\Bigr)^2\Bigl(1-\frac{p^{k-2-j}}{\xi}\Bigr)^2&\textrm{else,}\end{cases}
%\]
%with $\mathfrak{g}(\omega^j\eta\psi_\zeta^{-1})$ the Gauss sum, 
and $\langle f,f\rangle_N$ is the Petersson norm of $f$ with  
\[
\langle f_1,f_2\rangle_{N}:=\int_{\Gamma_1(N)\backslash\mathfrak{H}}f_1(\tau)\overline{f_2(\tau)}y^{k-2}dxdy
\] 
for $\tau=x+iy\in\mathfrak{H}$.
\end{thm}

\begin{proof}
This is essentially contained \cite[Thm.~7.1.5]{LZ-Coleman}; since the statement in \emph{loc.\,cit.} only includes the case $\zeta=1$, while for our purposes we will need the result for arbitrary $p$-power roots of unity $\zeta$, we provide some details.

Let $f^\xi$ denote the $p$-stabilisation of $f$ with $U_p$-eigenvalue $\xi$, and let $\mathbf{f}^\xi$ denote the Coleman family passing through $f^\xi$. Let $L_p^{\rm Urb}(\mathbf{f}^\xi,\mathbf{h}_v,1+\mathbf{j})$ be the $3$-variable $p$-adic Rankin $L$-function constructed by Urban \cite{urban-rankin} (see also \cite[App.~B]{AI-triple} and \cite[\S{7.3}]{GPRJ}). On the other hand, let
\[
L_p^{\rm geo}(\mathbf{f}^\xi,\mathbf{h}_v,1+\mathbf{j}):=(-1)^{1+\mathbf{j}}\lambda_N(\mathbf{f}^\xi)\cdot\left\langle\mathcal{L}(\mathcal{BF}_{1,1}^{[\mathbf{f}^\xi,\mathbf{h}_v]}),\eta_{\mathbf{f}^\xi}\otimes\omega_{\mathbf{h}_v}\right\rangle
\]
be the $3$-variable \emph{geometric} $p$-adic $L$-function introduced in \cite[Def.~9.1.1]{LZ-Coleman}, so by Theorem~7.1.5 in \cite{LZ-Coleman} one has the equality 
\[
L_p^{\rm geo}(\mathbf{f}^\xi,\mathbf{h}_v,1+\mathbf{j})=L_p^{\rm Urb}(\mathbf{f}^\xi,\mathbf{h}_v,1+\mathbf{j})
\]
(indeed, this follows from the Zariski-density of the crystalline points at which their corresponding specialisations agree by virtue of \cite[Thm.~6.5.9]{KLZ1}). 

By the interpolation property satisfied by $L_p^{\rm Urb}(\mathbf{f}^\xi,\mathbf{h}_v,1+\mathbf{j})$, it follows that if we let $L_p^{\rm geo}(f^\xi,\mathbf{h}_v,1+\mathbf{j})$ denote the image of $L_p^{\rm geo}(\mathbf{f}^\xi,\mathbf{h}_v,1+\mathbf{j})$ under the specialisation map at $f^\xi$, for all  $\eta:\Delta\rightarrow\bZ_p^\times$ and  $\chi\in\Xi^{\rm (I)}$ as in the statement we have
\[
\phi_{\eta\chi}\bigl(
L_p^{\rm geo}(f^\xi,\mathbf{h}_v,1+\mathbf{j})\bigr)=\frac{e_{p,\xi,\eta}(f,\chi)^2}{\mathcal{E}_{p,\xi}(f)\cdot\mathcal{E}^*_{p,\xi}(f)}\cdot\frac{(j!)^2}{2^{2j+k+1}\cdot(-i)^{k-1}}\cdot\frac{L(f\otimes\omega^j\eta\psi_\zeta^{-1},\mathbf{h}_{v,1}^\circ,j+1)}{\mathfrak{g}(\omega^j\eta\psi_\zeta^{-1})^2\cdot\pi^{2j+2}\cdot\langle f,f\rangle_N},
\]
where $\mathcal{E}_{p,\xi}(f)=\Bigl(1-\frac{p^{k-2}}{\xi^2}\Bigr)$, $\mathcal{E}_{p,\xi}^*(f)=\Bigl(1-\frac{p^{k-1}}{\xi^2}\Bigr)$. 

Comparing the interpolation properties satisfied by  $\mathscr{L}_V$ and $\mathcal{L}$ (see \cite[Thm.~4.15]{LZ2} and \cite[Thm.~7.1.4]{LZ-Coleman}), the result thus follows from the relation \eqref{eq:dual-eta} and
%
%\[
%\eta_f'=\pm\frac{1}{r_f}\cdot v_{\xi}^*,
%\]
%where $\eta_f'$ is as in \cite[Def.~6.1.3]{KLZ1} 
%and $\xi'\in\{\alpha,\beta\}$ is different from $\xi$ (cf. \eqref{eq:dual-eta}); and 
the fact that by \cite[Cor.~6.4.3]{LZ-Coleman}, $\eta_{\mathbf{f}^\gamma}$ specialises at $f^\xi$ to
$\lambda_N(f)^{-1}\mathcal{E}_{p,\xi}(f)^{-1}\mathcal{E}^*_{p,\xi}(f)^{-1}\cdot\eta_\xi$.
%
%see also \cite[Thm.~3.6.1]{BLLV}
%
%a key input is \cite{nakamura-JIMJ}
%
%Let $\mathbf{D}_{\rm rig}^\dagger$ be the functor from $p$-adic representations of $G_{\bQ_p}$ to $(\varphi,\Gamma)$-modules over $\mathscr{R}=\mathbb{B}_{{\rm rig},\bQ_p}^\dagger$, as extended in \cite{berger-colmez-familles} to $A$-representations of $G_{\bQ_p}$ and $(\varphi,\Gamma)$-modules over $\mathscr{R}_A=\mathscr{R}\hat\otimes A$ for  reduced affinoid algebras $A$ over $\bQ_p$.
%
%***** relate to $\mathscr{L}_{f^\gamma,\mathbf{h}_v}$
\end{proof}

\subsubsection{Cuspidality of Beilinson--Flach classes}

To motivate the discussion in the next two subsections, let us note that our next goal is the proof of a second explicit reciprocity law for the natural image of ${\rm res}_p(\mathcal{BF}^{\xi,\mathbf{h}_v,\eta})$ under the %two-variable 
$p$-adic regulator map $e_\eta\mathscr{L}_{V,\xi,\overline{v}}$, yielding an analogue of Theorem~\ref{thm:ERL-I} with $\overline{v}$ in place of $v$. This will be the essentially content of Theorem~\ref{thm:ERL-II} below, but a difficulty that arises is that the above evaluation only makes sense after multiplying the classes $\mathcal{BF}^{\xi,\mathbf{h}_v,\eta}$ by $T_v$. Thus we shall first consider this modification, and building on the first explicit reciprocity law (as in the proof of Theorem~\ref{thm:BSTW}) and a key computation by B\"uy\"ukboduk--Lei (as in Corollary~\ref{cor:BF-pm}) arrive at the desired Theorem~\ref{thm:ERL-II}. %for the original classes $\mathcal{BF}^{\xi,\mathbf{h}_v,\eta}$. 

We begin by noting that as in the proof of \cite[Thm.~6.22]{BSTW}, multiplication by $T_v$ annihilates the quotient $\widetilde{\mathbb{T}}_{\mathbf{h}_v}^{\square}/\mathbb{T}_{\mathbf{h}_v}^{\square}$, and so the classes
\begin{equation}\label{eq:BF-cusp}
\mathcal{BF}^{\xi,\mathbf{h}_v}_{\rm cusp}:=T_v\cdot\mathcal{BF}^{\xi,\mathbf{h}_v}
\end{equation}
land in the image of the natural map
\[
\rH^1(\bQ,V\otimes(\mathbb{T}_{\mathbf{h}_v}^{\square})\otimes\Lambda_{\scO}(\widetilde{\Gamma})^\iota)\otimes\mathcal{H}_L(\widetilde{\Gamma})\longrightarrow
\rH^1(\bQ,V\otimes(\widetilde{\mathbb{T}}_{\mathbf{h}_v}^{\square})\otimes\Lambda_{\scO}(\widetilde{\Gamma})^\iota)\otimes\mathcal{H}_L(\widetilde{\Gamma}),
\]
which is an injection. Viewing $\mathcal{BF}_{\rm cusp}^{\xi,\mathbf{h}_v}$ in the former space, and taking $\eta$-isotypic components $\mathcal{BF}_{\rm cusp}^{\xi,\mathbf{h}_v,\eta}$,  from Theorem~\ref{thm:BF-factor} (and its proof) we have 
classes  
\[
{\mathcal{BF}_{\rm cusp}^{i,\mathbf{h}_v,\eta}}\in 
\rH^1(\bZ[1/p],T(\eta)\hat\otimes({\mathbb{T}}_{\mathbf{h}_v}^{\square})\hat\otimes\Lambda_{\bZ_p}({\Gamma})^\iota)\otimes_{}L
\] 
such that
\begin{equation}\label{eq:BF-cusp-dec}
\left(\begin{array}{cc}
\mathcal{BF}_{\rm cusp}^{\alpha,\mathbf{h}_v,\eta}\\
\mathcal{BF}_{\rm cusp}^{\beta,\mathbf{h}_v,\eta}
\end{array}\right)
=Q_{\alpha,\beta}^{-1}M_{{\rm log}}\cdot
\left(\begin{array}{cc}
{\mathcal{BF}}_{\rm cusp}^{1,\mathbf{h}_v,\eta}\\
{\mathcal{BF}}_{\rm cusp}^{2,\mathbf{h}_v,\eta}
\end{array}\right),
\end{equation}
and similarly as in \eqref{eq:varpi-s} we put
\begin{equation}\label{eq:BF-cusp-int}
\mathcal{BF}_{\rm cusp,int}^{i,\mathbf{h}_v,\eta}:=\varpi^s\cdot\mathcal{BF}_{\rm cusp}^{i,\mathbf{h}_v,\eta}\in\rH^1(\bZ[1/p],T(\eta)\hat\otimes({\mathbb{T}}_{\mathbf{h}_v}^{\square})\hat\otimes\Lambda_{\bZ_p}({\Gamma})^\iota). 
\end{equation}
Since Theorem~\ref{thm:BSTW} together with Shapiro's lemma give isomorphisms
\begin{align*}
\rH^1(\bQ,T(\eta)\otimes(\mathbb{T}_{\mathbf{h}_v}^{\square})\otimes\Lambda_{\scO}(\Gamma)^\iota)
&\cong\rH^1(\bQ,T(\eta)\otimes{\rm Ind}_{K}^{\bQ}(\Lambda_{\bZ_p}(\Gamma^v)^\iota)\otimes\Lambda_{\scO}(\Gamma)^\iota)\\
&\cong
\rH^1(K,T(\eta)\otimes\Lambda_{\scO}(\Gamma_K)^\iota),
\end{align*}
for every prime $w\in\{v,\overline{v}\}$ of $K$ above $p$ we may consider the image of ${\rm res}_w(\mathcal{BF}_{\rm cusp}^{\xi,\mathbf{h}_v,\eta})$ and ${\rm res}_w(\mathcal{BF}_{\rm cusp}^{i,\mathbf{h}_v,\eta})$ under the maps \eqref{eq:def-2var} and \eqref{eq:def-2var-H}.
%(taking their $\mathcal{H}_L(\Gamma)$-linear extension to $\rH^1(K,\Lambda_{\scO}(\Gamma_K)^\iota)\otimes\mathcal{H}_L(\Gamma)$ for the latter).

\begin{defn}\label{def:2var-unb-tilde}
For $\xi,\xi'\in\{\alpha,\beta\}$ and $i,j\in\{1,2\}$ put
\begin{align*}
\widetilde{L}_{\xi,\xi'}^{\rm (I)}(f,\mathbf{h}_v)^\eta&:=e_\eta\mathscr{L}_{V,\xi,v}({\rm res}_v(\mathcal{BF}_{\rm cusp}^{\xi',\mathbf{h}_v,\eta})),&\qquad
\widetilde{L}^{\rm (I)}_{i,j}(f,\mathbf{h}_v)^\eta&:=e_\eta\widetilde{\mathscr{C}}_{i,v}({\rm res}_v(\mathcal{BF}_{\rm cusp,int}^{j,\mathbf{h}_v,\eta})),\\
\widetilde{L}_{\xi,\xi'}^{\rm (II)}(\mathbf{h}_v,f)^\eta&:=e_\eta\mathscr{L}_{V,\xi,\overline{v}}({\rm res}_{\overline{v}}(\mathcal{BF}_{\rm cusp}^{\xi',\mathbf{h}_v,\eta})),&\qquad
\widetilde{L}^{\rm (II)}_{i,j}(\mathbf{h}_v,f)^\eta&:=e_\eta\widetilde{\mathscr{C}}_{i,\overline{v}}({\rm res}_{\overline{v}}(\mathcal{BF}_{\rm cusp,int}^{j,\mathbf{h}_v,\eta})),
\end{align*}
viewed as elements in $\mathcal{H}_L(\Gamma_K)$ and $\Lambda_{\scO}(\Gamma_K)$, respectively.
%which yields an element in $\Lambda_{\scO}(\Gamma_K)$.
\end{defn}

\begin{rem}\label{rem:unb-tilde}
For $\xi=\xi'$, we have 
\[
\widetilde{L}_{\xi,\xi}^{\rm (I)}(f,\mathbf{h}_v)^\eta=T_v\cdot L^{\rm (I)}_{\xi,\xi}(f,\mathbf{h}_v)^\eta,
\]
where $L_{\xi,\xi}^{\rm (I)}(f,\mathbf{h}_v)^\eta$ is as in Definition~\ref{def:2var-unb}.
\end{rem}

In analogy with \cite{pollack}, we have the following decomposition.

\begin{cor}\label{cor:dec-2var-signed}
We have
\begin{align*}
\varpi^s\cdot\biggl(\begin{array}{cc}\widetilde{L}^{\rm (I)}_{\alpha,\alpha}(f,\mathbf{h}_v)^\eta&\widetilde{L}^{\rm (I)}_{\alpha,\beta}(f,\mathbf{h}_v)^\eta\\
\widetilde{L}^{\rm (I)}_{\beta,\alpha}(f,\mathbf{h}_v)^\eta&\widetilde{L}^{\rm (I)}_{\beta,\beta}(f,\mathbf{h}_v)^\eta
\end{array}\biggr)
&=Q_{\alpha,\beta}^{-1}M_{{\rm log},v}
\cdot\biggl(\begin{array}{cc}\widetilde{L}^{\rm (I)}_{1,1}(f,\mathbf{h}_v)^\eta&\widetilde{L}^{\rm (I)}_{1,2}(f,\mathbf{h}_v)^\eta\\
\widetilde{L}^{\rm (I)}_{2,1}(f,\mathbf{h}_v)^\eta&\widetilde{L}^{\rm (I)}_{2,2}(f,\mathbf{h}_v)^\eta
\end{array}\biggr)
\cdot(Q_{\alpha,\beta}^{-1}M_{{\rm log},{v}})^{\rm tr},
\end{align*}
and similarly
\begin{align*}
\varpi^s\cdot\biggl(\begin{array}{cc}\widetilde{L}^{\rm (II)}_{\alpha,\alpha}(\mathbf{h}_v,f)^\eta&\widetilde{L}^{\rm (II)}_{\alpha,\beta}(\mathbf{h}_v,f)^\eta\\
\widetilde{L}^{\rm (II)}_{\beta,\alpha}(\mathbf{h}_v,f)^\eta&\widetilde{L}^{\rm (II)}_{\beta,\beta}(\mathbf{h}_v,f)^\eta
\end{array}\biggr)
&=Q_{\alpha,\beta}^{-1}M_{{\rm log},\overline{v}}
\cdot\biggl(\begin{array}{cc}\widetilde{L}^{\rm (II)}_{1,1}(\mathbf{h}_v,f)^\eta&\widetilde{L}_{1,2}^{\rm (II)}(\mathbf{h}_v,f)^\eta\\
\widetilde{L}_{2,1}^{\rm (II)}(\mathbf{h}_v,f)^\eta&\widetilde{L}^{\rm (II)}_{2,2}(\mathbf{h}_v,f)^\eta
\end{array}\biggr)
\cdot(Q_{\alpha,\beta}^{-1}M_{{\rm log},v})^{\rm tr},
\end{align*}
where $M_{{\rm log},w}\in M_{2\times 2}(\mathcal{H}_L(\Gamma^w))$ %for $w\in\{v,\overline{v}\}$ 
is as in \eqref{eq:factor-L-2var-w}, and $A^{\rm tr}$ denotes the transpose of a matrix $A$.
\end{cor}

\begin{proof}
Immediate from \eqref{eq:factor-L-2var-w}, \eqref{eq:BF-cusp-dec}, and \eqref{eq:BF-cusp-int}.
\end{proof}

%\begin{rem} Similarly as in Remark~\ref{rem:cf-ap=0}, 
%the interpolation property in Theorem~\ref{thm:ERL-I} for $L_{\xi,\xi}(f,\mathbf{h}_v)^\eta$  does not directly translate into a similar interpolation property for $L_{i,i}(f,\mathbf{h}_v)^\eta$ due to the non-explicit nature of the logarithm matrix in the general nonordinary case (cf. \S\ref{subsubsec:restr-cyc} below). %(cf. Theorem~6.4 in \emph{op.\,cit.}).
%\end{rem}

\subsubsection{Cuspidality of Beilinson--Flach classes, bis}

\begin{thm}\label{thm:BSTW}
Suppose $l\in\{1,2\}$ and $\eta:\Delta\rightarrow\bZ_p^\times$ are such that 
\begin{equation}\label{eq:Lp-nonzero}
(L_{l,l}^{\rm (I)}(f,\mathbf{h}_v)^\eta\;\,{\rm mod}\,I^-)\neq 0,
%L_i(f)^\eta\cdot L_i(f\otimes\epsilon_K)^\eta\neq 0.
\end{equation}
where $I^-={\rm ker}(\Lambda_{\scO}(\Gamma_K)\twoheadrightarrow\Lambda_{\scO}(\Gamma^+)
)$. 
Then the $\eta$-component $\mathcal{BF}_{\rm int}^{l,\mathbf{h}_v,\eta}$  of \eqref{eq:varpi-s} is in the image of the natural injection
\[
\rH^1(\bZ[1/p],T(\eta)\hat\otimes(\mathbb{T}_{\mathbf{h}_v}^{\square})^{}\hat\otimes\Lambda_{\bZ_p}(\Gamma)^\iota)\longrightarrow\rH^1(\bZ[1/p],T(\eta)\hat\otimes(\widetilde{\mathbb{T}}_{\mathbf{h}_v}^{\square})^{}\hat\otimes\Lambda_{\bZ_p}(\Gamma)^\iota).
\]
%and we have 
%\[
%\mathbb{T}_{\mathbf{h}_v/{\rm tor}}\simeq{\rm Ind}_K^\bQ\,\Lambda(\Gamma_K^v)^\iota
%\] 
%as $G_{\bQ}$-modules. %with $\mathbb{T}^+%\simeq\Lambda_K^v(\Psi_K^v)$.
\end{thm}

\begin{proof}
Following the argument in the proof of \cite[Cor.~6.22]{BSTW}, the desired result follows from the nonvanishing of the composite map
\[
e_\eta\widetilde{\mathscr{C}}_{l,v}\circ{\rm res}_v:\rH^1_{{\rm rel},l}(K,\mathbb{T}_K^\eta)\longrightarrow\Lambda_{\scO}(\Gamma_K),
\]
and this follows from Theorem~\ref{thm:ERL-I} and the nonvanishing hypothesis.
\end{proof}

%For $j\in\bZ$, let ${\rm Tw}_{j}:\Lambda_{\bZ_p}(\Gamma)\rightarrow\Lambda_{\bZ_p}(\Gamma)$ be the $\bZ_p$-linear twisting isomorphism given by $\gamma\mapsto\epsilon(\gamma)^j\gamma$ for $\gamma\in\Gamma$, and also denote by ${\rm Tw}_j$ its  extension of $\mathcal{H}_{\bQ_p}(\Gamma)$. 
In view of Theorem~\ref{thm:Ind} and Shapiro's lemma, we consider the composition
\begin{equation}\label{eq:Sh-omega}
\begin{aligned}
\rH^1(\bZ[1/p],T(\eta)\hat\otimes(\mathbb{T}_{\mathbf{h}_v}^{\square})\hat\otimes\Lambda_{\bZ_p}(\Gamma)^\iota)
%&\simeq\rH^1(\cO_K[1/p],T_f^*\hat\otimes\mathbb{T}^+\hat\otimes\Lambda)\\
&\overset{\simeq}\longrightarrow\rH^1(\cO_K[1/p],T(\eta)\hat\otimes\Lambda_{\bZ_p}(\Gamma^v)^\iota\hat\otimes\Lambda_{\bZ_p}(\Gamma)^\iota)\\
&\overset{\simeq}\longrightarrow\rH^1(\cO_K[1/p],T(\eta)\hat\otimes\Lambda_{\bZ_p}(\Gamma_K)^\iota).
%&\overset{{{\rm Tw}_{1-k/2}}}\longrightarrow\rH^1(\cO_K[1/p],T(1-k/2)(\eta\omega^{k/2-1})\hat\otimes\Lambda_{\bZ_p}(\Gamma_K)^\iota),
\end{aligned}
\end{equation}
%where ${\rm Tw}_{1-k/2}:\Lambda_{\bZ_p}(\Gamma)\rightarrow\Lambda_{\bZ_p}(\Gamma)$ is the $\bZ_p$-linear twisting isomorphism given by $\gamma\mapsto\langle\varepsilon(\gamma)\rangle^{1-k/2}\gamma$ for $\gamma\in\Gamma$.
%
%the first arrow is the isomorphism arising from Shapiro's lemma, the second arising from the inverse of the  isomorphism 
%\[
%\theta:\Gamma_K\simeq\Gamma^v\times\Gamma\xrightarrow{\sim}\Gamma^v\times\Gamma
%\]
%given by $\gamma_v\mapsto\gamma_v^{-1}$ for $\gamma_v\in\Gamma^v$ and the identity on $\Gamma$. 
%and the last arrow is the twist by $\epsilon^{1-k/2}$.

Abusing notation, for $l\in\{1,2\}$ and $\eta:\Delta\rightarrow\bZ_p^\times$ such that \eqref{eq:Lp-nonzero} holds, we shall still denote by 
\[
\mathcal{BF}^{l,\mathbf{h}_v,\eta}_{\rm int}\in\rH^1(\cO_K[1/p],T(\eta)\hat\otimes\Lambda_{\bZ_p}(\Gamma_K)^\iota)
\]
the image of the $\eta$-component of \eqref{eq:varpi-s} for $i=l$ under the composition \eqref{eq:Sh-omega}.

%\begin{defn}\label{def:Z}
%Let $i\in\{1,2\}$ and $\eta:\Delta\rightarrow\bZ_p^\times$ be such that \eqref{eq:Lp-nonzero} holds. We denote by
%\[
%\mathcal{BF}^{i,\eta}(f_{/K})\in
%\rH^1(\cO_K[1/p],T(\eta)\hat\otimes\Lambda_{\bZ_p}%(\Gamma_K)^\iota)
%%\rH^1(\cO_K[1/p],T(1-k/2)(\eta\omega^{k/2-1})\hat\otimes\Lambda_{\bZ_p}(\Gamma_K)^\iota)
%\]
%the image of $\mathcal{BF}_{\rm int}^{i,\mathbf{h}_v,\eta}$ under the composition \eqref{eq:Sh-omega}.
%\end{defn}

%The splitting of Theorem~\ref{thm:BF-factor} thus takes the following form, which we record for our later use.

%\begin{cor}\label{cor:signed-BF}
%For $i\in\{1,2\}$ for which $\mathcal{BF}^i(f_{/K})$ is defined, we have
%\begin{equation}
%\left(\begin{array}{cc}
%\varpi^s\cdot\mathcal{BF}_{\alpha,\mathbf{h}_v}\\
%\varpi^s\cdot\mathcal{BF}_{\beta,\mathbf{h}_v}
%\end{array}\right)
%=Q_{\alpha,\beta}^{-1}\cdot{\rm Tw}_{k/2-1}(M_{\rm log})\cdot
%\left(\begin{array}{cc}
%{\mathcal{BF}}_1\\
%{\mathcal{BF}}_2
%\end{array}\right),\nonumber
%\end{equation}
%where ${\rm Tw}_{k/2-1}(M_{\rm log})\in M_{2\times 2}(\mathcal{H}(\Gamma))$ is the result of applying ${\rm Tw}_{k/2-1}$ to $M_{\rm log}$ coordinate-wise.
%\end{cor}

\begin{rem}
In our application to Kato's main conjecture, $K$ will be chosen so that 
\begin{equation}\label{eq:nonzero-L-value}
L(f\otimes\epsilon_K\omega^{k/2-1},1)\neq 0.\nonumber
\end{equation}
For such $K$, by virtue of \cite[Prop.~2]{shimuraCPAM} and \cite[Prop.~3.28]{LLZ} the nonvanishing condition \eqref{eq:Lp-nonzero} for $\eta=\omega^{1-k/2}$ is satisfied for all $l\in\{1,2\}$ unless $k=2$ and $a_p\neq 0$, in which case Proposition~\ref{prop:L-nonzero} shows that the condition is satisfied for at least some $l\in\{1,2\}$, and this will suffice for our purposes.
\end{rem}

%\begin{rem}
%For $f$ of weight $2$, the class $\mathcal{BF}^{i_0}$ is the  \emph{zeta element} attached to $f$  and $K$ in %the terminology of 
%\cite{BSTW} in the supersingular case with $a_p=0$  (denoted $\mathcal{Z}^\circ(f_{/K})$ for $\circ=\{-,+\}$, and corresponding to $i_0=\{1,2\}$ in our notation). In that case, thanks to the interpolation property of the $p$-adic $L$-functions $\mathscr{L}_p^\circ(g)$, one can show that \eqref{eq:Lp-nonzero} holds for both choices of $i_0$.
%\end{rem}

%Following \cite{BSTW}, we shall refer to $\mathcal{BF}^\circ(f_{/K})$ as the \emph{zeta elements} associated to  $f_{/K}$.

\subsubsection{Explicit reciprocity law for $w=\overline{v}$}\label{subsec:ERL-vbar}

\begin{prop}\label{prop:BF-geo}
For every $\xi\in\{\alpha,\beta\}$ and $\eta:\Delta\rightarrow\bZ_p^\times$, we have
\[
%e_\eta\mathscr{L}_{V,\xi,\overline{v}}({\rm res}_{\overline{v}}(\mathcal{BF}^{\xi,\mathbf{h}_v,\eta}_{\rm cusp}))=0.
\widetilde{L}^{(\rm II)}_{\xi,\xi}(\mathbf{h}_v,f)^\eta=0.
\]
Moreover, if $\xi'\in\{\alpha,\beta\}$ is different from $\xi$, then
\[
%e_\eta\mathscr{L}_{V,\xi,\overline{v}}({\rm res}_{\overline{v}}(\mathcal{BF}_{\rm cusp}^{\xi',\mathbf{h}_v,\eta}))=-e_\eta\mathscr{L}_{V,\xi',\overline{v}}({\rm res}_{\overline{v}}(\mathcal{BF}_{\rm cusp}^{\xi,\mathbf{h}_v,\eta})),
\widetilde{L}^{(\rm II)}_{\xi,\xi'}(\mathbf{h}_v,f)^\eta=-\widetilde{L}^{(\rm II)}_{\xi',\xi}(\mathbf{h}_v,f)^\eta.
\]
\end{prop}

\begin{proof}
By Theorem~\ref{thm:BSTW} we have $G_K$-module isomorphisms 
$\mathbb{T}_{\mathbf{h}_v}^+\cong\Lambda_{\bZ_p}(\Gamma^v)^\iota$ and $\mathbb{T}_{\mathbf{h}_v}^-\cong c\cdot\mathbb{T}_{\mathbf{h}_v}^+$. Thus    
%the isomorphism
%\[
%\rH^1(\bQ,V\otimes\mathbb{T}_{\mathbf{h}_v}^{\square,[k-2]}\otimes\Lambda_{\scO}(\Gamma)^\iota)\otimes\mathcal{H}_L(\Gamma)\cong\rH^1_{\rm Iw}(\bQ_p(\mu_{p^\infty}),\mathbb{D}_{\rm rig}^\dagger(V\otimes\mathbb{T}_{\mathbf{h}_v}^{\square,[k-2]}))
%\]
%from 
by Shapiro's lemma, ${\rm res}_{\overline{v}}(\mathcal{BF}_{\rm cusp}^{\xi,\mathbf{h}_v,\eta})$ is identified with the image of ${\rm res}_p(\mathcal{BF}_{\rm cusp}^{\xi,\mathbf{h}_v,\eta})$ under the natural map
\begin{equation}\label{eq:proj-F-}
\rH^1_{\rm Iw}(\bQ_p(\mu_{p^\infty}),\mathbb{D}_{\rm rig}^\dagger(V\otimes\mathbb{T}_{\mathbf{h}_v}^{\square}))^\eta\longrightarrow\rH^1_{\rm Iw}(\bQ_p(\mu_{p^\infty}),\mathscr{F}_{\xi,{\rm ord}}^{\emptyset,-}\mathbb{D}_{\rm rig}^\dagger(V\otimes\mathbb{T}_{\mathbf{h}_v}^{\square}))^\eta.
\end{equation}

On the other hand, by \cite[Thm.~7.1.2]{LZ-Coleman} the image of ${\rm res}_p(\mathcal{BF}_{\rm cusp}^{\xi,\mathbf{h}_v,\eta})$ under \eqref{eq:proj-F-} is contained in the image of 
\begin{equation}\label{eq:image-F+}
\rH^1_{\rm Iw}(\bQ_p(\mu_{p^\infty}),\mathscr{F}_{\xi,{\rm ord}}^{+,-}\mathbb{D}_{\rm rig}^\dagger(V\otimes{\mathbb{T}}_{\mathbf{h}_v}^{\square}))^\eta\longrightarrow
\rH^1_{\rm Iw}(\bQ_p(\mu_{p^\infty}),\mathscr{F}_{\xi,{\rm ord}}^{\emptyset,-}\mathbb{D}_{\rm rig}^\dagger(V\otimes{\mathbb{T}}_{\mathbf{h}_v}^{\square}))^\eta
\end{equation}
Since under the above identifications the regulator map $e_\eta\mathscr{L}_{V,\xi,\overline{v}}$ factors through 
\[
\rH^1_{\rm Iw}(\bQ_p(\mu_{p^\infty}),\mathscr{F}_{\xi,{\rm ord}}^{\emptyset,-}\mathbb{D}_{\rm rig}^\dagger(V\otimes{\mathbb{T}}_{\mathbf{h}_v}^{\square}))^\eta
\longrightarrow\rH^1_{\rm Iw}(\bQ_p(\mu_{p^\infty}),\mathscr{F}_{\xi,{\rm ord}}^{-,-}\mathbb{D}_{\rm rig}^\dagger(V\otimes{\mathbb{T}}_{\mathbf{h}_v}^{\square}))^\eta,
\]
whose composition with \eqref{eq:image-F+} is zero, the proof of the first claim follows. The second claim is shown in  \cite[Prop.~3.14]{BL-non-ord}. 
\end{proof}

\begin{cor}\label{cor:BF-pm}
For every $i\in\{1,2\}$ and $\eta:\Delta\rightarrow\bZ_p^\times$, we have
\[
%e_\eta\mathscr{C}_{i,\overline{v}}({\rm res}_{\overline{v}}(\mathcal{BF}_{\rm cusp,int}^{i,\mathbf{h}_v,\eta}))=0.
\widetilde{L}^{(\rm II)}_{i,i}(\mathbf{h}_v,f)^\eta=0.
\]
%hence the inclusion $\mathcal{BF}^i_K\in\rH^1_{{\rm rel},i}(K,T_f^*\hat\otimes\Lambda_K)$. Moreover,
Moreover, writing  
%$Q_{\alpha,\beta}^{-1}{\rm Tw}_{k/2-1}M_{{\rm log}}=\bigl(\begin{smallmatrix}a&b\\c&b\end{smallmatrix}\bigr)$.
$Q_{\alpha,\beta}^{-1}M_{{\rm log}}=\bigl(\begin{smallmatrix}\mathfrak{A}&\mathfrak{B}\\\mathfrak{C}&\mathfrak{D}\end{smallmatrix}\bigr)$ we have
\begin{align*}
\widetilde{L}^{(\rm II)}_{1,2}(\mathbf{h}_v,f)^\eta\cdot(\mathfrak{AD-BC})
&=-
\widetilde{L}^{(\rm II)}_{2,1}(\mathbf{h}_v,f)^\eta\cdot(\mathfrak{AD-BC})\\
&=-\widetilde{L}^{(\rm II)}_{\beta,\alpha}(\mathbf{h}_v,f)^\eta\cdot\varpi^s.
%e_\eta\mathscr{C}_{1,\overline{v}}({\rm res}_{\overline{v}}(\mathcal{BF}^{2,\mathbf{h}_v,\eta}_{\rm cusp,int}))(\mathfrak{AD-BC})&=-e_\eta\mathscr{C}_{2,\overline{v}}({\rm res}_{\overline{v}}(\mathcal{BF}^{1,\mathbf{h}_v,\eta}_{\rm cusp,int}))(\mathfrak{AD-BC})\\
%&=\varpi^{s}\cdot e_\eta\mathscr{L}_{V,\beta,\overline{v}}({\rm res}_{\overline{v}}(\mathcal{BF}_{\rm cusp}^{\alpha,\mathbf{h}_v,\eta})).
\end{align*}
%where we write 
%%$Q_{\alpha,\beta}^{-1}{\rm Tw}_{k/2-1}M_{{\rm log}}=\bigl(\begin{smallmatrix}a&b\\c&b\end{smallmatrix}\bigr)$.
%$Q_{\alpha,\beta}^{-1}M_{{\rm log}}=\bigl(\begin{smallmatrix}\mathfrak{A}&\mathfrak{B}\\\mathfrak{C}&\mathfrak{D}\end{smallmatrix}\bigr)$.
%$j$ is the element in $\{1,2\}$ different from $i$.
\end{cor}

\begin{proof}
%This is \cite[Cor.~3.15]{BL-non-ord}: 
%and follows immediately from  Corollary~\ref{cor:BF-pm} and Proposition~\ref{prop:BF-geo}. Indeed, 
%writing $Q_{\alpha,\beta}^{-1}{\rm Tw}_{k/2-1}M_{{\rm log}}=\bigl(\begin{smallmatrix}a&b\\c&b\end{smallmatrix}\bigr)$ and 
Multiplying the second equality of Corollary~\ref{cor:dec-2var-signed} by $\bigl(\begin{smallmatrix}\mathfrak{D}&-\mathfrak{B}\\-\mathfrak{C}&\mathfrak{A}\end{smallmatrix}\bigr)$ on the left and $\bigl(\begin{smallmatrix}\mathfrak{D}&-\mathfrak{C}\\-\mathfrak{B}&\mathfrak{A}\end{smallmatrix}\bigr)$ on the right, and using Proposition~\ref{prop:BF-geo} yields
\begin{align*}
(\mathfrak{AD-BC})^2&\biggl(\begin{array}{cc}\widetilde{L}^{(\rm II)}_{1,1}(\mathbf{h}_v,f)^\eta&\widetilde{L}^{(\rm II)}_{1,2}(\mathbf{h}_v,f)^\eta\\
\widetilde{L}^{(\rm II)}_{2,1}(\mathbf{h}_v,f)^\eta&\widetilde{L}^{(\rm II)}_{2,2}(\mathbf{h}_v,f)^\eta
\end{array}\biggr)\\
&\quad=\biggl(\begin{array}{cc}0&-(\mathfrak{AD-BC})\cdot\widetilde{L}^{(\rm II)}_{\beta,\alpha}(\mathbf{h}_v,f)^\eta\cdot\varpi^s\\
(\mathfrak{AD-BC})\cdot\widetilde{L}^{(\rm II)}_{\beta,\alpha}(\mathbf{h}_v,f)^\eta\cdot\varpi^s&0
\end{array}\biggr),
\end{align*}
whence the result (cf. \cite[Cor.~3.15]{BL-non-ord}).
\end{proof}

%\begin{defn}%[$\overline{v}$-Rankin--Eisenstein $p$-adic $L$-function]

\begin{cor}\label{cor:div-Tv}
Suppose $l\in\{1,2\}$ and $\eta:\Delta\rightarrow\bZ_p^\times$ are such that \eqref{eq:Lp-nonzero} holds. 
Then $\widetilde{L}_{1,2}^{\rm (II)}(\mathbf{h}_v,f)^\eta=-\widetilde{L}_{2,1}^{\rm (II)}(\mathbf{h}_v,f)^\eta$ are divisible by $T_v$.
\end{cor}

\begin{proof}
Note that the stated equality follows from Corollary~\ref{cor:BF-pm}. By Theorem~\ref{thm:BSTW}, under the stated assumption the class $\mathcal{BF}_{\rm int}^{l,\mathbf{h}_v,\eta}$ is naturally contained in  $\rH^1(\cO_K[1/p],T(\eta)\hat\otimes\Lambda_{\bZ_p}(\Gamma_K)^\iota)$, and so $e_\eta\widetilde{\mathscr{C}}_{l',\overline{v}}$, for any $l'\in\{1,2\}$, may be evaluated at ${\rm res}_{\overline{v}}(\mathcal{BF}_{\rm int}^{l,\mathbf{h}_v,\eta})$. Since by definition (see \eqref{eq:BF-cusp} and \eqref{eq:BF-cusp-dec}) we have
\[
T_v\cdot e_\eta\widetilde{\mathscr{C}}_{l'
,\overline{v}}({\rm res}_{\overline{v}}(\mathcal{BF}_{\rm int}^{l,\mathbf{h}_v,\eta}))=e_\eta\widetilde{\mathscr{C}}_{l',\overline{v}}({\rm res}_{\overline{v}}(\mathcal{BF}_{{\rm cusp},{\rm int}}^{l,\mathbf{h}_v,\eta}))=\widetilde{L}_{l',l}^{\rm (II)}(\mathbf{h}_v,f)^\eta,
\]
together with Corollary~\ref{cor:BF-pm} this yields the result.
\end{proof}

%This motivates the following.

\begin{defn}\label{def:L-12}
In light of Corollary~\ref{cor:div-Tv}, for every character $\eta:\Delta\rightarrow\bZ_p^\times$ we put
\[
L_{1,2}^{\rm (II)}(\mathbf{h}_v,f)^\eta:=\frac{\widetilde{L}_{1,2}^{\rm (II)}(\mathbf{h}_v,f)^\eta}{T_v}\in\Lambda_{\scO}(\Gamma_K).
\]
%and note that this is an element in $\Lambda_{\scO}(\Gamma_K)\otimes_{\scO}L$.
\end{defn}

%For $i,j\in\{1,2\}$ with $i\neq j$, put
%\[
%\widetilde{L}_{i,j}^{}(\mathbf{h}_v,f)^\eta:=\varpi^{-s}\cdot\frac{[\varphi(\omega_f),\omega_f]_{\rm dR}}{\delta_{k-1}}\cdot e_\eta\widetilde{\mathscr{C}}_{i,\overline{v}}({\rm res}_{\overline{v}}(\mathcal{BF}_{\rm cusp, int}^{j,\mathbf{h}_v,\eta})).
%\]
%\end{defn}

\begin{rem}\label{rem:L12-L}
As shown in \cite[Cor.~3.2]{LLZ-ANT}, the determinant of $M_{\rm log}$ is given up to a $p$-adic unit by 
\[
{\rm det}(M_{\rm log})\,\sim_p\,\frac{{\rm log}_{p,k-1}}{\delta_{k-1}}\sim_p\,\frac{p^{k-1}\prod_{j=1}^{k-1}\ell_{-j}}{\delta_{k-1}},
\]
where $\ell_{-j}=\frac{{\rm log}_p(\gamma)}{{\rm log}_p\varepsilon(\gamma)}+j$ for $\gamma\in\Gamma$ any topological generator. 
%
%or equivalently, $p^{k-1}\bigl(\prod_{j=1}^{k-1}\ell_{-j}\bigr)/\delta_{k-1}$ up to a unit. 
Hence  Corollary~\ref{cor:BF-pm} shows that
\[
\widetilde{L}_{i,j}^{\rm (II)}(\mathbf{h}_v,f)^\eta\,\sim_p\,\frac{e_\eta\mathscr{L}_{V,\beta,\overline{v}}({\rm res}_{\overline{v}}(\mathcal{BF}_{\rm cusp}^{\alpha,\mathbf{h}_v,\eta}))}{(\alpha-\beta)\cdot\prod_{j=1}^{k-1}\ell_{-j}}\cdot\delta_{k-1}\cdot\varpi^s.
\]
%up to a unit.
\end{rem}

To describe the interpolation property of $L_{1,2}^{}(\mathbf{h}_v,f)^\eta$, put
\[
\Xi^{\rm (II)}:=\biggl\{(\chi_1,\chi_2):\Gamma^v\hat\otimes\Gamma\rightarrow\overline{\bQ}_p^\times\;\biggl|\;\begin{aligned}\chi_1(\gamma_v^{h_p})&=\langle\varepsilon(\gamma)\rangle^m,\;\,m\geq 1,\;\,m\equiv 0\;({\rm mod}\;{p-1})\\\chi_2&=\psi_\zeta\langle\varepsilon\rangle^n,\;\,0\leq n< m,\;\,\zeta\in\mu_{p^\infty}\end{aligned}\biggr\}.
\]
%where $\gamma_v\in\Gamma^v$ and $\gamma\in\Gamma$ are topological generators. 

\begin{thm}[Explicit Reciprocity Law II]\label{thm:ERL-II}
%The measures
%\[
%L_p^{i,j}(\mathbf{h}_v,f):={\rm Col}_{i,\overline{v}}({\rm res}_{\overline{v}}(\mathcal{BF}^{j,\mathbf{h}_v}))\in\Lambda(\Gamma_K)
%\]
%satisfies the interpolation property: 
For all $\chi=(\chi_1,\chi_2)\in\Xi^{\rm (II)}$ with $\chi_2=\psi_\zeta\langle\varepsilon\rangle^n$ for a primitive $p^t$-th root of unity $\zeta$ for some $t\geq 0$, we have
\begin{align*}
\phi_\chi\bigl(
L_{1,2}^{\rm (II)}(\mathbf{h}_v,f)^\eta\bigr)&=\frac{\mathcal{E}_{p,\eta}(\mathbf{h}_v,f,\chi)}{\mathcal{E}_p(\mathbf{h}_v)\cdot\mathcal{E}_p^*(\mathbf{h}_v)}\cdot\frac{n!(n+1-k)!}{\pi^{2n+3-k}\cdot(-i)^{m+1-k}\cdot 2^{2n+m+3-k}}\cdot\phi_\chi\biggl(\frac{c\cdot\omega_{\mathbf{h}_v}}{\eta_{\mathbf{h}_v}}\biggr)\cdot\phi_\chi(\delta_{k-1})\\
&\quad\times\frac{L(\mathbf{h}_{v,\chi_1}^\circ,f\otimes\eta\omega^n\psi_\zeta^{-1},n+1)}{\langle\mathbf{h}_{v,\chi_1}^\circ,\mathbf{h}_{v,\chi_1}^\circ\rangle_{D_K}}\cdot\frac{\varpi^s}{[\varphi(\omega_f),\omega_f]_{\rm dR}},
\end{align*}
where 
\begin{itemize}
\item{}
$\mathcal{E}_{p,\eta}(\mathbf{h}_v,f,\chi)=\begin{cases}\Bigl(1-\frac{p^n}{\chi_1(\overline{v})\alpha}\Bigr)\Bigl(1-\frac{p^n}{\chi_1(\overline{v})\beta}\Bigr)\Bigl(1-\frac{p^{m}\alpha}{p^{n+1}\chi_1(\overline{v})}\Bigr)\Bigl(1-\frac{p^{m}\beta}{p^{n+1}\chi_1(\overline{v})}\Bigr),%&\textrm{if $\eta=\omega^{-n}$ ant $t=0$},
\\[0.6em]
\Bigl(\frac{p^{t+1}}{\mathfrak{g}(\eta\omega^n\psi_\zeta^{-1})}\Bigr)^2\Bigl(\frac{p^{2n+1-k}}{\chi_1(\overline{v})^2}\Bigr)^{t+1}, 
%&\textrm{else},
\end{cases}$
\end{itemize}
according to whether or not $\eta=\omega^{-n}$ and $t=0$,
\begin{itemize}
\item
$\mathcal{E}_p(\mathbf{h}_v)=\Bigl(1-\frac{p^{m-1}}{\chi_1(\overline{v})^2}\Bigr)$,
\item
$\mathcal{E}_p^*(\mathbf{h}_v)=\Bigl(1-\frac{p^m}{\chi_1(\overline{v})^2}\Bigr)$,
\end{itemize}
and $\mathbf{h}_{v,\chi_1}^\circ$ is the newform of level $D_K$ and weight $m+1$ with ordinary $p$-stabilization $\chi_1(\mathbf{h}_v)$. 
%Moreover,
%\[
%L_{2,1}^{\mathbf{h}_v}(\mathbf{h}_v,f)=-L_{1,2}^{\mathbf{h}_v}(\mathbf{h}_v,f)
%\]
%and $L_{1,1}^{\mathbf{h}_v}(\mathbf{h}_v,f)=L_{2,2}^{\mathbf{h}_v}(\mathbf{h}_v,f)=0$.
\end{thm}

\begin{proof}
With the same notations as in the proof of Theorem~\ref{thm:ERL-I}, but reversing the roles of $\mathbf{f}^\beta$ and $\mathbf{h}_v$, let $L_p^{\rm Urb}(\mathbf{h}_v,\mathbf{f}^\beta,1+\mathbf{j})$ be Urban's $3$-variable $p$-adic Rankin $L$-function, and put
\[
L_p^{\rm geo}(\mathbf{h}_v,\mathbf{f}^\beta,1+\mathbf{j}):=(-1)^{1+\mathbf{j}}\lambda_N(\mathbf{h}_v)\cdot\left\langle\mathcal{L}(\mathcal{BF}_{1,1}^{[\mathbf{f}^\beta,\mathbf{h}_v]}),\omega_{\mathbf{f}^\xi}\otimes\eta_{\mathbf{h}_v}\right\rangle
\]
for the geometric $p$-adic $L$-function of Loeffler--Zerbes \cite{LZ-Coleman}; so Theorem~7.1.5 in \emph{op.\,cit.} gives
\begin{equation}\label{eq:LZ-ERL2}
L_p^{\rm geo}(\mathbf{h}_v,\mathbf{f}^\beta,1+\mathbf{j})=L_p^{\rm Urb}(\mathbf{h}_v,\mathbf{f}^\beta,\mathbf{h}_v,1+\mathbf{j}).\nonumber
\end{equation}
Letting $L_p^{\rm geo}(\mathbf{h}_v,f^\beta,1+\mathbf{j})$ be the image of $L_p^{\rm geo}(\mathbf{h}_v,\mathbf{f}^\beta,1+\mathbf{j})$ under the specialisation map at $f^\beta$, by %\eqref{eq:LZ-ERL2} and
the interpolation property of $L_p^{\rm Urb}(\mathbf{h}_v,\mathbf{f}^\beta,\mathbf{h}_v,1+\mathbf{j})$,   it follows that for all characters $\eta:\Delta\rightarrow\bZ_p^\times$ and $\chi=(\chi_1,\chi_2)\in\Xi^{\rm (II)}$ %with $\chi_2(\gamma_{\rm cyc})=\zeta\epsilon(\gamma_{\rm cyc})^n$ for a primitive $p^t$-th root of unity $\zeta$
as in the statement, we have
\begin{equation}\label{eq:interp-geo}
\begin{aligned}
\phi_{\eta\chi}\bigl(
L_{p}^{\rm geo}(\mathbf{h}_v,f^\beta,1+\mathbf{j})\bigr)&=\frac{\mathcal{E}_{p,\eta}(\mathbf{h}_v,f,\chi)}{\mathcal{E}_p(\mathbf{h}_v)\cdot\mathcal{E}_p^*(\mathbf{h}_v)}\cdot\frac{n!(n+1-k)!}{\pi^{2n+3-k}\cdot(-i)^{m+1-k}\cdot 2^{2n+m+3-k}}\\
&\quad\times\frac{L(\mathbf{h}_{v,\chi_1}^\circ,f\otimes\eta\omega^n\psi_\zeta^{-1},n+1)}{\langle\mathbf{h}_{v,\chi_1}^\circ,\mathbf{h}_{v,\chi_1}^\circ\rangle_{D_K}}.
\end{aligned}
\end{equation}
%where 
%\begin{itemize}
%\item{}
%$\mathcal{E}_p(\mathbf{h}_v,f,\chi)=\begin{cases}\Bigl(1-\frac{p^n}{\chi_1(\overline{v})\alpha}\Bigr)\Bigl(1-\frac{p^n}{\chi_1(\overline{v})\beta}\Bigr)\Bigl(1-\frac{p^{m}\alpha}{p^{n+1}\chi_1(\overline{v})}\Bigr)\Bigl(1-\frac{p^{m}\beta}{p^{n+1}\chi_1(\overline{v})}\Bigr),\\[0.6em]
%\Bigl(\frac{p^{t+1}}{\mathfrak{g}(\psi_\zeta^{-1}\omega^n)}\Bigr)^2\Bigl(\frac{p^{2n+1-k}}{\chi_1(\overline{v})^2}\Bigr)^{t+1},
%\end{cases}$
%\end{itemize}
%according to whether $\zeta=1$ and $n\equiv 0\;({\rm mod}\,p-1)$ or otherwise,
%\begin{itemize}
%\item
%$\mathcal{E}_p(\mathbf{h}_v)=\Bigl(1-\frac{p^{m-1}}{\chi_1(\overline{v})^2}\Bigr)$,
%\item
%$\mathcal{E}_p^*(\mathbf{h}_v)=\Bigl(1-\frac{p^m}%{\chi_1(\overline{v})^2}\Bigr)$,
%\end{itemize}
%and $\mathbf{h}_{v,\chi_1}^\circ$ is the newform of level $D_K$ and weight $m+1$ with ordinary $p$-stabilization $\chi_1(\mathbf{h}_v)$.
%
%(this should be with a shift by $k-1$ wrt to what we want; and the correct range after dividing by $\ell_{1-k}\circ\cdots\circ\ell_{-2}\circ\ell_{-1}$)
%*****
Let $\omega_{f^\beta}$ denote the image of $\omega_{\mathbf{f}^\beta}$ under the specialisation map at $f^\beta$. Comparing the interpolation formulas of the maps $\mathscr{L}_V$ and $\mathcal{L}$ (as given in \cite[Thm.~4.15]{LZ2} and \cite[Thm.~7.1.4]{LZ-Coleman}, respectively), and noting the relation up to a $p$-adic unit
\[
\omega_{f^\beta}=\frac{[\varphi(\omega_f),\omega_f]_{\rm dR}}{\alpha-\beta}\cdot v_\beta^*,
\]
as follows from a straightforward computation (and also noted in the proof of \cite[Thm.~3.6.5]{BLLV}), 
from Remark~\ref{rem:L12-L} we thus find that
\begin{align*}
\phi_\chi(L^{\rm (II)}_{1,2}(\mathbf{h}_v,f)^\eta)&=
\phi_\chi\biggl(\frac{e_\eta\mathscr{L}_{V,\beta,\overline{v}}(\mathcal{BF}^{\alpha,\mathbf{h}_v,\eta}_{\rm cusp})}{T_v\cdot\prod_{j=1}^{k-1}\ell_{-j}}\biggr)\cdot\phi_\chi(\delta_{k-1})\cdot\frac{\varpi^s}{(\alpha-\beta)}\\
&=\phi_{\eta\chi}\biggl(\frac{(-1)^{1+\mathbf{j}}}{\lambda_N(\mathbf{h}_v)}\cdot L_p^{\rm geo}(\mathbf{h}_v,f^\beta,1+\mathbf{j})\cdot\frac{c\cdot\omega_{\mathbf{h}_v}}{\eta_{\mathbf{h}_v}}\biggr)\cdot\phi_\chi(\delta_{k-1})\cdot\frac{\varpi^s}{[\varphi(\omega_f),\omega_f]_{\rm dR}}
\end{align*}
using that $\phi_\chi(\mathcal{BF}_{\rm cusp}^{\alpha,\mathbf{h}_v,\eta})=\phi_\chi(T_v)\cdot\phi_\chi(\mathcal{BF}^{\alpha,\mathbf{h}_v,\eta})$ by definition, with $\phi_\chi(T_v)\neq 0$, for the last equality. Together with \eqref{eq:interp-geo}, this gives the result.
\end{proof}

We conclude this section with the following observation on the nonvanishing on $L_{i,j}^{\rm (II)}(\mathbf{h}_v,f)^\eta$, which lies much less deep than that of $L_{i,i}^{\rm (I)}(f,\mathbf{h}_v)^\eta$.

\begin{cor}\label{cor:h-unb-nonzero}
For $i\neq j\in\{1,2\}$ and $\eta:\Delta\rightarrow\bZ_p^\times$ a character, %the $p$-adic $L$-function 
$L_{i,j}^{\rm (II)}(\mathbf{h}_v,f)^\eta$ is nonzero.
\end{cor}

\begin{proof}
One just needs to note that the range of $p$-adic interpolation in Theorem~\ref{thm:ERL-II} contains values  $L(\mathbf{h}_{v,\chi_1}^\circ,f\otimes\eta\omega^n\psi_\zeta^{-1},n+1)$ for which the defining Euler product is absolutely convergent (cf. \cite[Prop.~4.7]{BL-non-ord}). 
\end{proof}

\section{Descent}\label{sec:descent}

In this section we relate the two-variable $p$-adic $L$-functions $L^{\rm (II)}_{i,j}(\mathbf{h}_v,f)^\eta$ for $i\neq j$ (resp. $L^{\rm (I)}_{i,i}(f,\mathbf{h}_v)^\eta$) to certain anticyclotomic (resp. cyclotomic and anticyclotomic) $p$-adic $L$-functions constructed independently. 

%For the latter two-variable $p$-adic $L$-functions, it will be convenient to normalise it as follows: 
%
%\begin{defn}\label{def:int-Lii}
%For $\xi\in\{\alpha,\beta\}$, $i\in\{1,2\}$, and a character $\eta:\Delta\rightarrow\bZ_p^\times$, put
%\[
%\mathcal{L}^{\rm (I)}_{\xi,\xi}(f/K)^\eta:=c_f\cdot %L_{\xi,\xi}^{\rm (I)}(f,\mathbf{h}_v)^\eta,\quad\quad
%\mathcal{L}_{i,i}^{\rm (I)}(f/K)^\eta:=c_f\cdot L^{\rm (I)}_{i,i}(f,\mathbf{h}_v)^\eta,
%\]
%where $c_f$ is the congruence number of $f$.
%\end{defn}

\subsection{Anticyclotomic descent: Indefinite}\label{subsubsec:restr-ac}

%\subsubsection{In:}
Let $\pi_f=\otimes_{\ell\leq\infty}\pi_{f,\ell}$ be the cuspidal automorphic representation of ${\rm GL}_2(\mathbb{A})$ generated by $f$, and let $BC_K(\pi_f)$ denote its base change to %an automorphic representation of 
${\rm GL}_2(\mathbb{A}_K)$. %Following the conventions in \cite[\S{4.1}]{JSW}, 
Write
\[
N=N^+ N^-
\]
with $N^+$ (resp. $N^-$) divisible only by primes that are split or ramified (resp. inert) in $K$. Throughout this section, the we assume the following \emph{generalised Heegner hypotheses}:
\begin{equation}\label{eq:gen-H}
%(N,D_K)=1\quad\textrm{and}\quad
N^-=1\quad\textrm{and}\quad\textrm{$\varepsilon_v(BC_K(\pi_f))=+1$ for all $v\mid N^+$ nonsplit in $K$,}\tag{gen-H}
\end{equation}
where $\varepsilon_v(BC_K(\pi_f))$ is the local root number at $v$ (normalised as in \cite[p.\,710]{hsieh}).

The action of complex conjugation yields an eigenspace decomposition
\[
\Gamma_K\cong\Gamma^+\times\Gamma^-,
\]
where $\Gamma^+$ (resp. $\Gamma^-$) is identified with the Galois group of the cyclotomic $\bZ_p$-extension $K_\infty^+/K$ (resp. the anticyclotomic $\bZ_p$-extension $K_\infty^-/K$). 

The following is a refinement and generalisation of the anticyclotomic $p$-adic $L$-function introduced by Bertolini--Darmon--Prasanna \cite{bdp1}.

\begin{thm}\label{thm:BDP-Lp}
There exists a unique ``square-root'' $p$-adic $L$-function 
\[
\mathscr{L}_{\overline{v}}^{\rm BDP}(f/K)\in\Lambda_{\scO^{\rm ur}}(\Gamma^-)
\]
such that for every  character $\xi$ of $\Gamma^-$ crystalline at both $v$ and $\overline{v}$ corresponding to a Hecke character of $K$ of infinity type $(-j,j)$ for some $j\geq k/2$ with $j\equiv 0\pmod{p-1}$ we have
\begin{align*}
\mathscr{L}_{\overline{v}}^{\rm BDP}(f/K)^2(\xi)&=\biggl(\frac{\Omega_p}{\Omega_\infty}\biggr)^{4j}\cdot(k/2+j-1)!(j-k/2)!\cdot\biggl(\frac{2\pi}{\sqrt{D_K}}\biggr)^{2j-1}\\
&\quad\times(1-a_p\xi(v)p^{-k/2}+\xi(v)^2p^{-1})^2\cdot L(f/K,\xi,k/2),
\end{align*}
where $(\Omega_\infty,\Omega_p)=(2\pi i\cdot\Omega_K,\Omega_p)\in\bC^\times\times\cO_{\bC_p}^\times$ are CM periods as in \cite[\S{2.5}]{cas-hsieh1}.
\end{thm}

\begin{proof}
This follows from the results of \cite[\S{3}]{cas-hsieh1} %(refining the construction of $p$-adic $L$-functions in \cite{bdp1}) 
as in Theorem~2.1.3 and Remark~2.1.4 of \cite{cas-TNC} (whose notations we have largely adopted).
\end{proof}

We also need to recall the Katz $p$-adic $L$-function from \cite{Katz49,de_shalit}. With notations as in \cite[Thm.~2.2.1]{cas-TNC}, this is an element $\mathscr{L}_{\overline{v}}^{\rm Katz}\in\Lambda_{\scO^{\rm ur}}(\Gamma_K)$ such that for every character $\xi$ of $\Gamma_K$ crystalline at both $v$ and $\overline{v}$ corresponding to a Hecke character of infinity type $(b,a)$ with $a<0$ and $b\geq 1$ satisfies
\begin{equation}\label{eq:interp-katz}
\mathscr{L}^{\rm Katz}_{\overline{v}}(\xi)=\biggr(\frac{\Omega_p}{\Omega_\infty}\biggr)^{a-b}\cdot(a-1)!\cdot\biggl(\frac{\sqrt{D_K}}{2\pi}\biggr)^{b}\cdot(1-\xi^{-1}(\overline{v})p^{-1})(1-\xi(v)\bigr)\cdot L(\xi,0).
\end{equation}
Moreover, as shown in \cite[Thm.~II.6.4]{de_shalit} we have the functional equation
\begin{equation}\label{eq:Katz-equation}
\mathscr{L}_{\overline{v}}^{\rm Katz}(\xi)=\mathscr{L}_{\overline{v}}^{\rm Katz}(\xi^{-1}\mathbf{N}^{-1}),
\end{equation}
where the equality is up to a $p$-adic unit. 

\begin{defn}\label{def:L-Gr}
Denote by $\mathscr{L}_{\overline{v}}^{{\rm Katz},-}$ the image of $\mathscr{L}_{\overline{v}}^{\rm Katz}$ under map $\Lambda_{\scO^{\rm ur}}(\Gamma_K)\rightarrow\Lambda_{\scO^{\rm ur}}(\Gamma^-)$ induced by sending $\gamma\in\Gamma_K$ to the image of $\gamma^{1-c}$ in $\Gamma^-$, and for every character $\eta:\Delta\rightarrow\bZ_p^\times$ put
\[
\mathcal{L}_{1,2}^{\rm (II)}(f/K)^\eta:=\frac{h_K}{w_K}\cdot\mathscr{L}_{\overline{v}}^{{\rm Katz},-}\cdot\frac{\eta_{\mathbf{h}_v}}{c\cdot\omega_{\mathbf{h}_v}}\cdot\frac{[\varphi(\omega_f),\omega_f]_{\rm dR}}{\delta_{k-1}}\cdot L_{1,2}^{\rm (II)}(\mathbf{h}_v,f)^\eta,
\]
where $h_K$ is the class number of $K$ and $w_K=\#(\cO_K^\times)/2$.%/\{\pm{1}\})$.
\end{defn}

\begin{prop}\label{prop:L12-CLW}
The $p$-adic $L$-function $\varpi^{-s}\cdot\mathcal{L}_{1,2}^{\rm (II)}(f/K)^\eta$ is integral, i.e. lives in $\Lambda_{\scO^{\rm ur}}(\Gamma_K)$.
\end{prop}

\begin{proof}
This follows from a direct comparison with the interpolation property of the $p$-adic $L$-function $\mathcal{L}_{\overline{v}}(f/K)^\eta\in\Lambda_{\scO^{\rm ur}}(\Gamma_K)$ in \cite[Prop.~8.2.2]{CLW}, yielding the equality
\begin{equation}\label{eq:L12-CLW}
\varpi^{-s}\cdot\mathcal{L}_{1,2}^{\rm (II)}(f/K)^\eta=\mathcal{L}_{\overline{v}}(f/K)^\eta \nonumber
\end{equation}
up to a unit. The details are virtually the same as in  Proposition~\ref{prop:Gr-BDP} below. \end{proof}

\begin{prop}\label{prop:Gr-BDP}
Let $I^+$ be the kernel of the natural projection $\Lambda_{\scO^{\rm ur}}(\Gamma_K)\rightarrow\Lambda_{\scO^{\rm ur}}(\Gamma^-)$. Then
\[
\bigl(\varpi^{-s}\cdot{\rm Tw}_{k/2-1}(\mathcal{L}_{1,2}^{\rm (II)}(f/K)^{\omega^{1-k/2}})\;{\rm mod}\,I^+\bigr)=\bigl(\mathscr{L}_{\overline{v}}^{\rm BDP}(f/K)^2\bigr)
\]
as ideals in $\Lambda_{\scO^{\rm ur}}(\Gamma^-)$, where ${\rm Tw}_{k/2-1}:\Lambda_{\scO}(\Gamma_K)\rightarrow\Lambda_{\scO}(\Gamma_K)$ denotes the $\scO$-linear isomorphism given by $\gamma\mapsto\langle\varepsilon(\gamma)\rangle^{k/2-1}\gamma$ for $\gamma\in\Gamma^+$ and the identity on $\Gamma^-$.
\end{prop}

\begin{proof}
This follows similarly as in \cite[Prop.~2.4.5]{CGS}, but we provide the details for the convenience of the reader. For $\eta=\omega^{1-k/2}$, $\chi=(\chi_1,\chi_2)\in\Xi^{\rm (II)}$ with $m=2j$ for $j\geq k/2$ with $j\equiv 0\pmod{p-1}$, $n=k/2+j-1$, and $\zeta=1$, the interpolation in Theorem~\ref{thm:ERL-II} becomes
\begin{equation}\label{eq:esp-12}
\begin{aligned}
\phi_\chi({\rm Tw}_{k/2-1}(L_{1,2}^{\rm (II)}(\mathbf{h}_v,f)^{\omega^{1-k/2}}))&=\frac{1}{\rm  (denom)}\cdot\biggl(1-\frac{p^{j-k/2}\alpha}{\chi_1(\overline{v})}\biggr)^2\biggl(1-\frac{p^{j-k/2}\beta}{\chi_1(\overline{v})}\biggr)^2\\
&\quad\times(k/2+j-1)!(j-k/2)!\cdot L(\mathbf{h}_{v,\chi_1}^\circ,f,k/2+j),
\end{aligned}
\end{equation}
where
\[
{\rm(denom)}=\pi^{2j+1}\cdot(-i)^{2j+1+k}\cdot 2^{4j+1}\cdot\biggl(1-\frac{p^{2j-1}}{\chi_1(\overline{v})^2}\biggr)\biggl(1-\frac{p^{2j}}{\chi_1(\overline{v})^2}\biggr)\cdot\langle\mathbf{h}_{v,\chi_1}^\circ,\mathbf{h}_{v,\chi_1}^\circ\rangle.
\]
Composed with the projection $\Gamma_K\twoheadrightarrow\Gamma^v$, the character $\chi_1\varepsilon^j$ corresponds to a Hecke character $\xi$ of infinity type $(-j,j)$ with
\begin{equation}\label{eq:euler-xi}
\biggl(1-\frac{p^{2j-1}}{\chi_1(\overline{v})^2}\biggr)\biggl(1-\frac{p^{2j}}{\chi_1(\overline{v})^2}\biggr)=\bigl(1-\xi(v)^2p^{-1}\bigr)\bigl(1-\xi(v)^2\bigr).
\end{equation}
In terms of $\xi$, Hida's adjoint $L$-value formula (see \cite[Thm.~7.1]{HT-ENS}) combined with Dirichlet's class number formula give the equality up to a $p$-adic unit
\[
\langle\mathbf{h}_{v,\chi_1}^\circ,\mathbf{h}_{v,\chi_1}^\circ\rangle\,\sim_p\,(2j)!\cdot\frac{1}{2^{4j-1}\cdot\pi^{2j+1}}\cdot\frac{h_K}{w_K}\cdot L(\xi^2,1),
\]
%where $\sim_p$ denotes equality up to a $p$-adic unit, and
where we used the fact that $\xi/\xi^\tau=\xi^2$. 
%, since $\xi$ is anticyclotomic. 
Thus together with  \eqref{eq:interp-katz} (expressing $L(\xi^2,1)$ in terms of the value of $\mathscr{L}_{\overline{v}}^{\rm Katz}$ at $\xi^2\mathbf{N}^{-1}$, of infinity type $(1-2j,1+2j)$) and \eqref{eq:euler-xi} we obtain
\begin{equation}\label{eq:denom}
\begin{aligned}
{\rm (denom)}&=\bigl(1-\xi(v)^2p^{-1}\bigr)\bigl(1-\xi(v)^2\bigr)\cdot\langle\mathbf{h}_{v,\chi_1}^\circ,\mathbf{h}_{v,\chi_1}^\circ\rangle\cdot\pi^{2j+1}\cdot(-i)^{2j+1-k}\cdot 2^{4j+1}\\
%&=\bigl(1-\xi(v)^2p^{-1}\bigr)\bigl(1-\xi(v)^2\bigr)\cdot(2j)!\cdot 2^2\cdot(-i)^{2j+1-2r}\cdot L(\xi^2,1)\\
&=\frac{\frac{h_K}{w_K}\cdot\mathscr{L}_{\overline{v}}^{\rm Katz}(\xi^2\mathbf{N}^{-1})\cdot (-i)^{2j+1-k}\cdot 2^4}{\bigl(\frac{\Omega_p}{\Omega_\infty}\bigr)^{4j}\cdot\bigl(\frac{2\pi}{\sqrt{D_K}}\bigr)^{2j-1}}.
\end{aligned}
\end{equation}
Taking the product of \eqref{eq:esp-12} and \eqref{eq:denom}, using the relation
\[
\biggl(1-\frac{p^{j-k/2}\alpha}{\chi_1(\overline{v})}\biggr)^2\biggl(1-\frac{p^{j-k/2}\beta}{\chi_1(\overline{v})}\biggr)^2=\bigl(1-a_p\xi(v)p^{-k/2}+\xi(v)^2p^{-1}\bigr)^2
\]
and the functional equation \eqref{eq:Katz-equation}, the interpolation property in Theorem~\ref{thm:BDP-Lp} yields the result.
\end{proof}

\subsection{Cyclotomic descent}\label{subsubsec:restr-cyc}

%\subsubsection{$p$-adic $L$-functions}

\begin{defn}\label{def:normalize-I}
For $\xi\in\{\alpha,\beta\}$, $i\in\{1,2\}$, and a character $\eta:\Delta\rightarrow\bZ_p^\times$, put
\[
\mathcal{L}^{\rm (I)}_{\xi,\xi}(f/K)^\eta:=c_f\cdot L_{\xi,\xi}^{\rm (I)}(f,\mathbf{h}_v)^\eta,\quad\quad
\mathcal{L}_{i,i}^{\rm (I)}(f/K)^\eta:=c_f\cdot L^{\rm (I)}_{i,i}(f,\mathbf{h}_v)^\eta,
\]
where $c_f$ is the congruence number of $f$.
\end{defn}

\begin{rem}
Note that the specialisation of $\mathbf{h}_{v}$ at the trivial character is the ordinary $p$-stabilisation of the weight one Eisenstein series 
\[
\mathbf{h}_{v,1}^\circ=E_1(1,\epsilon_K).
\]
Thus from Theorem~\ref{thm:ERL-I} we see that %for $\xi\in\{\alpha,\beta\}$ 
the $L$-values interpolated by $\phi_\chi(\mathcal{L}_{\xi,\xi}^{\rm (I)}(f/K)^\eta)$ for $\chi=(1,\chi_2)\in\Xi^{\rm (I)}$ with $\chi_2=\psi_\zeta\langle\varepsilon\rangle^j$ (for $0\leq j\leq k-2$ and $\zeta\in\mu_{p^\infty}$) are given by
\begin{equation}\label{eq:can-period}
\frac{L(f\otimes\omega^j\eta\psi_\zeta^{-1},j+1)\cdot L(f\otimes\epsilon_K\omega^j\eta\psi_\zeta^{-1},j+1)}{\pi^{2j-k+2}\cdot(-i)^{k-1}\cdot 2^{2j-k+1}\cdot\Omega_f^{\rm can}},
\end{equation}
where 
\[
\Omega_f^{\rm can}:=\frac{(4\pi)^k\cdot\langle f,f\rangle_N}{c_f}
\]
is Hida's canonical period (see e.g. \cite[\S{6}]{ChHs1}).
\end{rem}

In the following, we shall often make the identification $\Gamma^+\cong\Gamma$.

\begin{prop}\label{prop:L-cyc}
Let $I^-$ be the kernel of the natural projection $\Lambda_L(\Gamma_K)\twoheadrightarrow\Lambda_{L}(\Gamma)$. Then for every $i\in\{1,2\}$ and  $\eta:\Delta\rightarrow\bZ_p^\times$ we have
\[
\bigl(\varpi^{-s}\cdot\mathcal{L}_{i,i}^{\rm (I)}(f/K)^\eta\;{\rm mod}\,I^-\bigr)=\bigl(L_{i}(f)^\eta\cdot L_{i}(f\otimes\epsilon_K)^\eta\bigr)
\]
as ideals in $\Lambda_{\scO}(\Gamma)$.
%, where $r_f\in\scO$ is as in \eqref{eq:comp-omega}.
\end{prop}

\begin{proof}
This can be shown in different ways (see e.g. \cite{lei-amburgh} for an approach in the weight $2$ case with $v_p(a_p)>1/p$). Here we adapt the argument given in \cite[Lem.~4.23]{sprung-IMC} for rational elliptic curves. 

Write $Q_{\alpha,\beta}^{-1}M_{\rm log}=\bigl(\begin{smallmatrix}
\mathfrak{A}&\mathfrak{B}\\
\mathfrak{C}&\mathfrak{D}
\end{smallmatrix}\bigr)$. The decomposition \eqref{eq:factor-Lp} and its analogue for $f^K:=f\otimes\epsilon_K$ give
\begin{equation}\label{eq:f-fK}
\begin{aligned}
L_\alpha(f)^\eta\cdot L_\alpha(f^K)^\eta&=\mathfrak{A}^2\cdot L_1(f)^\eta \cdot L_1(f^K)^\eta+\mathfrak{AB}\cdot\bigl(L_1(f)^\eta \cdot L_2(f^K)^\eta+L_2(f)^\eta \cdot L_1(f^K)^\eta\bigr)\\
&\quad+\mathfrak{B}^2\cdot L_2(f)^\eta\cdot L_2(f^K)^\eta.
\end{aligned}
\end{equation}
On the other hand, from Corollary~\ref{cor:dec-2var-signed} we have
\begin{equation}\label{eq:fK}
\begin{aligned}
\varpi^s\cdot\mathcal{L}^{\rm (I)}_{\alpha,\alpha}(f/K)^\eta&=\mathfrak{A}^2\cdot \mathcal{L}^{\rm (I)}_{1,1}(f/K)^\eta+\mathfrak{AB}\cdot\bigl(\mathcal{L}^{\rm (I)}_{1,2}(f/K)^\eta+\mathcal{L}^{\rm (I)}_{2,1}(f/K)^\eta\bigr)\\
&\quad+\mathfrak{B}^2\cdot\mathcal{L}^{\rm (I)}_{2,2}(f/K)^\eta.
\end{aligned}
\end{equation}

By \cite{bellaiche-interp} and \cite{LZ-Coleman}, respectively, there exist $p$-adic $L$-functions interpolating $L_\alpha(f)^\eta$ and $L_\alpha(f^K)^\eta$ (resp.  $\mathcal{L}_{\alpha,\alpha}^{\rm (I)}(f/K)^\eta$) along the Coleman family passing through the $p$-stabilisation $f^\alpha$ of $f$ with $U_p$-eigenvalue $\alpha$, and a direct comparison of the interpolation formulas in Theorem~\ref{thm:MTT} and Theorem~\ref{thm:ERL-I}, using the equalities up to a $p$-adic unit
\[
\Omega_f^{\rm can}\,\sim_p\,\Omega_{\omega_f}^+\cdot\Omega_{\omega_f}^-,\quad\quad\Omega_{\omega_f}^\pm\,\sim_p\Omega_{\omega_{f^K}}^{\mp}
\]
(see \cite[\S{4.4}]{DDT} and \cite[Lem.~9.6]{skinner-zhang}) and Remark~\ref{eq:can-period} 
%and the fact that the specialization $\mathbf{h}_{v,1}^\circ$ is the weight one Eisenstein series $E_1(1,\epsilon_K)$, and so
%\[
%L(f\otimes\psi_\zeta^{-1}\omega^j,\mathbf{h}_{v,1}^\circ,j+1)=L(f\otimes\psi_\zeta^{-1}\omega^j,j+1)\cdot L(f\otimes\epsilon_K\psi_\zeta^{-1}\omega^j,j+1),
%\]
shows that 
\begin{equation}\label{eq:unb}
\bigl(\mathcal{L}^{\rm (I)}_{\xi,\xi}(f/K)^\eta\;{\rm mod}\,I^-\bigr)=\bigl(L_{\xi}(f)^\eta\cdot L_{\xi}(f^K)^\eta\bigr).
\end{equation}

Combining \eqref{eq:f-fK}, \eqref{eq:fK}, and \eqref{eq:unb}, we thus see that it suffices to show that $\mathfrak{A}^2$, $\mathfrak{AB}$, and $\mathfrak{B}^2$ are linearly independent over $\Lambda_{\scO}(\Gamma)$. Since $\mathfrak{A}$ and $\mathfrak{B}$ have only finitely many zeros in common (indeed, this follows from the fact that by \cite[Cor.~3.2]{LLZ-ANT} such common zeroes are of the form $z=\langle\varepsilon(\gamma)\rangle^j\zeta$ for $0\leq j\leq k-2$ and $\zeta$ a $p$-power root of unity, but the relation $L_\alpha(f)^\eta=\mathfrak{A}\cdot L_1(f)^\eta+\mathfrak{B}\cdot L_2(f)^\eta$ from \eqref{eq:factor-L} and the main result of \cite{rohrlich-division} imply that $\mathfrak{A}$ and $\mathfrak{B}$ can vanish simultaneously at only finitely many such $z$), the $\Lambda_{\scO}(\Gamma)$-linear independence of $\mathfrak{A}^2$, $\mathfrak{AB}$, and $\mathfrak{B}^2$ is readily checked by a case by case analysis, whence the result.
\end{proof}

\begin{rem}\label{rem:integrality-Lii}
Note that \emph{a priori} $\varpi^{-s}\cdot\mathcal{L}^{\rm (I)}_{i,i}(f/K)$ has bounded denominators, i.e. it lands  in $\Lambda_L(\Gamma_K)$ rather than $\Lambda_{\scO}(\Gamma_K)$. Proposition~\ref{prop:L-cyc} above shows that its projection to $\Lambda_L(\Gamma)$ is contained in $\Lambda_{\scO}(\Gamma)$, and later we shall show that in fact $\varpi^{-s}\cdot\mathcal{L}_{i,i}^{\rm (I)}(f/K)\in\Lambda_{\scO}(\Gamma_K)$ (see the end of proof of Corollary~\ref{cor:ii-IMC}).
\end{rem}

Recall that by Proposition~\ref{prop:L-nonzero} for every $\eta:\Delta\rightarrow\bZ_p^\times$ we have $L_i(f)^\eta\neq 0$ for some $i\in\{1,2\}$.

\begin{cor}\label{cor:L-cyc}
Suppose $i\in\{1,2\}$ and $\eta:\Delta\rightarrow\bZ_p^\times$ are such that $L(f\otimes\epsilon_K\eta,1)\neq 0$ and $L_i(f)^\eta\neq 0$. Then 
\[
(\mathcal{L}^{\rm (I)}_{i,i}(f/K)^\eta\;{\rm mod}\,I^-)\neq 0,
\]
and so condition \eqref{eq:Lp-nonzero} holds.
\end{cor}

\begin{proof}
By \cite[Prop.~3.28]{LLZ-AJM}, %for any $i\in\{1,2\}$ 
the $p$-adic $L$-function $L_i(f\otimes\epsilon_K)^\eta$ interpolates a nonzero multiple of $L(f\otimes\epsilon_K\eta,1)$ at the trivial character, so 
Proposition~\ref{prop:L-cyc} yields the result.
\end{proof}

\begin{rem}\label{rem:k>2}
As is well-known (see e.g. \cite[Prop.~2]{shimuraCPAM}), the condition in Corollary~\ref{cor:L-cyc} is automatic unless $k=2$.
\end{rem}

%\section{$p$-adic Hodge theory}

%\subsection{Cyclotomic descent}\label{subsec:cyc-descent}

Replacing $K_\infty$ by $K_\infty^+$ and  $e_\eta\widetilde{\mathscr{C}}_{i,w}$ by its reduction modulo $I^-$, the Selmer groups $X_{i,i}(f_\eta/K_\infty^+)$ are defined in the same manner as in $\S\ref{subsec:2var-Sel}$. We conclude this section with a counterpart of Proposition~\ref{prop:L-cyc} for algebraic $p$-adic $L$-functions (see Corollary~\ref{cor:cyc-descent} below). 

\begin{prop}\label{prop:f-fK}
Let $i\in\{1,2\}$ and $\eta:\Delta\rightarrow\bZ_p^\times$ be a character. %and assume that $X_{i,i}(f_\eta/K_\infty^+)$ is $\Lambda_{\scO}(\Gamma)$-torsion.
Then 
\[
{\rm char}_{\Lambda_{\scO}(\Gamma)}\bigl(X_{i,i}(f_\eta/K_\infty^+)\big)={\rm char}_{\Lambda_{\scO}(\Gamma)}\bigl(X_i(f/\bQ(\mu_{p^\infty}))^\eta\bigr)\cdot{\rm char}_{\Lambda_{\scO}(\Gamma)}\bigl(X_i(f\otimes\epsilon_K/\bQ(\mu_{p^\infty}))^\eta\bigr).
\]
%The restriction map from $G_\bQ$ to $G_K$ yields a $\Lambda_{\scO}(\Gamma)$-module isomorphism
%\[
%X_{i,i}(f_\eta/K_\infty^+)\cong X_i(f/\bQ(\mu_{p^\infty}))^\eta\oplus X_i(f\otimes\epsilon_K/\bQ(\mu_{p^\infty}))^\eta.
%\]
%In particular,  
%\[
%{\rm char}_{\Lambda_{\scO}(\Gamma)}\bigl(X_{i,i}(f_\eta/K_\infty^+)\big)={\rm char}_{\Lambda_{\scO}(\Gamma)}\bigl(X_i(f/\bQ(\mu_{p^\infty}))^\eta\bigr)\cdot{\rm char}_{\Lambda_{\scO}(\Gamma)}\bigl(X_i(f\otimes\epsilon_K/\bQ(\mu_{p^\infty}))^\eta\bigr).
%\]
\end{prop}

\begin{proof}
Let $\Sigma$ be the set of rational primes dividing $N$. We begin by showing that the restriction map from $G_\bQ$ to $G_K$ yields a $\Lambda_{\scO}(\Gamma)$-module isomorphism
\begin{equation}\label{eq:impr-iso}
X_{i,i}^\Sigma(f_\eta/K_\infty^+)\cong X_i^\Sigma(f/\bQ(\mu_{p^\infty}))^\eta\oplus X_i^\Sigma(f\otimes\epsilon_K/\bQ(\mu_{p^\infty}))^\eta
\end{equation}
for the $\Sigma$-imprimitive variants (obtained by removing the local conditions at the primes above $\Sigma$) of the Selmer groups in the statement.

The action of complex conjugation yields a decomposition
\[
\rH^1(K,T(\eta)\otimes\Lambda_{\scO}(\Gamma)^\iota)=\rH^1(K,T(\eta)\otimes\Lambda_{\scO}(\Gamma)^\iota)^+\oplus\rH^1(K,T(\eta)\otimes\Lambda_{\scO}(\Gamma)^\iota)^-,
\]
and from inflation-restriction we see that the restriction map from $G_\bQ$ to $G_K$ induces $\Lambda_{\scO}(\Gamma)$-module isomorphisms 
\begin{align*}
\rH^1(\bQ,T(\eta)\otimes\Lambda_{\scO}(\Gamma)^\iota)&\cong\rH^1(K,T(\eta)\otimes\Lambda_{\scO}(\Gamma)^\iota)^+,\\
\rH^1(\bQ,T(\eta)\otimes\epsilon_K\otimes\Lambda_{\scO}(\Gamma)^\iota)&\cong\rH^1(K,T(\eta)\otimes\Lambda_{\scO}(\Gamma)^\iota)^-.
\end{align*}
Thus, putting $K_p=K\otimes_{\bQ}\bQ_p=K_v\oplus K_{\overline{v}}$, on which ${\rm Gal}(K/\bQ)$ act by interchanging the factors, it suffices to show that the restriction map from $G_\bQ$ to $G_K$ induces $\Lambda_{\scO}(\Gamma)$-module isomorphisms
\begin{align*}
\rH^1_i(\bQ_p,T(\eta)\otimes\Lambda_\scO(\Gamma)^\iota)&\cong\rH^1_i(K_p,T(\eta)\otimes\Lambda_\scO(\Gamma)^\iota)^+,\\
\rH^1_i(\bQ_p,T(\eta)\otimes\epsilon_K\otimes\Lambda_{\scO}(\Gamma)^\iota)&\cong\rH^1_i(K_p,T(\eta)\otimes\Lambda_{\scO}(\Gamma)^\iota)^-;
\end{align*}
but this is clear from the fact that  
\[
z_p=(z_v,z_{\overline{v}})\in\rH^1(K_p,T(\eta)\otimes\Lambda_{\scO}(\Gamma)^\iota)^\pm\;\Longleftrightarrow\; c\cdot z_v=\pm z_v,
\]
and via the identification $K_v=\bQ_p$ we have
$\rH^1(K_v,T(\eta)\otimes\Lambda_{\scO}(\Gamma)^\iota)={\rm ker}(e_\eta\widetilde{\mathscr{C}}_i)$, and 
$\rH^1(K_{\overline{v}},T(\eta)\otimes\Lambda_{\scO}(\Gamma)^\iota)={\rm ker}(e_\eta\widetilde{\mathscr{C}}_i\circ c)$. Thus we have \eqref{eq:impr-iso}, and so
\[
{\rm char}_{\Lambda_{\scO}(\Gamma)}\bigl(X^\Sigma_{i,i}(f_\eta/K_\infty^+)\big)={\rm char}_{\Lambda_{\scO}(\Gamma)}\bigl(X_i^\Sigma(f/\bQ(\mu_{p^\infty}))^\eta\bigr)\cdot{\rm char}_{\Lambda_{\scO}(\Gamma)}\bigl(X_i^\Sigma(f\otimes\epsilon_K/\bQ(\mu_{p^\infty}))^\eta\bigr).
\]
That this equality of characteristic ideals then also holds for the primitive Selmer groups (interpreted as the equality ``$0=0$'' when either of the Selmer groups involved is not $\Lambda_{\scO}(\Gamma)$-torsion) follows by a standard application of \cite[Prop.~2.4]{GV} as in e.g. the proof of \cite[Thm.~6.1.6]{JSW}.
%
%the semi-local Coleman map ${\rm Col}_{K_p}^\circ:\rH^1(K_p,T_f\otimes\Lambda)\rightarrow\Lambda$ is defined (after the identification $K_v=\bQ_p$) as $1\otimes{\rm Col}_{\bQ_p}^\circ$ under the isomoprhism
%\[
%\rH^1(K_p,T_f\otimes\Lambda)\cong\bZ_p[{\rm Gal}(K/\bQ)]\otimes_{\bZ_p}\rH^1(K_v,T_f\otimes\Lambda)
%\]
%and $\bZ_p[{\rm Gal}(K/\bQ)]\cong\bZ_p\oplus\bZ_p(\epsilon_K)$ as $G_\bQ$-modules.
\end{proof}

\begin{prop}\label{prop:cyc-restr}
Let $i\in\{1,2\}$ and $\eta:\Delta\rightarrow\bZ_p^\times$ be a character. Then 
%
%The restriction map from $G_{K_\infty^+}$ to $G_{K_\infty}$ yields a $\Lambda_{\scO}(\Gamma)$-module isomorphism
%\[
%X_{i,i}(f_\eta/K_\infty)\otimes_{\Lambda_\scO(\Gamma_K)}\Lambda_{\scO}(\Gamma)\cong X_{i,i}(f_\eta/K_\infty^+).
%\]
%In particular, letting $I^-={\rm ker}(\Lambda_{\scO}(\Gamma_K)\rightarrow\Lambda_{\scO}(\Gamma))$ 
we have the divisibility 
\[
({\rm char}_{\Lambda_{\scO}(\Gamma_K)}\bigl(X_{i,i}(f_\eta/K_\infty)\bigr)\,{\rm mod}\;I^-)\supset
{\rm char}_{\Lambda_{\scO}(\Gamma)}\bigl(X_{i,i}(f_\eta/K_\infty^+)\bigr),
\]
where $I^-={\rm ker}(\Lambda_{\scO}(\Gamma_K)\rightarrow\Lambda_{\scO}(\Gamma))$.
\end{prop}

\begin{proof}
With notations as in the proof of Proposition~\ref{prop:f-fK}, we begin by noting that the restriction map from $G_{K_\infty^+}$ to $G_{K_\infty}$ yields a $\Lambda_{\scO}(\Gamma)$-module isomorphism
\[
X_{i,i}^\Sigma(f_\eta/K_\infty)\otimes_{\Lambda_\scO(\Gamma_K)}\Lambda_{\scO}(\Gamma)\cong X_{i,i}^\Sigma(f_\eta/K_\infty^+),
\]
(Indeed, this follows from a standard application of the Snake Lemma using \cite[Thm.~4.7]{LZ2}.) By part (ii) of  \cite[Cor.~3.8]{SU} it follows that
\[
({\rm char}_{\Lambda_{\scO}(\Gamma_K)}\bigl(X^\Sigma_{i,i}(f_\eta/K_\infty)\bigr)\,{\rm mod}\;I^-)\supset
{\rm char}_{\Lambda_{\scO}(\Gamma)}\bigl(X^\Sigma_{i,i}(f_\eta/K_\infty^+)\bigr),
\]
from where the result follows similarly as in Proposition~\ref{prop:f-fK}.
\end{proof}

%Combining Propositions~\ref{prop:f-fK} and \ref{prop:cyc-restr}, in particular we deduce the following.

\begin{cor}\label{cor:cyc-descent}
For every $i\in\{1,2\}$ and every character $\eta:\Delta\rightarrow\bZ_p^\times$ we have the divisibility 
\begin{align*}
({\rm char}_{\Lambda_\scO(\Gamma_K)}\bigl(X_{i,i}(f_\eta/K_\infty)\bigr)
\,&{\rm mod}\;I^-)\\
&
\supset\;{\rm char}_{\Lambda_\scO(\Gamma)}\bigr(X_i(f/\bQ(\mu_{p^\infty}))^\eta\bigr)\cdot {\rm char}_{\Lambda}\bigr(X_i(f\otimes\epsilon_K/\bQ(\mu_{p^\infty}))^\eta\bigr).
\end{align*}
%where $I^-={\rm ker}(\Lambda_{\scO}(\Gamma_K)\rightarrow\Lambda_{\scO}(\Gamma))$.
\end{cor}

\begin{proof}
This is the combination of Propositions~\ref{prop:f-fK} and \ref{prop:cyc-restr}.
\end{proof}

\subsection{Anticyclotomic descent: Definite}\label{subsec:ac-descent-def}

In this section we restrict to the case\footnote{Largely for simplicity; the general case $2\leq k\leq p-1$ should be similar.} $k=2$, and assume $(D_K,N)=1$ and 
\begin{equation}
\textrm{$N^-$ is the squarefree product of an \emph{odd} number of primes}.\tag{def}
\end{equation}
%(On the other hand, we keep our assumption that $p=v\overline{v}$ splits in $K$.)

Let $\gamma_{\rm ac}\in\Gamma^-$ be a topological generator, and for any $n\geq 1$ put $\Phi_n=\frac{\omega_n}{\omega_{n-1}}$, where $\omega_n=\gamma_{\rm ac}^{p^n}-1$,  for the $p^n$-th cyclotomic polynomial in $\gamma_{\rm ac}$. Following \cite[Def.~3.3]{BBL-adv} define $\mathscr{M}_{{\rm log},{\rm ac}}\in M_{2\times 2}(\mathcal{H}_L(\Gamma^-))$ by
\begin{equation}\label{eq:Mlog-ac}
\mathscr{M}_{{\rm log},{\rm ac}}:=\lim_{n\to\infty}\left(\begin{array}{cc}a_p&1\\
-p& 0
\end{array}
\right)^{-(n+1)}\left(\begin{array}{cc}a_p&1\\
-\Phi_n& 0
\end{array}
\right)
%\left(\begin{array}{cc}a_p&1\\
%-\Phi_{n-1}& 0
%\end{array}
%\right)
\cdots\left(\begin{array}{cc}a_p&1\\
-\Phi_1& 0
\end{array}
\right),
\end{equation}
Noting the congruence $\Phi_{n+1}\equiv p\pmod{\omega_n}$, the convergence of the limit in \eqref{eq:Mlog-ac} to an element in $M_{2\times 2}(\mathcal{H}_L(\Gamma^-))$ follows from \cite[Lem.~1.2.1]{PR115}.

For every $p$-power root of unity $\zeta$, let $\chi_\zeta:\Gamma^-\rightarrow\overline{\bQ}_p^\times$ denote the finite order character determined by $\chi_\zeta(\gamma_{\rm ac})=\zeta$. The following is a refinement and generalisation of the anticyclotomic $p$-adic $L$-function introduced by Bertolini--Darmon \cite{BDmumford-tate}.

\begin{thm}\label{thm:BD-Lp}
For every $\xi\in\{\alpha,\beta\}$, there exists a ``square-root'' $p$-adic $L$-function
\[
\Theta_{\xi}^{\rm BD}(f/K)\in\mathcal{H}_L(\Gamma^-)
\]
such that for every $p^t$-th root of unity $\zeta$ with $t>0$ we have
\[
\Theta_\xi^{\rm BD}(f/K)^2(\chi_\zeta)=\frac{1}{\xi^{2t}}\cdot\frac{L(f/K,\chi_\zeta,1)}{\Omega_{f,N^-}}\cdot\sqrt{D_K}p^t\cdot w_K^2,
\]
where
\[
\Omega_{f,N^-}=\frac{(4\pi)^2\cdot\langle f,f\rangle_N}{\eta_{f,N^-}}
\]
is Gross's period as in \cite[(4.3)]{cas-hsieh1}. Moreover, for $i\in\{1,2\}$ there exist $\Theta_i^{\rm BD}(f/K)\in\Lambda_{\scO}(\Gamma^-)$ such that
\begin{equation}\label{eq:theta-dec}
\left(\begin{array}{cc}
\Theta_\alpha^{\rm BD}(f/K)\\[0.2em]
\Theta_\beta^{\rm BD}(f/K)
\end{array}\right)=Q_{\alpha,\beta}^{-1}\mathscr{M}_{{\rm log},{\rm ac}}\cdot
\left(\begin{array}{cc}
\Theta_1^{\rm BD}(f/K)\\[0.2em]
\Theta_2^{\rm BD}(f/K)
\end{array}\right),
\end{equation}
where $Q_{\alpha,\beta}=\frac{1}{\alpha-\beta}\left(\begin{smallmatrix}\alpha & -\beta\\ -p&p\end{smallmatrix}\right)$.
\end{thm}

\begin{proof}
For each $n\ge 1$ let $\Gamma_n^-={\rm Gal}(K_n^-/K)$ denote the Galois group of the unique subfield $K_n^-\subset K_\infty^-$ with $[K_n^-:K]=p^n$. From \cite{ChHs1} we have elements $\theta_{n}^{\rm BD}(f/K)\in\scO[\Gamma_n^-]$ satisfying the three-term norm-relation
\begin{equation}\label{eq:3-term}
\pi_{n+1,n}(\theta_{n+1}^{\rm BD}(f/K))=a_p\cdot\theta_n^{\rm BD}(f/K)-\nu_{n-1,n}(\theta_{n-1}^{\rm BD}(f/K)),
\end{equation}
where $\pi_{n+1,n}:\scO[\Gamma_{n+1}^-]\rightarrow\scO[\Gamma_n^-]$ is the natural projection and $\nu_{n-1,n}:\scO[\Gamma_{n-1}^-]\rightarrow\scO[\Gamma_n^-]$ is given by $\sigma\mapsto\sum_{\tau}\tau$ with $\tau$ running over all the elements of $\Gamma_n^-$ with $\tau\vert_{K_{n-1}^-}=\sigma$. Setting
\[
\theta_{\xi,n}^{\rm BD}(f/K):=\theta_n^{\rm BD}(f/K)-\frac{1}{\xi}\cdot\nu_{n-1,n}(\theta_{n-1}^{\rm BD}(f/K)),
\]
the compatibility $\pi_{n+1,n}(\theta_{n+1,\xi}^{\rm BD}(f/K))=\xi\cdot\theta_{n,\xi}^{\rm BD}(f/K)$ follows readily from \eqref{eq:3-term}. Thus by \cite[Lem.~1.2.1]{PR115} there exist elements $\Theta_\xi^{\rm BD}(f/K)\in\mathcal{H}_L(\Gamma^-)$ projecting to $\xi^{-n}\cdot\theta_{\xi,n}^{\rm BD}(f/K)$ under $\mathcal{H}_L(\Gamma^-)\rightarrow L[\Gamma_n^-]$ for all $n\geq 1$ and interpolating the special values $L(f/K,\chi_\zeta,1)$ as in the statement by virtue of \cite[Prop.~4.3]{ChHs1}. 

On the other hand, as shown in \cite[Thm.~3.4]{BBL-adv},  from \eqref{eq:3-term} we obtain elements $\Theta_i^{\rm BD}(f/K)\in\Lambda_{\scO}(\Gamma^-)$, $i\in\{1,2\}$, satisfying
\begin{equation}\label{eq:3-term-bounded}
\left(\begin{array}{cc}a_p&1\\
-\Phi_n& 0
\end{array}
\right)
\cdots\left(\begin{array}{cc}a_p&1\\
-\Phi_1& 0
\end{array}
\right)\left(\begin{array}{cc}
\Theta_1^{\rm BD}(f/K)\\[0.2em]
\Theta_2^{\rm BD}(f/K)
\end{array}\right)\equiv\left(\begin{array}{cc}
\theta_n^{\rm BD}(f/K)\\[0.2em]
-\nu_{n-1,n}(\theta_{n-1}^{\rm BD}(f/K))
\end{array}\right)\pmod{\omega_n}
\end{equation}
for all $n\geq 1$, and the relation \eqref{eq:theta-dec} then follows as in the proof of Theorem~3.9 in \emph{op.\,cit.} (where we note that there is a typo in the first displayed equation).
\end{proof}

\begin{prop}\label{prop:ac-descent-def}
Let $i\in\{1,2\}$ and $I^+$ be the kernel of the natural projection $\Lambda_{L}(\Gamma_K)\rightarrow\Lambda_{L}(\Gamma^-)$. Then
\[
\bigl(\varpi^{-s}\cdot\mathcal{L}^{\rm (I)}_{i,i}(f/K)^\eta\;{\rm mod}\,I^+\bigr)=\Bigl(\Theta_i^{\rm BD}(f/K)^2\cdot\frac{c_f}{\eta_{f,N^-}}\Bigr)
\]
as ideals in $\Lambda_{\scO}(\Gamma^-)$, where $\eta$ is the trivial character of $\Delta$.
\end{prop}

\begin{proof}
This can be shown similarly as in the proof of Proposition~\ref{prop:L-cyc}. Writing $Q_{\alpha,\beta}^{-1}\mathscr{M}_{{\rm log},{\rm ac}}=\bigl(\begin{smallmatrix}
\mathfrak{A}&\mathfrak{B}\\
\mathfrak{C}&\mathfrak{D}
\end{smallmatrix}\bigr)$,  
from  \eqref{eq:theta-dec} we have
\begin{equation}\label{eq:theta^2}
\Theta_\alpha^{\rm BD}(f/K)^2=\mathfrak{A}^2\cdot \Theta_1^{\rm BD}(f/K)^2+2\cdot\mathfrak{AB}\cdot\Theta_1^{\rm BD}(f/K)\cdot\Theta_2^{\rm BD}(f/K)+\mathfrak{B}^2\cdot\Theta_2^{\rm BD}(f/K)^2.
\end{equation}
On the other hand, from Corollary~\ref{cor:dec-2var-signed} we have
\begin{equation}\label{eq:fK-ac}
\varpi^s\cdot\mathcal{L}^{\rm (I)}_{\alpha,\alpha}(f/K)^\eta=\mathfrak{A}^2\cdot \mathcal{L}^{\rm (I)}_{1,1}(f/K)^\eta+\mathfrak{AB}\cdot\bigl(\mathcal{L}^{\rm (I)}_{1,2}(f/K)^\eta+\mathcal{L}^{\rm (I)}_{2,1}(f/K)^\eta\bigr)+\mathfrak{B}^2\cdot\mathcal{L}^{\rm (I)}_{2,2}(f/K)^\eta.
\end{equation}
As noted in the proof of Proposition~\ref{prop:L-cyc}, by \cite{LZ-Coleman} $\mathcal{L}^{\rm (I)}_{\alpha,\alpha}(f/K)^\eta$ can be interpolated along the Coleman family passing through the $p$-stabilisation $f^\alpha$ of $f$ with $U_p$-eigenvalue $\alpha$, and similarly as in \cite{cas-longo}\footnote{Where the $p$-ordinary case is considered, but whose extension to Coleman families can be easily obtained using the techniques in e.g. \cite{BDI}.}, the same holds for $\Theta_\alpha^{\rm BD}(f/K)$. A direct comparison of the interpolation formulas of Theorem~\ref{thm:ERL-I} (as extended in \cite{loeffler-note}) and Theorem~\ref{thm:BD-Lp} then shows that 
\begin{equation}\label{eq:unb-ac}
\bigl(\mathcal{L}^{\rm (I)}_{\xi,\xi}(f/K)^\eta\;{\rm mod}\,I^+\bigr)=\Bigl(\Theta_{\xi}^{\rm BD}(f/K)^2\cdot\frac{c_f}{\eta_{f,N^-}}\Bigr).
\end{equation}

Combining \eqref{eq:theta^2}, \eqref{eq:fK-ac}, and \eqref{eq:unb-ac}, we are reduced to showing that $\mathfrak{A}^2$, $\mathfrak{AB}$, and $\mathfrak{B}^2$ are linearly independent over $\Lambda_{\scO}(\Gamma^-)$; and this follows by the same argument as in the proof of Proposition~\ref{prop:cyc-restr}, replacing the appeal to \cite{rohrlich-division} by an appeal to the nonvanishing result of \cite[Thm.~D]{ChHs1}. %(for a suitable choice of $\ell$). 
\end{proof}

We conclude this section with the following result allowing us to deduce the vanishing of Iwasawa's $\mu$-invariant for $\mathcal{L}_{i,i}^{\rm (I)}(f/K)^\eta$ under suitable ramification hypotheses on $\bar{\rho}_f$ (see Corollary~\ref{cor:ii-IMC}).

\begin{thm}\label{thm:mu=0-definite}
%Suppose $\bar{\rho}_f$ is ramified at all primes $\ell\vert N^-$. Then 
For every $i\in\{1,2\}$ we have
\[
\mu(\Theta_i^{\rm BD}(f/K))=0.
\]
\end{thm}

\begin{proof}
Since $a_n\equiv 0\pmod{\varpi}$, for each $n\geq 1$ we have the congruence
%For each $n\geq 1$, %put $\omega_n=\gamma^{p^n}-1$, 
%let $C_n\in M_{2\times 2}(L[\Gamma])$ denote the matrix with entries given by polynomials in $(\gamma-1)$ of degree $<p^n$ satisfying
%\[
%C_n\equiv\mathfrak{M}^{-1}\bigl((1+\pi)\varphi^n(P_0)\cdots\varphi(P_0)\bigr)\;({\rm mod}\;{\gamma^{p^n}-1}),
%\]
%where $P_0$ is as in Lemma~\ref{lem:Wach-basis}, and let $C_{n,{\rm ac}}$ denote the image of $C_n$ under the map $\gamma\mapsto\gamma^-$. By the same calculation as in \cite[Prop.~3.9]{gajek-leonard-lei}, from the congruence $M_{\rm log}\equiv A_\varphi^{n+1}C_n\;({\rm mod}\;\gamma^{p^n}-1)$ of \cite[Lem.~3.7]{LLZ-Sha} (with notations as in Lemma~\ref{lem:good-int-basis}) and the definitions we obtain
%\begin{equation}\label{eq:3-term-bounded}
%\left(\begin{array}{cc}
%\Theta_n^{\rm BD}(f/K)\\[0.2em]
%-\nu_{n-1,n}(\Theta_{n-1}^{\rm BD}(f/K))
%\end{array}\right)
%\equiv C_{n,{\rm ac}}\cdot
%\left(\begin{array}{cc}
%\Theta_1^{\rm BD}(f/K)\\[0.2em]
%\Theta_2^{\rm BD}(f/K)
%\end{array}\right)\;({\rm mod}\;{\gamma^{p^n}-1}).
%\end{equation}
%Moreover, as shown in \cite[Prop.~3.8]{gajek-leonard-lei} for all $n\geq 1$ we have the congruences
\begin{equation}\label{eq:cong-Cn}
\left(\begin{array}{cc}a_p&1\\
-\Phi_n& 0
\end{array}
\right)
\cdots\left(\begin{array}{cc}a_p&1\\
-\Phi_1& 0
\end{array}\right)
\equiv
\begin{cases}
\left(\begin{array}{cc}0&(-1)^{\frac{n-1}{2}}\Phi_n^+\\
(-1)^{\frac{n-1}{2}}\Phi_n^-& 0
\end{array}
\right)\;({\rm mod}\,\varpi)&\textrm{if $n$ is odd},\\[0.2em]
\left(\begin{array}{cc}(-1)^{\frac{n}{2}}\Phi_n^-&0\\
0& (-1)^{\frac{n}{2}}\Phi_n^+
\end{array}\right)\;({\rm mod}\,\varpi)&\textrm{if $n$ is even},
\end{cases}
\end{equation}
where $\Phi_n^\pm$ are appropriate products of the $\Phi_i$ for $1\leq i\leq n$ of the corresponding parity (cf. \cite[Prop.~3.8]{gajek-leonard-lei}). %We can now conclude similarly as in the proof of \cite[Thm.~2.5]{pollack-weston}: 
Since the proof of \cite[Thm.~5.7]{ChHs1}\footnote{Note that the absolute irreducibility hypothesis of \emph{loc.\,cit.} holds in our case by \cite[Thm.~2.6]{Edi} and our running assumption that $p$ splits in $K$.} shows that $\mu(\theta_n^{\rm BD}(f/K))=0$ for $n\gg 0$, combining \eqref{eq:3-term-bounded} and \eqref{eq:cong-Cn} the result follows.
\end{proof}

%\newpage

\section{Proof of the main results}

In this section we piece everything together to conclude the proof of our main result toward Kato's main Conjecture~\ref{introconj:kato}.

\subsection{Signed main conjectures of Lei--Loeffler--Zerbes}\label{subsec:reduction-signed}

%Recall that for a cuspidal eigenform $f$, we let $\pi_f=\otimes_v\pi_{f,v}$ denote the cuspidal automorphic representation of ${\rm GL}_2(\mathbb{A})$ it generates.

By Propositions~\ref{prop:L-nonzero} and \ref{prop:equiv-kato} (see also Remark~\ref{rem:equiv-KatoIMC}), the proof of Theorem~\ref{thmintro:kato} is reduced to the following result. 
%(Note that here we are considering the representation $T=T_f(k-1)$, while Theorem~\ref{thmintro:kato} is stated in terms of $T_f$, hence the switch between $\omega^{k/2}$ and $\omega^{1-k/2}$.)

\begin{thm}\label{thm:signed-Q}
Let $f\in S_k(\Gamma_0(N))$ be a newform and $p>2$ be a prime  of good nonordinary reduction. %for $f$. 
Suppose
\begin{itemize}
%\item[(i)] $f$ has good nonordinary reduction at $p$;
\item $2\leq k< p$; 
\item $N$ is squarefree; 
%\item if $N=\ell$ is prime, then $\pi_{f,\ell}\cong\sigma_\ell\otimes\xi_\ell$; 
\item $f$ is \emph{$p$-regular};
\item either $\begin{cases}
\textrm{$k=2$}, or \\ %and $a_p=0$}, or\\
\textrm{Condition~\eqref{eq:rtn} holds}.
\end{cases}$
\end{itemize}
Let $\eta=\omega^{1-k/2}$ and $i\in\{1,2\}$ be such that $L_i(f)^\eta\neq 0$. Then $X_i(f/\bQ(\mu_{p^\infty}))^\eta$ is $\Lambda_{\scO}(\Gamma)$-torsion, with 
\begin{equation}\label{eq:signed-IMC}
(\xi_{\eta,i})\cdot{\rm char}_{\Lambda_{\scO}(\Gamma)}\bigl(X_i(f/\bQ(\mu_{p^\infty}))^\eta\bigr)=\bigl(L_{i}(f)^\eta\bigr)\nonumber
\end{equation}
as ideals in $\Lambda_{\scO}(\Gamma)$.
\end{thm}

The rest of this section is devoted to the proof of Theorem~\ref{subsec:reduction-signed}.

%\begin{thm}\label{thm:signed-Q}
%Let $f\in S_k(\Gamma_0(N))$ be a newform of weight $k$ and squarefree conductor $N$, and let $p$ be an odd prime of good non-ordinary reduction for $f$. Let $i\in\{1,2\}$ be such that $L_{i}(f)\neq 0$, and assume that:
%\begin{itemize}
%\item $2\leq k\leq p$; 
%\item either $N$ is divisible by at least two primes, or condition \eqref{eq:mult-indef} holds; and
%\item condition \eqref{eq:mult-indef}; and
%\item %the residual representation 
%$\bar\rho_f\vert_{G_{\bQ_p}}$ is absolutely irreducible. 
%\end{itemize}
%Then $X_i(f)$ is $\Lambda$-torsion, with
%\begin{equation}\label{eq:signed-IMC}
%(\xi_{\mathds{1},i})\cdot{\rm char}_\Lambda\bigl(X_i(f)\bigr)=\bigl(L_{i}(f)\bigr)\nonumber
%\end{equation}
%as ideals in $\Lambda$.
%\otimes\bQ_p$. If in addition  $\bar{\rho}_f$ is ramified at some prime $\ell\Vert N$, then equality \eqref{eq:signed-IMC} holds in $\Lambda$. 
%and hence the Conjecture~\ref{conj:signed} holds true.
%\end{thm}

\subsection{Kato's divisibility}

As shown in \cite{LLZ-AJM,LLZ-ANT}, a divisibility towards Conjecture~\ref{conj:signed} follows from Kato's work.

\begin{thm}%[Kato, Lei--Loeffler--Zerbes]
\label{thm:katodiv-signed}
Let $f\in S_k(\Gamma_0(N))$ be a newform %of weight $k\geq 2$ 
and $p>2$ be a prime of good nonordinary reduction for $f$. Let $\eta=\omega^{1-k/2}$ and $i\in\{1,2\}$ be such that $L_i(f)^\eta\neq 0$. Then $X_i(f/\bQ(\mu_{p^\infty}))^\eta$ is $\Lambda_{\scO}(\Gamma)$-torsion, and  
\begin{equation}\label{eq:div-signed-IMC}
(\xi_{\eta,i})\cdot{\rm char}_{\Lambda_{\scO}(\Gamma)}\bigl(X_i(f/\bQ(\mu_{p^\infty}))^\eta\bigr)\supset\bigl(L_{i}(f)^\eta\bigr)
\end{equation}
as ideals in $\Lambda_{\scO}(\Gamma)\otimes\bQ_p$. If in addition
%\begin{equation}\label{eq:ram}
%\textrm{$\bar\rho_f$ is ramified at some prime $\ell\Vert N$},\tag{ram}
%\end{equation}
\begin{itemize}
\item[(i)] $\bar\rho_f$ is irreducible; and
\item[(ii)] $\bar\rho_f$ is ramified at some prime $\ell\Vert N$, 
\end{itemize}
then the divisibility \eqref{eq:div-signed-IMC} holds in $\Lambda_{\scO}(\Gamma)$.
\end{thm}

\begin{proof}
By \cite[Thm.~12.5]{Kato295}, $X_{\rm str}(f/\bQ(\mu_{p^\infty}))^\eta$ is $\Lambda_{\scO}(\Gamma)$-torsion, %\mathbf{z}_f\in{\rm H}_{\rm Iw}^1(\bQ(\mu_{p^\infty}),V_f)$ is contained in ${\rm H}^1_{\rm Iw}(\bQ(\mu_{p^\infty},T_f)$, and we have the divisibility 
with
\begin{equation}\label{eq:div-Kato-IMC}
{\rm char}_{\Lambda_{\scO}(\Gamma)}\bigl(X_{\rm str}(f/\bQ(\mu_{p^\infty}))^\eta\bigr)\supset{\rm char}_{\Lambda_{\scO}(\Gamma)}\biggl(\frac{{\rm H}_{\rm Iw}^1(\bQ(\mu_{p^\infty}),T)^\eta}{(\mathbf{z}_{f,k-1}^{\rm int})^\eta}\biggr)
\end{equation}
in $\Lambda_{\scO}(\Gamma)\otimes\bQ_p$. That the divisibility \eqref{eq:div-Kato-IMC} holds in $\Lambda_{\scO}(\Gamma)$ under the additional hypotheses (i)-(ii) is explained in  \cite[\S{2.5}]{skinner-mult}. %that $\bar\rho_f$  ramified at some prime $\ell\Vert N$ and absolutely irreducible. 
Together with Proposition~\ref{prop:equiv-kato}, this yields the result.
\end{proof}

\subsection{Eisenstein congruence divisibility over $K$}

Ultimately, a converse to Kato's divisibility will be deduced from the following  main result from our previous work \cite{CLW}. 

Recall the equality (up to a unit) of $p$-adic $L$-functions
\begin{equation}\label{eq:L12-CLW}
\varpi^{-s}\cdot\mathcal{L}_{1,2}^{(\rm II)}(f/K)^\eta=\mathcal{L}_{\overline{v}}(f/K)^\eta
\end{equation}
from Proposition~\ref{prop:L12-CLW}. Here and throughout the following, we put 
\[
\eta=\omega^{1-k/2}
\]
for the ease of notation. 
%Recall that 
%\[
%\pi_f=\otimes_{\ell\leq\infty}\pi_{f,\ell}
%\]
%enotes the cuspidal automorphic representation of ${\rm GL}_2(\mathbb{A})$ generated by $f$, and 
%For a prime $v$ of a quadratic imaginary field $K$, let $\varepsilon_v(\frac{1}{2},BC_K(\pi_f))$ be the local root number at $v$ 
%of the base-change of $\pi_f$ to ${\rm GL}_2(\mathbb{A}_K)$.

\begin{thm}\label{thm:CLW}
Let $f\in S_k(\Gamma_0(N))$ be a newform of squarefree level $N$ and weight $k\geq 2$, and let $p$ be an odd prime of good nonordinary reduction for $f$. Let $K$ be an imaginary quadratic field satisfying:
\begin{itemize}
\item[(i)] $p$ splits in $K$;
\item[(ii)] if $N$ is odd, then $2$ splits in $K$;
\item[(iii)] some prime factor of $N$ is nonsplit in $K$;
\item[(iv)] $\bar{\rho}_f\vert_{G_{K}}$ is %absolutely
irreducible.
\end{itemize}
Then 
\begin{equation}\label{eq:CLW}
{\rm char}_{\Lambda_{\scO}(\Gamma_K)}\bigl(X_{{\rm str},{\rm rel}}^{}(f_\eta/K_\infty)\bigr)\Lambda_{\scO^{\rm ur}}(\Gamma_K)\subset\bigl(\mathcal{L}_{\overline{v}}(f/K)^\eta\bigr)
\end{equation}
in $\Lambda_{\scO^{\rm ur}}(\Gamma_K)\otimes_{\Lambda_{\scO}(\Gamma_K)}{\rm Frac}(\Lambda_{\scO}(\Gamma))$. If in addition 
\begin{itemize}
\item[(v)] $\varepsilon_v(BC_K(\pi_f))=+1$ for every prime $v$ factor of $N$ nonsplit in $K$, 
%either 
%$\begin{cases}
%\textrm{$\varepsilon_v(\frac{1}{2},BC_K(\pi_f))=+1$ for every prime $v$ factor of $N$ nonsplit in $K$}; or\\[0.3em]
%\textrm{\eqref{eq:gen-H} holds with $N^-\neq 1$,}
%\end{cases}$
%\item[(v)] $\varepsilon_v(\frac{1}{2},BC_K(\pi_f))=+1$ for every prime $v$ factor of $N$ nonsplit in $K$; or
%\item[(v)'] \eqref{eq:gen-H} holds with $N^-\neq 1$,
\end{itemize}
%$N$ is divisible by at least two primes, or $N=\ell$ is a prime ramified in $K$ and $\varepsilon_\(\frac{1}{2},BC(\pi_\ell))=+1$,
%for every prime $\ell\Vert N$ non-split in $K$, we have $\ell$ ramified in $K$ and $\pi_\ell\cong\sigma_\ell\otimes\xi_\ell$,
%\end{itemize} 
then the divisibility \eqref{eq:CLW} holds in $\Lambda_{\scO^{\rm ur}}(\Gamma_K)$.
\end{thm}

\begin{proof}
The divisibility \eqref{eq:CLW} is proved in \cite[Thm.~8.21]{CLW} in weight $2$; the required modifications of the argument to extend the result to higher weights are explained in \cite[\S{4}]{wan-nonord} (cf. \cite[\S{7}]{fouquet-wan}).

Under the additional hypothesis (v) we have  $\mu(\mathscr{L}_{\overline{v}}^{\rm BDP}(f/K))=0$  by \cite[Thm.~B]{hsieh}. %and \cite[Thm.~B]{burungale-II}. 
By Proposition~\ref{prop:Gr-BDP} and \eqref{eq:L12-CLW}, this implies that $\mathcal{L}_{\overline{v}}(f/K)^\eta$ is not divisible by any height one prime of $\Lambda_{\scO}(\Gamma_K)$ from the inverse image of the projection $\Lambda_{\scO}(\Gamma_K)\rightarrow\Lambda_{\scO}(\Gamma)$, whence the result.
\end{proof}

%\begin{rem}
%The integral divisibility in Theorem~\ref{thm:CLW} will allow us to establish Theorem~\ref{thm:signed-Q} in \emph{Case~(1)} and \emph{Case~(2b)} above, i.e. except the case when $f$ has prime level, say $N=\ell$, with $\pi_{f,\ell}\cong\sigma_\ell$. However, we can also include this case building on the results in the next subsection and the $\mu=0$ results of \cite{vatsal,ChHs1} 
%a third case under which the ambiguity by cyclotomic height one primes in \eqref{eq:CLW} can be removed 
%(see Corollary~\ref{cor:CLW}). 
%\end{rem}

\subsection{Transferring divisibilities}\label{subsec:transfer}

%In this section, %building on the explicit reciprocity laws of the preceding sections, 
We next deduce from the lower bound divisibility in  Theorem~\ref{thm:CLW} a proof of corresponding divisibilities in related Iwasawa main conjectures; the main result is Corollary~\ref{cor:ii-IMC}. With the explicit reciprocity laws of the preceding sections at hand, the required arguments go along similar lines as those in e.g. \cite[Prop.~3.2.1]{CGS}; however, since several additional complications arise in our setting, we provide the details.

%Let $\lambda_f$ and $\xi_{\mathds{1},i}$ be as in Theorem~\ref{thm:MTT} and Proposition~\ref{prop:image-Col}, respectively.

\begin{cor}\label{cor:BF-IMC}
In the setting of and under hypotheses {\rm (i)}--{\rm (iv)} of Theorem~\ref{thm:CLW}, if $l\in\{1,2\}$ is such that  condition \eqref{eq:Lp-nonzero} holds, then  
%%in Theorem~\ref{thm:BSTW} holds (so in particular $\mathcal{L}_{i,i}(f/K)\neq 0$ by Proposition~\ref{prop:L-cyc}).
%$X_{{\rm str},l}(f/K)$ is $\Lambda_{\scO}(\Gamma_K)$-torsion, $\rH^1_{{\rm rel},l}(K,T\hat\otimes\Lambda_K)$ has $\Lambda_{\scO}(\Gamma_K)$-rank one, and
\begin{equation}\label{eq:div-BF}
(\varpi^s\cdot\xi_{\eta,l})\cdot{\rm char}_{\Lambda_{\scO}(\Gamma_K)}\bigl(X_{{\rm str},l}(f_{\eta}/K_\infty)\bigr)\subset
{\rm char}_{\Lambda_{\scO}(\Gamma_K)}\biggl(\frac{\rH^1_{{\rm rel},l}(K,\mathbf{T}_K^\eta)}{(\mathcal{BF}_{\rm int}^{l,\mathbf{h}_v,\eta})}\biggr)
\end{equation}
as ideals in $\Lambda_{\scO^{}}(\Gamma_K)\otimes_{\Lambda_{\scO}(\Gamma_K)}{\rm Frac}(\Lambda_{\scO}(\Gamma))$. If in addition  hypothesis {\rm (v)} in Theorem~\ref{thm:CLW} holds, then the divisibility \eqref{eq:div-BF} holds in $\Lambda_{\scO^{}}(\Gamma_K)$.
\end{cor}

\begin{proof}
For $i,j\in\{1,2\}$ put $\rH^1_i(K_{\overline{v}})^\eta=\rH^1_i(K_{\overline{v}},\mathbf{T}_K^\eta)$ for the ease of notation, 
let $\rH^1_{{\rm rel},i\cap j}(K,\mathbf{T}_K^\eta)$ denote the submodule of $\rH^1_{{\rm rel},j}(K,\mathbf{T}_K^\eta)$ consisting of classes $c$ satisfying ${\rm res}_{\overline{v}}(c)\in\rH^1_i(K_{\overline{v}})^\eta\cap\rH^1_j(K_{\overline{v}})^\eta$, and let $X_{{\rm str},i+j}(f_\eta/K_\infty)$ be the Pontryagin dual of the Selmer group dual to $\rH^1_{{\rm rel},i\cap j}(K,\mathbf{T}_K^\eta)$ in the sense of $\S\ref{subsec:2var-Sel}$. Then Poitou--Tate duality gives rise to the exact sequence
\begin{equation}\label{eq:PT:i-int-j}
\begin{aligned}
0\longrightarrow\rH^1_{{\rm rel},i\cap j}(K,\mathbf{T}_K^\eta)\longrightarrow&\rH^1_{{\rm rel},j}(K,\mathbf{T}_K^\eta)\xrightarrow{{\rm res}_{\overline{v}}}\frac{\rH^1_j(K_{\overline{v}})^\eta}{\rH^1_i(K_{\overline{v}})^\eta\cap\rH^1_j(K_{\overline{v}})^\eta}\\
&\longrightarrow X_{{\rm str},i+j}(f_\eta/K_\infty)\longrightarrow X_{{\rm str},j}(f_\eta/K_\infty)\longrightarrow 0.
\end{aligned}
\end{equation}
For the proof of the stated divisibility, it suffices to consider the case where $X_{{\rm str},l}(f/K_\infty)^\eta$ is $\Lambda_\scO(\Gamma_K)$-torsion (otherwise the result is obvious).  Letting $l'$ denote the element in $\{1,2\}$ different from $l$, and taking $(i,j)=(l',l)$ in \eqref{eq:PT:i-int-j}, the nonvanishing of $L_{l',l}^{\rm (II)}(\mathbf{h}_v,f)^\eta$ (see Corollary~\ref{cor:h-unb-nonzero}) and Theorem~\ref{thm:ERL-II} implies that ${\rm res}_{\overline{v}}$ has $\Lambda_{\scO}(\Gamma_K)$-torsion cokernel. Moreover, by \cite[Prop.~4.11]{LZ2} %the two-variable Perrin-Riou regulator
the map $e_\eta\mathscr{L}_{V,\overline{v}}$ is injective, and so by the decomposition \eqref{eq:factor-L-2var-w} so is $e_\eta\widetilde{\mathscr{C}}_1\oplus e_\eta\widetilde{\mathscr{C}}_2$. Thus 
\[
\rH^1_1(K_{\overline{v}})^\eta\cap\rH^1_2(K_{\overline{v}})^\eta=0,
\] 
which implies $\rH^1_{{\rm rel},l'\cap l}(K,\mathbf{T}_K^\eta)=\rH^1_{{\rm rel},{\rm str}}(K,\mathbf{T}_K^\eta)=0$, and so \eqref{eq:PT:i-int-j} reduces to
\begin{equation}\label{eq:PT:ll'}
0\longrightarrow\rH^1_{{\rm rel},l}(K,\mathbf{T}_K^\eta)\xrightarrow{{\rm res}_{\overline{v}}}\rH^1_l(K_{\overline{v}})^\eta\longrightarrow X_{{\rm str},{\rm rel}}(f_\eta/K_\infty)\longrightarrow X_{{\rm str},l}(f_\eta/K_\infty)\longrightarrow 0.
\end{equation}
By the construction in \cite[Prop.~10.1.1]{KLZ1} the differential $\eta_{\mathbf{h}_v}$ has denominators controlled by the %congruence ideal $I_{\mathbf{h}_v}\subset{\rm Frac}(\Lambda_{\bZ_p}(\Gamma^v))$
congruence power series $H_v\in{\rm Frac}(\Lambda_{\bZ_p}(\Gamma^v))$ associated to $\mathbf{h}_v$ as in \cite[\S{4.3}]{BSTW}. By Theorem~4.20 in \emph{loc.\,cit.}, we have
\[
(H_v)=\Bigl(\frac{h_K}{w_K}\cdot\mathscr{L}_{\overline{v}}^{{\rm Katz},-}\Bigr)
\]
as ideals in $\Lambda_{\scO^{\rm ur}}(\Gamma^v)$. 
%As shown in \cite[Prop.~8.3]{wanIMC}, it follows from the work of Hida--Tilouine \cite{HT-117} and Rubin's proof of the Iwasawa main conjecture for $K$ \cite{rubin-IMC} that we have 
%\begin{equation}\label{eq:div-HT-rubin}
%\frac{h_K}{w_K}\cdot\mathscr{L}_{\overline{v}}^{{\rm Katz},-}\subset\frac{c\cdot\omega_{\mathbf{h}_v}}{\eta_{\mathbf{h}_v}}.
%\end{equation}
Thus together with the argument in \cite[Lem.~4.5]{wan-ss} showing the divisibility $(H_v)\subset(\frac{c\cdot\omega_{\mathbf{h}_v}}{\eta_{\mathbf{h}_v}})$, we deduce that Theorem~\ref{thm:CLW} (see also Definition~\ref{def:L-Gr} and \eqref{eq:L12-CLW}) implies the divisibility
\begin{equation}\label{eq:CLW+HT}
(\delta_{k-1})\cdot{\rm char}_{\Lambda_{\scO^{}}(\Gamma_K)}\bigl(X_{{\rm str},{\rm rel}}^{}(f_\eta/K_\infty)\bigr)\subset\bigl(\varpi^{-s}\cdot L_{1,2}^{\rm (II)}(\mathbf{h}_v,f)^\eta\bigr)
\end{equation}
in $\Lambda_{\scO^{}}(\Gamma_K)\otimes_{\Lambda_{\scO}(\Gamma_K)}{\rm Frac}(\Lambda_{\scO}(\Gamma))$, and even in $\Lambda_{\scO^{}}(\Gamma_K)$ if hypothesis (v) of Theorem~\ref{thm:CLW} holds.
On the other hand, consider the tautological exact sequence
\begin{equation}\label{eq:tauto-ses}
0\longrightarrow\rH^1_l(K_{\overline{v}})^\eta\cong\frac{\rH^1_l(K_{\overline{v}})^\eta+\rH^1_{l'}(K_{\overline{v}})^\eta}{\rH^1_{l'}(K_{\overline{v}})^\eta}\longrightarrow\frac{\rH^1(K_{\overline{v}})^\eta}{\rH^1_{l'}(K_{\overline{v}})^\eta}\longrightarrow\frac{\rH^1(K_{\overline{v}})^\eta}{\rH^1_l(K_{\overline{v}})^\eta+\rH^1_{l'}(K_{\overline{v}})^\eta}\longrightarrow 0,
\end{equation}
using that ${\rm H}^1_l(K_{\overline{v}})^\eta\cap{\rm H}^1_{l'}(K_{\overline{v}})^\eta=0$, as shown as above, for the left isomorphism. By \cite[Prop.~2.17]{BL-non-ord} and the discussion following \emph{loc.\,cit.}, the right-most term of \eqref{eq:tauto-ses} is $\Lambda_\scO(\Gamma_K)$-torsion, with characteristic ideal generated by $\varpi^n\cdot\delta_{k-1}/\xi_{\eta,2}$ for some integer $n$. %up to powers of $\varpi$. 
Thus from Proposition~\ref{prop:image-Col-2var} and Theorem~\ref{thm:ERL-II} we find 
\begin{equation}\label{eq:ERLii}
(\varpi^n\cdot\delta_{k-1})\cdot{\rm char}_{\Lambda_\scO(\Gamma_K)}\biggl(\frac{\rH^1_l(K_{\overline{v}})^\eta}{({\rm res}_{\overline{v}}(\mathcal{BF}^{l,\mathbf{h}_v,\eta}_{\rm int}))}\biggr)=(L_{1,2}(\mathbf{h}_v,f)^\eta)\cdot\frac{\xi_{\eta,2}}{\xi_{\eta,l'}}
\end{equation}
in $\Lambda_{\scO}(\Gamma_K)$,  
%\otimes_{\scO}L$, and even in $\Lambda_{\scO}(\Gamma_K)$ under hypothesis (v) of Theorem~\ref{thm:CLW}, since in that case from Proposition~\ref{prop:Gr-BDP} and \cite[Thm.~B]{hsieh} 
%we can remove the ambiguity by powers of $\varpi$ coming from \cite[Prop.~2.17]{BL-non-ord}.  
and we note that since $\xi_{\eta,1}=1$ for all $\eta$ (see Remark~\ref{rem:xi}) we have the equality  $\xi_{\eta,2}/\xi_{\eta,l'}=\xi_{\eta,l}$. Thus from \eqref{eq:ERLii} we see that \eqref{eq:CLW+HT} amounts to the divisibility
\[
(\varpi^s\cdot\xi_{\eta,l})\cdot{\rm char}_{\Lambda_{\scO}(\Gamma_K)}\bigl(X_{{\rm str},{\rm rel}}(f_{\eta}/K_\infty)\bigr)\subset
(\varpi^n)\cdot{\rm char}_{\Lambda_\scO(\Gamma_K)}\biggl(\frac{\rH^1_l(K_{\overline{v}})^\eta}{({\rm res}_{\overline{v}}(\mathcal{BF}^{l,\mathbf{h}_v,\eta}_{\rm int}))}\biggr),
\]
which together with \eqref{eq:PT:ll'} and multiplicativity of characteristic ideals yields the result.
%\[
%(\varpi^s\cdot\xi_{\eta,l})\cdot{\rm char}_{\Lambda_{\scO}(\Gamma_K)}\bigl(X_{{\rm str},l}(f_{\eta}/K_\infty)\bigr)\subset
%(\varpi^n)\cdot{\rm char}_{\Lambda_{\scO}(\Gamma_K)}\biggl(\frac{\rH^1_{{\rm rel},l}(K,\mathbf{T}_K^\eta)}{(\mathcal{BF}_{\rm int}^{l,\mathbf{h}_v,\eta})}\biggr)
%\]
%whence the result.
%
%Since arguing as in \cite[Lem.~2.1]{PR-exp} we see that $[\varphi(\omega_f),\omega_f]_{\rm dR}$ is a $p$-adic unit, combining \eqref{eq:PT:ll'}, \eqref{eq:CLW+HT}, and \eqref{eq:ERLii} yields the result. %using that by \cite[Cor.~6.10.6]{LLZ} the constant $\alpha[\varphi(v_1),v_1]/r_f$ is divisible by $c_f$.
\end{proof}

\begin{cor}\label{cor:ii-IMC}
In the setting of and under hypotheses {\rm (i)}--{\rm (iv)} of Theorem~\ref{thm:CLW}, if $l\in\{1,2\}$ is such that  condition \eqref{eq:Lp-nonzero} holds, then  
%$X_{l,l}(f/K)$ is  $\Lambda_{\scO}(\Gamma_K)$-torsion, and
\begin{equation}\label{eq:div-ii}
(\xi_{\eta,l})^2\cdot{\rm char}_{\Lambda_{\scO}(\Gamma_K)}\bigl(X_{l,l}(f_\eta/K_\infty)\bigr)\subset\bigl(\varpi^{-s}\cdot\mathcal{L}_{l,l}^{\rm (I)}(f/K)^\eta\bigr)\nonumber
\end{equation}
as ideals in  $\Lambda_{\scO}(\Gamma_K)\otimes_{\Lambda_{\scO}(\Gamma_K)}{\rm Frac}(\Lambda_{\scO}(\Gamma))$. Moreover, if either:
\begin{itemize}
\item[(v)] hypothesis {\rm (v)} in Theorem~\ref{thm:CLW} holds, or
\item[(v)'] $k=2$ and $\bar{\rho}_f$ is ramified at every prime $\ell\vert N^-$, 
\end{itemize}
then the divisibility \eqref{eq:div-ii} holds in $\Lambda_{\scO^{}}(\Gamma_K)$.
\end{cor}

\begin{proof}
With similar abbreviated notations as in the proof of Corollary~\ref{cor:BF-IMC},
Poitou--Tate duality gives rise to the exact sequence
\begin{equation}\label{eq:PT-res-v}
\begin{aligned}
0\rightarrow\rH^1_{l,l}(K,\mathbf{T}_K^\eta)\longrightarrow&\rH^1_{{\rm rel},l}(K,\mathbf{T}_K^\eta)\xrightarrow{{\rm res}_v}\frac{\rH^1(K_v)^\eta}{\rH^1_l(K_v)^\eta}\\
\quad&\longrightarrow X_{l,l}(f_\eta/K_\infty)\longrightarrow X_{{\rm str},l}(f_\eta/K_\infty)\longrightarrow 0.
\end{aligned}
\end{equation}

As before, it suffices to consider the case where $X_{l,l}(f_\eta/K_\infty)$ is $\Lambda_{\scO}(\Gamma_K)$-torsion. %(otherwise, the desired divisibility is obvious). 
Then $X_{{\rm str},l}(f_\eta/K_\infty)$ is also $\Lambda_{\scO}(\Gamma_K)$-torsion, and from \eqref{eq:PT:i-int-j} it follows that $\rH^1_{{\rm rel},l}(K,\mathbf{T}_K^\eta)$ has $\Lambda_{\scO}(\Gamma_K)$-rank one. By Theorem~\ref{thm:ERL-I} and the nonvanishing of $\mathcal{L}_{l,l}^{\rm (I)}(f/K)^\eta$ (by our assumption on $l$), it follows that $\rH^1_{l,l}(K,\mathbf{T}_K^\eta)=0$, and so after quotienting out by the image of $\mathcal{BF}_{\rm int}^{l,\mathbf{h}_v,\eta}$ the exact sequence \eqref{eq:PT-res-v} reduces to
\begin{equation}\label{eq:PT-rev-v-bis}
0\longrightarrow\frac{\rH^1_{{\rm rel},l}(K,\mathbf{T}_K^\eta)}{(\mathcal{BF}_{\rm int}^{l,\mathbf{h}_v,\eta})}\longrightarrow\frac{\rH^1(K_v)^\eta/\rH^1_l(K_v)^\eta}{({\rm res}_v(\mathcal{BF}_{\rm int}^{l,\mathbf{h}_v,\eta}))}\longrightarrow X_{l,l}(f_\eta/K_\infty)\longrightarrow X_{{\rm str},l}(f_\eta/K_\infty)\longrightarrow 0.
\end{equation}
Since by Proposition~\ref{prop:image-Col-2var} and Theorem~\ref{thm:ERL-I} we have the equality
\[
{\rm char}_{\Lambda_{\scO}(\Gamma_K)}\biggl(\frac{\rH^1(K_v)^\eta/\rH^1_l(K_v)^\eta}{({\rm res}_v(\mathcal{BF}_{\rm int}^{l,\mathbf{h}_v,\eta}))}\biggr)%={\rm char}_{\Lambda_K}\biggl(\frac{{\rm Image}(\mathscr{C}_{v,l})}{(\mathcal{L}_{l,l}(f/K))}\biggr)
=\frac{(\mathcal{L}^{\rm (I)}_{l,l}(f/K)^\eta)}{(\xi_{\eta,l})},
\]
taking characteristic ideals in \eqref{eq:PT-rev-v-bis}, the divisibility 
\begin{equation}\label{eq:div-ii}
(\xi_{\eta,l})^2\cdot{\rm char}_{\Lambda_{\scO}(\Gamma_K)}\bigl(X_{l,l}(f_\eta/K_\infty)\bigr)\subset\bigl(\varpi^{-s}\cdot\mathcal{L}_{l,l}^{\rm (I)}(f/K)^\eta\bigr)
\end{equation}
as ideals in  $\Lambda_{\scO}(\Gamma_K)\otimes_{\Lambda_{\scO}(\Gamma_K)}{\rm Frac}(\Lambda_{\scO}(\Gamma))$, and even $\Lambda_{\scO}(\Gamma_K)$ under the additional hypothesis (v) of Theorem~\ref{thm:CLW}, follows from Corollary~\ref{cor:BF-IMC}. 

Finally, it remains to show that that ambiguity by primes from $\Lambda_{\scO}(\Gamma)$ is \eqref{eq:div-ii} can also be removed under the additional hypothesis (v)' in the statement. For this, we first argue the inclusion
\begin{equation}\label{eq:Lii-integral}
\varpi^{-s}\cdot\mathcal{L}_{l,l}^{\rm (I)}(f/K)^\eta\in\Lambda_{\scO}(\Gamma_K).
\end{equation}
We proceed by contradiction. Suppose we had $\varpi^{-s}\cdot\mathcal{L}_{l,l}^{\rm (I)}(f/K)^\eta\in\varpi^{-t}\Lambda_{\scO}(\Gamma_K)\smallsetminus\Lambda_{\scO}(\Gamma_K)$ for some $t>0$. Then after multiplying $\varpi^{-s}\cdot\mathcal{L}_{l,l}^{\rm (I)}(f/K)^\eta$ by an appropriate product of elements of $\Lambda_{\scO}(\Gamma_K)$ supported at height one primes from $\Lambda_{\scO}(\Gamma)$, using Proposition~\ref{prop:cyc-restr} and Corollary~\ref{cor:cyc-descent} we can deduce from the divisibility \eqref{eq:div-ii} a proof of the \emph{strict} divisibility
\[
(\xi_{\eta,l})^2\cdot{\rm char}_{\Lambda}\bigl(X_l(f/\bQ(\mu_{p^\infty}))^\eta\bigr)\cdot{\rm char}_{\Lambda}\bigl(X_{l}(f\otimes\epsilon_K/\bQ(\mu_{p^\infty}))^\eta\bigr)\subsetneq\bigl(L_l(f)^\eta\cdot L_i(f\otimes\epsilon_K)^\eta\bigr),
\]
but this contradicts the divisibility obtained from  
Theorem~\ref{thm:katodiv-signed} applied to $f$ and $f\otimes\epsilon_K$, and so \eqref{eq:Lii-integral} holds. Now, by the formula for ${\rm ord}_\varpi(c_f/\eta_{f,N^-})$ in terms of Tamagawa numbers in \cite[Thm.~6.8]{pollack-weston} (see also \cite[Cor.~6.7]{kim-ota}), from Proposition~\ref{prop:ac-descent-def} and Theorem~\ref{thm:mu=0-definite} we see that if $k=2$ (in which case $\eta$ is the trivial character of $\Delta$) 
and $\bar{\rho}_f$ is ramified at every prime $\ell\vert N^-$ then 
\begin{equation}\label{eq:mu-Lii}
\mu(\varpi^{-s}\cdot
\mathcal{L}_{l,l}^{\rm (I)}(f/K)^\eta\;{\rm mod}\,I^+)=0,\nonumber
\end{equation}
and hence %that the ambiguity by height one primes of $\Lambda_{\scO}(\Gamma_K)$ that are pullback from height one primes of $\Lambda_{\scO}(\Gamma)$ in 
together with \eqref{eq:Lii-integral} we conclude that the divisibility \eqref{eq:div-ii} holds in $\Lambda_{\scO}(\Gamma_K)$.
%
%can be removed under the additional hypothesis (v) (resp. (v)') in the statement follows from \eqref{eq:mu-Lii} (resp.  the last claim in Corollary~\ref{cor:BF-IMC}).
\end{proof}

%\subsection{Eisenstein congruence divisibility over $K$, bis}

%\begin{cor}\label{cor:CLW}
%Let $f\in S_k(\Gamma_0(N))$ be a newform of squarefree level $N$ and weight $k\geq 2$, and let $p$ be an odd prime of good nonordinary reduction for $f$. Let $K$ be an imaginary quadratic field satisfying:
%\begin{itemize}
%\item[(i)] $p$ splits in $K$;
%\item[(ii)] if $N$ is odd, then $2$ splits in $K$;
%\item[(iii)] some prime factor of $N$ is nonsplit in $K$;
%\item[(iv)] $\bar{\rho}_f\vert_{G_{K}}$ is irreducible.
%\end{itemize}
%Then 
%\begin{equation}\label{eq:CLW-bis}
%{\rm char}_{\Lambda_{\scO}(\Gamma_K)}\bigl(X_{{\rm str},{\rm rel}}^{}(f/K_\infty)\bigr)\Lambda_{\scO^{\rm ur}}(\Gamma_K)\subset\bigl(\mathcal{L}_{\overline{v}}^{}(f/K)\bigr)
%\end{equation}
%in $\Lambda_{\scO^{\rm ur}}(\Gamma_K)$. 
%\end{cor}

%\begin{proof}
%Arguing similarly as in Proposition~\ref{prop:L-cyc}, we can relate the image of ${\rm Tw}_{k/2-1}(\mathcal{L}_{i,i}(f/K)^{\omega^{1-k/2}})$ under the natural projection $\Lambda_{\scO}(\Gamma_K)\rightarrow\Lambda_{\scO}(\Gamma^-)$ to the square of a signed theta element $\Theta_i(f/K)$ deduced from \cite{ChHs1}. The resulting refinement of Theorem~\ref{thm:CLW} then follows from the $\mu=0$ result of [\emph{op.\,cit.}, Thm.~C] similarly as in \cite[\S{2.4}]{pollack-weston}.
%\end{proof}

\subsection{Proof of Theorem~\ref{thmintro:kato}} 

We can now complete the proof building on Theorems~\ref{thm:katodiv-signed} and \ref{thm:CLW}.

By Ribet's level-lowering result \cite{ribet-eps}, there exists a prime $\ell\Vert N$ (recall that $N$ is assumed to be squarefree) at which $\bar{\rho}_f$ is ramified. If $k=2$, we fix such a prime $\ell$; while if $k>2$ we let $\ell\Vert N$ be a prime as in Condition~\eqref{eq:rtn}. Choose an imaginary quadratic field $K$ satisfying: 
\begin{itemize}
\item[(a)] $\ell$ is inert (resp. ramified) in $K$ if $k=2$ (resp. $k>2$);
\item[(b)] every prime dividing $N/\ell$ splits in $K$;
%\item[(b)] \eqref{eq:gen-H} holds;
%\item[(b)] every prime dividing $N/\ell$ splits in $K$;
\item[(c)] if $N$ is odd, then $2$ splits in $K$;
%\item[(c)] every prime dividing $N/\ell$ splits in $K$;
\item[(d)] $p$ splits in $K$;
\item[(e)] $L(f\otimes\epsilon_K\eta^{-1},1)\neq 0$, where $\eta=\omega^{1-k/2}$.
\end{itemize}
%If $N=\ell$ is prime, instead of (a) and (b) we require that $\ell$ be ramified (resp. inert) in $K$ in \emph{Case~(2b)} (resp. \emph{(2a)}). In all cases, 
The existence of infinitely many $K$ satisfying these properties is ensured by \cite[Thm.~B]{FH} (in the case $k=2$ where the nonvanishing $L(f\otimes\epsilon_K\eta^{-1},1)$ is non-trivial; for $k>2$ condition (e) follows from \cite[Prop.~2]{shimuraCPAM}). 

Put $\Lambda_K=\Lambda_{\scO}(\Gamma_K)$, $\Lambda_K^{\rm ur}=\Lambda_{\scO^{\rm ur}}(\Gamma_K)$, and $\Lambda=\Lambda_\scO(\Gamma)$ for the ease of notation. %With this choice of $K$, 
%by Theorem~\ref{thm:CLW} 
%we have the  divisibility 
%\[
%\bigl(\mathcal{L}_{\overline{v}}(f/K)^\eta\bigr)\supset
%{\rm char}_{\Lambda_K}\bigl(X_{{\rm str},{\rm rel}}(f_\eta/K_\infty)\bigr)\Lambda^{\rm ur}_K,
%\]
%in $\Lambda_K^{\rm ur}$. (Indeed, conditions (i)--(iii) in Theorem~\ref{thm:CLW} are embedded in the choice of $K$, while (iv) follows from  the irreducibility of $\bar{\rho}_f\vert_{G_{\bQ_p}}$ and the fact that $p$ splits in $K$; and (v) from Condition~\eqref{eq:rtn} and 
%the discussion in \cite[p.\,336]{skinner}.) 
By Corollary~\ref{cor:L-cyc}, the nonvanishing of $L_i(f)^\eta$ and condition (e) imply that $\mathcal{L}_{i,i}^{\rm (I)}(f/K)^{\eta}\neq 0$, and hence by Corollary~\ref{cor:ii-IMC} we have the divisibility 
\[
\bigl(\varpi^{-s}\cdot\mathcal{L}^{\rm (I)}_{i,i}(f/K)^\eta\bigr)\supset(\xi_{\eta,i})^2\cdot{\rm char}_{\Lambda_K}\bigl(X_{i,i}(f_\eta/K_\infty)\bigr)
\]
in $\Lambda_K$. (Note that hypotheses (i)--(iii) from Theorem~\ref{thm:CLW} are embedded in the choice of $K$, while (iv) follows from  the irreducibility of $\bar{\rho}_f\vert_{G_{\bQ_p}}$ and the fact that $p$ splits in $K$; and the additional hypothesis (v) (resp. (v)') in Corollary~\ref{cor:ii-IMC} follows from the above  choice of $K$ (resp. Condition~\eqref{eq:rtn} and 
the discussion in \cite[p.\,336]{skinner}).) Together with Proposition~\ref{prop:L-cyc} and Corollary~\ref{cor:cyc-descent}, it follows that
\begin{equation}\label{eq:combine}
\begin{aligned}
\bigl(L_i(f)^\eta\cdot L_i(f\otimes\epsilon_K)^\eta\bigr)&=
\bigl(\varpi^{-s}\cdot\mathcal{L}_{i,i}^{\rm (I)}(f/K)^\eta\;{\rm mod}\,I^-\bigr)\\
&\supset(\xi_{\eta,i})^2\cdot{\rm char}_\Lambda\bigl(X_{i,i}(f_\eta/K_\infty)\;{\rm mod}\,I^-\bigr)\\
&\supset
(\xi_{\eta,i})^2\cdot{\rm char}_{\Lambda}\bigl(X_i(f/\bQ(\mu_{p^\infty}))^\eta\bigr)\cdot{\rm char}_{\Lambda}\bigl(X_{i}(f\otimes\epsilon_K/\bQ(\mu_{p^\infty}))^\eta\bigr).
\end{aligned}
\end{equation}
Now we can conclude similarly as in \cite[Thm.~3.29]{SU}: From Theorem~\ref{thm:katodiv-signed} we have the divisibilities
\begin{align*}
(\xi_{\eta,i})\cdot{\rm char}_\Lambda\bigl(X_i(f/\bQ(\mu_{p^\infty}))^\eta\bigr)&\supset\bigl(L_{i}(f)^\eta\bigr),\\%\quad\quad
(\xi_{\eta,i})\cdot{\rm char}_\Lambda\bigl(X_i(f\otimes\epsilon_K/\bQ(\mu_{p^\infty}))^\eta\bigr)&\supset\bigl(L_{i}(f\otimes\epsilon_K)^\eta\bigr)
\end{align*}
in $\Lambda$. If the first of these inclusions was not an equality, they would give
\[
(\xi_{\eta,i})^2\cdot{\rm char}_{\Lambda}\bigl(X_i(f/\bQ(\mu_{p^\infty}))^\eta\bigr)\cdot{\rm char}_{\Lambda}\bigl(X_{i}(f\otimes\epsilon_K/\bQ(\mu_{p^\infty}))^\eta\bigr)\supsetneq\bigl(L_i(f)^\eta\cdot L_i(f\otimes\epsilon_K)^\eta\bigr),
\]
contradicting \eqref{eq:combine}. Thus we deduce that equality holds, 
%$(\xi_{\mathds{1},i})\cdot{\rm char}_\Lambda\bigl(X_i(f)\bigr)=\bigl(L_{i}(f)\bigr)$, 
concluding the proof of Theorem~\ref{thm:signed-Q}, and hence of 
%Theorem~\ref{thm:kato} and 
Theorem~\ref{thmintro:kato} in the Introduction.

%\begin{rem}\label{rem:ram}
%It should be possible to prove a version of Theorem~\ref{thm:kato} with condition \eqref{eq:mult-indef} replaced by the weaker assumption that
%\begin{equation}\label{eq:mult-def}
%\textrm{either $a_p(f)=0$ or there exists a prime $\ell\Vert N$}.\tag{mult'}
%\end{equation}
%Indeed, in the case $a_p(f)=0$ one can proceed as in \cite{BSTW}, while for $a_p(f)\neq 0$ one can proceed similarly after showing that $\mu(\mathscr{L}_p^{\circ}(f/L)^-)=0$, for an auxiliary imaginary quadratic field $L$ chosen as in \cite[\S{10.3}]{BSTW}, building on 
%the $\mu=0$ result of \cite{ChHs1} and an extension of \cite{CL} to Coleman families. Note however, that our proof of Theorem~\ref{thm:kato} will rely on the choice of an auxiliary imaginary quadratic field different from that in \cite[\S{10.3}]{BSTW} (in fact, in our proof the non-vanishing results of \cite{hsieh} will replace those of  \cite{vatsal} and generalizations \cite{pollack-weston,ChHs1}). 
%\end{rem}

%\subsection{Proof of Corollaries~\ref{corintro:BK} and 
%\ref{corintro:BSD}}

%\subsection{Application to the Tamagawa Number Conjecture} 

%In this section we deduce applications of our main result to the $p$-part of the Tamgawa Number Conjecture \cite{BF-doc-math} and the $p$-part of the Birch--Swinnerton-Dyer formula  for non-ordinary primes $p$ in analytic rank zero.

%\subsubsection{On the Tamagawa Number Conjecture}
\subsection{Proof of Theorem~\ref{thmintro:TNC} and Corollary~\ref{corintro:BSD-E}}

%\begin{thm}\label{thm:pTNC}
%Let $0\leq j\leq k-2$, and if $j=k/2-1$ suppose $L(f,k/2)\neq 0$. Then %${\rm Sel}(\bQ,A)$ is finite, with
%[
%\left[\mathscr{O}\colon\frac{L(f,1+j)}{(2\pi i)^j\Omega}\right]=\#{\rm Sel}(\bQ,A)\cdot\prod_{\ell\mid N}c_\ell(f).
%\]
%In other words, the $p$-part of the Tamagawa Number Conjecture for $f$ holds.
%\end{thm}

%\begin{proof}
\subsubsection{$p$-part of the Tamagawa Number Conjecture}

For the proof of Theorem~\ref{thmintro:TNC}, we essentially combine Theorem~\ref{thmintro:kato} with some of the calculations in \cite[\S{14}]{Kato295}, using a reformulation of Kato's main conjecture in terms of determinants of arithmetic complexes. Put 
\[
T=T_f(k/2),\quad A=T_f^\vee(1-k/2),\quad \Lambda=\Lambda_{\scO}(\Gamma),\quad\mathbf{T}=T\otimes_{\scO}\Lambda^\iota
\]
for the ease of notation. Let $S$ be the set of primes  dividing $Np$, and let $G_{\bQ,S}={\rm Gal}(\bQ^S/\bQ)$ denote the Galois group of the maximal extension of $\bQ$ unramified outside $S$. 

Let $\mathfrak{Sel}(\bQ,T)$ be the subgroup of $\rH^1(G_{\bQ,S},T)$ consisting of classes  crystalline at $p$ and unramified at the primes $\ell\vert N$, and let $\mathfrak{Sel}(\bQ,A)$ denote the Selmer group group dual to $\mathfrak{Sel}(\bQ,T)$. Poitou--Tate duality then gives rise to the exact sequence
\begin{equation}\label{eq:PT-9}
\begin{aligned}
0\rightarrow\mathfrak{Sel}(\bQ,T)\longrightarrow\rH^1(G_{\bQ,S},T)\longrightarrow\frac{\rH^1(\bQ_p,T)}{\rH^1_f(\bQ_p,T)}&\oplus\bigoplus_{\ell\mid N}\frac{\rH^1(\bQ_\ell,T)}{\rH^1_{\rm unr}(\bQ_\ell,T)}\longrightarrow\mathfrak{Sel}(\bQ,A)^\vee\\
&\quad\longrightarrow\rH^2(G_{\bQ,S},T)\longrightarrow\bigoplus_{\ell\mid Np}\rH^2(\bQ_\ell,T)\longrightarrow 0.
\end{aligned}
\end{equation}
It is easy to see that ${\rm Sel}(\bQ,A)$ is contained in $\mathfrak{Sel}(\bQ,A)$ with finite index (see \eqref{eq:def-Sel} below); since the former is finite by \cite[Thm.~14.2]{Kato295} and $\rH^1(G_{\bQ,S},T)$ is torsion-free by the irreducibility of $\bar{\rho}_f$, it follows that $\mathfrak{Sel}(\bQ,T)=0$. 

For a prime $\ell\nmid p$, let $P_\ell(A,\ell^{-s})$, where $P_\ell(A,X)=(1-\alpha_\ell X)(1-\beta_\ell X)$ (with one of both of $\alpha_\ell, \beta_\ell$ possibly zero), be the Euler factor of $L(f,s)$ at $\ell$, and put
\begin{equation}\label{eq:zeta-def}
\mathbf{z}_{f}^{\{N\}}:=\mathbf{z}_{f,k/2}^{\rm int}\cdot\prod_{\ell\mid N}\mathcal{P}_\ell\in\rH^1(G_{\bQ,S},\mathbf{T}),
%\quad\textrm{
%where\; $\mathcal{P}_\ell:=(1-\alpha_\ell\ell^{-1}\gamma_\ell)(1-\beta_\ell\ell^{-1}\gamma_\ell)\in\Lambda$,}
\end{equation}
where $\mathbf{z}^{\rm int}_{k/2}$ denotes the image of $\mathbf{z}_{f,k-1}^{{\rm int}}$ under the composite map 
\begin{align*}
\rH^1_{\rm Iw}(\bQ(\mu_{p^\infty}),T_f(k-1))
\xrightarrow{\otimes(\zeta_{p^n})_{n\geq 0}^{\otimes(1-k/2)}}&\rH^1_{\rm Iw}(\bQ(\mu_{p^\infty}),T_f(k/2))\\
\longrightarrow
&\rH^1_{\rm Iw}(\bQ(\mu_{p^\infty}),T_f(k/2))^\Delta\cong\rH^1(G_{\bQ,S},\mathbf{T}),
\end{align*}
and $\mathcal{P}_\ell:=(1-\alpha_\ell\ell^{-k/2}\gamma_\ell)(1-\beta_\ell\ell^{-k/2}\gamma_\ell)\in\Lambda$ with $\gamma_\ell\in\Gamma$ a Frobenius element at $\ell$. 

For the primes $\ell\vert N$ (hence coprime to $p$) we have 
\[
\rH^1_f(\bQ_\ell,T)=\rH^1(\bQ_\ell,T)_{\rm tors}=\rH^1(\bQ_\ell,T),
\]
and therefore quotienting out in \eqref{eq:PT-9} by the image of $\mathbf{z}_{f,0}^{\{N\}}$ we obtain
\begin{equation}\label{eq:PT-multiply}
\#\mathfrak{Sel}(\bQ,A)=\frac{\#\rH^2(G_{\bQ,S},T)}{\left[\rH^1(G_{\bQ,S},T):\mathbf{z}_{f,0}^{\{N\}}\right]}\cdot\frac{\left[\rH^1_{/f}(\bQ_p,T):\mathbf{z}_{f,0}^{\{N\}}\right]}{\#\rH^2(\bQ_p,T)}\cdot\prod_{\ell\mid N}\frac{\left[\rH^1_f(\bQ_\ell,T):\rH^1_{\rm unr}(\bQ_\ell,T)\right]}{\#\rH^2(\bQ_\ell,T)},
\end{equation}
where $\rH^1_{/f}(\bQ_p,T)$ denotes the quotient $\rH^1(\bQ_p,T)$ by $\rH^1_f(\bQ_p,T)$. Note that Poitou--Tate duality also gives
\begin{equation}\label{eq:def-Sel}
0\longrightarrow{\rm Sel}(\bQ,A)\longrightarrow\mathfrak{Sel}(\bQ,A)\longrightarrow\bigoplus_{\ell\mid N}\rH^1_{\rm unr}(\bQ_\ell,A)\longrightarrow{\rm Sel}(\bQ,T)^\vee,\nonumber
\end{equation}
where we used that $\rH^1_{\rm unr}(\bQ_\ell,A)=0$ for all primes $\ell\vert N$. 

By \cite[Thm.~14.2]{Kato295}, the nonvanishing of $L(f,k/2)$ implies $\#{\rm Sel}(\bQ,A)<\infty$, and so ${\rm Sel}(\bQ,T)=0$. Thus 
\begin{equation}\label{eq:frakS-to-S}
\#\mathfrak{Sel}(\bQ,A)=\#{\rm Sel}(\bQ,A)\cdot\prod_{\ell\mid N}\#\rH^1_{\rm unr}(\bQ_\ell,A).
\end{equation}
As explained in \cite[\S{2.2}]{c-sano}, as a consequence of \cite[Thm.~12.4]{Kato295} there is a canonical isomorphism
\begin{equation}\label{eq:Leop}
Q(\Lambda)\otimes_{\Lambda}{\rm det}_{\Lambda}^{-1}\mathbf{R}\Gamma(G_{\bQ,S},\mathbf{T})\cong Q(\Lambda)\otimes_{\Lambda}\rH^1(G_{\bQ,S},\mathbf{T}),
\end{equation}
where $Q(\Lambda)$ denotes the field of fractions of $\Lambda$.  Moreover, as explained in \cite[Prop.~2.2.3]{c-sano} and the references therein  (see esp. \S{2.3.2} and Proposition~3.10 in \cite{kataoka-sano}), Conjecture~\ref{conj:kato} for $\eta=\omega^{1-k/2}$ is then equivalent\footnote{For $k>2$, using the twisting \cite[Lem.~VI.1.2]{Rubin-ES} with $\rho=\langle\varepsilon\rangle^{1-k/2}:\Gamma\rightarrow\bZ_p^\times$} to the following: The inverse image %$\mathfrak{z}_{\bQ_\infty}$ 
of $\mathbf{z}_{f}^{\{N\}}$ under the canonical isomorphism \eqref{eq:Leop} is a $\Lambda$-basis 
\begin{equation}\label{eq:zeta-basis}
\mathfrak{z}_{\bQ_\infty}\in{\rm det}_{\Lambda}^{-1}\mathbf{R}\Gamma(G_{\bQ,S},\mathbf{T}).
\end{equation}

Thus letting $\nu:\Lambda\rightarrow\scO$ be the augmentation map, by Theorem~\ref{thmintro:kato} the image of $\mathfrak{z}_{\bQ}$ of \eqref{eq:zeta-basis} under the isomorphism
\[
{\rm det}_{\Lambda}^{-1}\mathbf{R}\Gamma(G_{\bQ,S},\mathbf{T})\otimes_{\Lambda,\nu}\scO\cong{\rm det}_{\scO}^{-1}\mathbf{R}\Gamma(G_{\bQ,S},T)
\]
from \cite[Prop.~1.6.5(3)]{fukaya-kato-AMS} gives an $\scO$-basis $\mathfrak{z}_{\bQ}\in{\rm det}_{\scO}^{-1}\mathbf{R}\Gamma(G_{\bQ,S},T)$; equivalently, the image $\mathbf{z}_{f,0}^{\{N\}}$ of \eqref{eq:zeta-def} under the natural projection $\rH^1(G_{\bQ,S},\mathbf{T})\rightarrow\rH^1(G_{\bQ,S},T)$ satisfies
\begin{equation}\label{eq:TNC}
\#\rH^2(G_{\bQ,S},T)=\left[\rH^1(G_{\bQ,S},T)\colon\mathbf{z}_{f,0}^{\{N\}}\right].
\end{equation}
(cf. \cite[Prop.~2.6]{BKS-kato-I}). 

On the other hand, the same computation as in \cite[Prop.~14.21]{Kato295} (cf. \eqref{eq:kato-ERL}) 
%(using Kato's explicit reciprocity law \eqref{eq:kato-ERL}) 
gives
\begin{equation}\label{eq:kato-FL}
\frac{\left[\rH^1_{/f}(\bQ_p,T):\mathbf{z}_0^{\{N\}}\right]}{\#\rH^2(\bQ_p,T)}=\left[\scO\colon\frac{L_{\{N\}}(f,k/2)}{(-2\pi i)^{k/2}\cdot\Omega_{\omega_f}^\pm}\right]={\rm Eul}_N(f,k/2)\cdot\left[\scO\colon\frac{L(f,k/2)}{(-2\pi i)^{k/2}\cdot\Omega_{\omega_f}^\pm}\right],
\end{equation}
where $\pm=(-1)^{k/2-1}$. Hence combining \eqref{eq:PT-multiply}, \eqref{eq:frakS-to-S}, \eqref{eq:TNC}, \eqref{eq:kato-FL} and noting the equalities
\[
%\prod_{\ell\mid N}
\#\rH^2(\bQ_\ell,T)=%\prod_{\ell\mid N}
\#\rH^0(\bQ_\ell,A)={\rm Eul}_\ell(f,k/2)\cdot%\prod_{\ell\mid N}
\left[\rH^1_f(\bQ_\ell,T)\colon\rH^1_{\rm unr}(\bQ_\ell,T)\right]
\]
for every prime $\ell\vert N$, this concludes the proof of Theorem~\ref{thmintro:TNC}.
%\end{proof}

\subsubsection{$p$-part of the Birch--Swinnerton-Dyer formula}

After a period comparison, Corollary~\ref{corintro:BSD-E} is just a special case of Theorem~\ref{thmintro:TNC}. More precisely, letting $f\in S_2(\Gamma_0(N))$ be the newform associated to $E$ by modularity \cite{BCDT}, the $p$-regularity hypothesis is known to hold by \cite{coleman-edixhoven} and by Theorem~\ref{thmintro:TNC} the nonvanishing of $L(E,1)=L(f,1)$ gives $\#E(\bQ)<\infty$, 
\[
\#{\rm Sel}(\bQ,T_f^\vee)=\#{\rm Sel}_{p^\infty}(E/\bQ)=\#\Sha(E/\bQ)[p^\infty]<\infty,
\]
and
\[
\left[\bZ_p\colon\frac{L(E,1)}{(-2\pi i)\cdot\Omega_{\omega_f}^\pm}\right]=\#\Sha(E/\bQ)[p^\infty]\cdot\prod_{\ell\mid N}c_\ell^{(p)}(E),
\]
where $c_\ell^{(p)}(E)$ denotes the $p$-part of the Tamagawa number of $E$ at $\ell$ (which agrees with ${\rm Tam}_\ell(T_f(1))=\#\rH^1_f(\bQ_\ell,T_f(1))$ by e.g. the discussion in \cite[p.\,74]{greenberg-cetraro}). It remains to show that the period $(-2\pi i)\cdot\Omega_{\omega_f}^+$ agrees up to a $p$-adic unit with the positive N\'{e}ron period %$\Omega_E\in\mathbb{R}_{>0}$, 
$\Omega_E$ of $E$. As shown in \cite[Prop.~3.1]{GV}, we have the equality up to a $p$-adic unit 
\[
(-2\pi i)\cdot\Omega_{\omega_f}^+\,\sim_p\,\Omega_{E^{\rm opt}},
\]
where $E^{\rm opt}/\bQ$ is the optimal curve in the sense of \cite[Prop.~1.4]{stevens-invmath} in the $\bQ$-isogeny class associated to $f$ (note that the $p$-ordinarity assumption in \cite[Prop.~3.1]{GV} is not needed for the argument, but only for the definition of $\delta_f^\pm$ in \emph{op.\,cit.} which for our application can be taken to  $\delta_{\cO}^\pm$ from $\S\ref{subsec:integral-normalisation}$). %; see also the discussion in \cite[\S{3.5}]{vatsal-integral}). 
Since the irreducibility of $\bar{\rho}_f\cong E[p]$ implies that $\Omega_{E^{\rm opt}}$ is a $p$-adic unit multiple of $\Omega_E$, this concludes the proof of Corollary~\ref{corintro:BSD-E}.

\bibliographystyle{amsalpha}
\bibliography{IMC-zeta-refs}

\providecommand{\bysame}{\leavevmode\hbox to3em{\hrulefill}\thinspace}
\providecommand{\MR}{\relax\ifhmode\unskip\space\fi MR }
\providecommand{\MRhref}[2]{%
  \href{http://www.ams.org/mathscinet-getitem?mr=#1}{#2}
}
\providecommand{\href}[2]{#2}
\begin{thebibliography}{BSTW24}

\bibitem[AI21]{AI-triple}
Fabrizio Andreatta and Adrian Iovita, \emph{Triple product {$p$}-adic
  {$L$}-functions associated to finite slope {$p$}-adic families of modular
  forms}, Duke Math. J. \textbf{170} (2021), no.~9, 1989--2083.

\bibitem[BBL24]{BBL-adv}
Ashay Burungale, K\^az\i~m{} B\"uy\"ukboduk, and Antonio Lei,
  \emph{Anticyclotomic {I}wasawa theory of abelian varieties of {$\rm
  GL_2$}-type at non-ordinary primes}, Adv. Math. \textbf{439} (2024), Paper
  No. 109465, 63.

\bibitem[BC08]{berger-colmez-familles}
Laurent Berger and Pierre Colmez, \emph{Familles de repr\'esentations de de
  {R}ham et monodromie {$p$}-adique}, no. 319, 2008, Repr\'esentations
  $p$-adiques de groupes $p$-adiques. I. Repr\'esentations galoisiennes et
  $(\phi,\Gamma)$-modules, pp.~303--337.

\bibitem[BCDT01]{BCDT}
Christophe Breuil, Brian Conrad, Fred Diamond, and Richard Taylor, \emph{On the
  modularity of elliptic curves over {$\bold Q$}: wild 3-adic exercises}, J.
  Amer. Math. Soc. \textbf{14} (2001), no.~4, 843--939 (electronic).

\bibitem[BCGS26]{BCGS}
Ashay Burungale, Francesc Castella, Giada Grossi, and Christopher Skinner,
  \emph{Non-vanishing of {K}olyvagin systems and {I}wasawa theory}, Camb. J.
  Math. \textbf{14} (2026), no.~2, 285--348.

\bibitem[BCS25]{BCS}
Ashay Burungale, Francesc Castella, and Christopher Skinner, \emph{Base change
  and {I}wasawa main conjectures for {${\rm GL}_2$}}, Int. Math. Res. Not. IMRN
  (2025), no.~8, Paper No. rnaf082, 15.

\bibitem[BD96]{BDmumford-tate}
M.~Bertolini and H.~Darmon, \emph{Heegner points on {M}umford-{T}ate curves},
  Invent. Math. \textbf{126} (1996), no.~3, 413--456.

\bibitem[BDI10]{BDI}
Massimo Bertolini, Henri Darmon, and Adrian Iovita, \emph{Families of
  automorphic forms on definite quaternion algebras and {T}eitelbaum's
  conjecture}, Ast\'erisque (2010), no.~331, 29--64.

\bibitem[BDP13]{bdp1}
Massimo Bertolini, Henri Darmon, and Kartik Prasanna, \emph{Generalized
  {H}eegner cycles and {$p$}-adic {R}ankin {$L$}-series}, Duke Math. J.
  \textbf{162} (2013), no.~6, 1033--1148, With an appendix by Brian Conrad.

\bibitem[BDP22]{betina-dimitrov-pozzi}
Adel Betina, Mladen Dimitrov, and Alice Pozzi, \emph{On the failure of
  {G}orensteinness at weight 1 {E}isenstein points of the eigencurve}, Amer. J.
  Math. \textbf{144} (2022), no.~1, 227--265.

\bibitem[Bei84]{beilinson-higher}
A.~A. Beilinson, \emph{Higher regulators and values of {$L$}-functions},
  Current problems in mathematics, {V}ol. 24, Itogi Nauki i Tekhniki, Akad.
  Nauk SSSR, Vsesoyuz. Inst. Nauchn. i Tekhn. Inform., Moscow, 1984,
  pp.~181--238.

\bibitem[Bel12]{bellaiche-interp}
Jo\"el Bella\"iche, \emph{Critical {$p$}-adic {$L$}-functions}, Invent. Math.
  \textbf{189} (2012), no.~1, 1--60.

\bibitem[Ber03]{berger-explicit}
Laurent Berger, \emph{Bloch and {K}ato's exponential map: three explicit
  formulas}, Doc. Math. (2003), no.~Extra Vol., 99--129, Kazuya Kato's fiftieth
  birthday.

\bibitem[Ber04]{berger-limit}
\bysame, \emph{Limites de repr\'esentations cristallines}, Compos. Math.
  \textbf{140} (2004), no.~6, 1473--1498.

\bibitem[BF01]{BF-doc-math}
D.~Burns and M.~Flach, \emph{Tamagawa numbers for motives with
  (non-commutative) coefficients}, Doc. Math. \textbf{6} (2001), 501--570.

\bibitem[BK90]{BK}
Spencer Bloch and Kazuya Kato, \emph{{$L$}-functions and {T}amagawa numbers of
  motives}, The {G}rothendieck {F}estschrift, {V}ol.\ {I}, Progr. Math.,
  vol.~86, Birkh\"auser Boston, Boston, MA, 1990, pp.~333--400.

\bibitem[BKS24]{BKS-kato-I}
David Burns, Masato Kurihara, and Takamichi Sano, \emph{On derivatives of
  {K}ato's {E}uler system for elliptic curves}, J. Math. Soc. Japan \textbf{76}
  (2024), no.~3, 855--919.

\bibitem[BL21]{BL-non-ord}
K{\^a}z{\i}m B\"{u}y\"{u}kboduk and Antonio Lei, \emph{Iwasawa theory of
  elliptic modular forms over imaginary quadratic fields at non-ordinary
  primes}, Int. Math. Res. Not. IMRN (2021), no.~14, 10654--10730.

\bibitem[BLLV19]{BLLV}
K{\^a}z{\i}m B\"uy\"ukboduk, Antonio Lei, David Loeffler, and Guhan Venkat,
  \emph{Iwasawa theory for {R}ankin-{S}elberg products of {$p$}-nonordinary
  eigenforms}, Algebra Number Theory \textbf{13} (2019), no.~4, 901--941.

\bibitem[BLZ04]{BLZ}
Laurent Berger, Hanfeng Li, and Hui~June Zhu, \emph{Construction of some
  families of 2-dimensional crystalline representations}, Math. Ann.
  \textbf{329} (2004), no.~2, 365--377.

\bibitem[BSTW24]{BSTW}
Ashay Burungale, Christopher Skinner, Ye~Tian, and Xin Wan, \emph{Zeta elements
  for elliptic curves and applications}, preprint,
  \href{https://arxiv.org/abs/2409.01350v2}{arXiv:2409.01350v2}.

\bibitem[Cas25]{cas-TNC}
Francesc Castella, \emph{Tamagawa number conjecture for {CM} modular forms and
  {R}ankin-{S}elberg convolutions}, Proc. Lond. Math. Soc. (3) \textbf{131}
  (2025), no.~4, Paper No. e70089, 50.

\bibitem[CC99]{CC}
Fr\'ed\'eric Cherbonnier and Pierre Colmez, \emph{Th\'eorie d'{I}wasawa des
  repr\'esentations {$p$}-adiques d'un corps local}, J. Amer. Math. Soc.
  \textbf{12} (1999), no.~1, 241--268.

\bibitem[C{\c{C}}SS18]{CCSS}
Francesc Castella, Mirela {\c{C}}iperiani, Christopher Skinner, and Florian
  Sprung, \emph{On the {I}wasawa main conjectures for modular forms at
  non-ordinary primes}, preprint,
  \href{https://arxiv.org/abs/1804.10993}{arXiv:1804.10993}.

\bibitem[CE98]{coleman-edixhoven}
Robert~F. Coleman and Bas Edixhoven, \emph{On the semi-simplicity of the
  {$U_p$}-operator on modular forms}, Math. Ann. \textbf{310} (1998), no.~1,
  119--127.

\bibitem[CGS25]{CGS}
Francesc Castella, Giada Grossi, and Christopher Skinner, \emph{Mazur's main
  conjecture at {E}isenstein primes}, Math. Ann. \textbf{393} (2025), no.~2,
  2451--2506.

\bibitem[CH18a]{cas-hsieh1}
Francesc Castella and Ming-Lun Hsieh, \emph{Heegner cycles and {$p$}-adic
  {$L$}-functions}, Math. Ann. \textbf{370} (2018), no.~1-2, 567--628.

\bibitem[CH18b]{ChHs1}
Masataka Chida and Ming-Lun Hsieh, \emph{Special values of anticyclotomic
  {$L$}-functions for modular forms}, J. Reine Angew. Math. \textbf{741}
  (2018), 87--131.

\bibitem[CL16]{cas-longo}
Francesc Castella and Matteo Longo, \emph{Big {H}eegner points and special
  values of {$L$}-series}, Ann. Math. Qu\'e. \textbf{40} (2016), no.~2,
  303--324.

\bibitem[CLW22]{CLW}
Francesc Castella, Zheng Liu, and Xin Wan, \emph{Iwasawa-{G}reenberg main
  conjecture for nonordinary modular forms and {E}isenstein congruences on
  {GU}(3,1)}, Forum Math. Sigma \textbf{10} (2022), Paper No. e110, 90.

\bibitem[CS26]{c-sano}
Francesc Castella and Takamichi Sano, \emph{On refined nonvanishing conjectures
  by {K}urihara and {K}olyvagin}, preprint,
  \href{https://arxiv.org/abs/2601.14504}{arXiv:2601.14504}.

\bibitem[CW22]{CW-adv}
Francesc Castella and Xin Wan, \emph{The {I}wasawa main conjectures for {$\rm
  GL_2$} and derivatives of {$p$}-adic {$L$}-functions}, Adv. Math.
  \textbf{400} (2022), Paper No. 108266, 45.

\bibitem[DDT94]{DDT}
Henri Darmon, Fred Diamond, and Richard Taylor, \emph{Fermat's last theorem},
  Current developments in mathematics, 1995 ({C}ambridge, {MA}), Int. Press,
  Cambridge, MA, 1994, pp.~1--154.

\bibitem[Del71]{deligne-ell-adic}
Pierre Deligne, \emph{Formes modulaires et repr\'esentations {$l$}-adiques},
  S\'eminaire {B}ourbaki. {V}ol. 1968/69: {E}xpos\'es 347--363, Lecture Notes
  in Math., vol. 175, Springer, Berlin, 1971, pp.~Exp. No. 355, 139--172.

\bibitem[DR25]{daronche}
Enrico Da~Ronche, \emph{Kolyvagin's conjecture for modular forms at
  non-ordinary primes}, preprint,
  \href{https://arxiv.org/abs/2503.09955}{arXiv:2503.09955}.

\bibitem[dS87]{de_shalit}
Ehud de~Shalit, \emph{Iwasawa theory of elliptic curves with complex
  multiplication}, Perspectives in Mathematics, vol.~3, Academic Press, Inc.,
  Boston, MA, 1987.

\bibitem[Edi92]{Edi}
Bas Edixhoven, \emph{The weight in {S}erre's conjectures on modular forms},
  Invent. Math. \textbf{109} (1992), no.~3, 563--594.

\bibitem[FH95]{FH}
Solomon Friedberg and Jeffrey Hoffstein, \emph{Nonvanishing theorems for
  automorphic {$L$}-functions on {${\rm GL}(2)$}}, Ann. of Math. (2)
  \textbf{142} (1995), no.~2, 385--423.

\bibitem[FJ95]{faltings-jordan}
Gerd Faltings and Bruce~W. Jordan, \emph{Crystalline cohomology and {${\rm
  GL}(2,{\bf Q})$}}, Israel J. Math. \textbf{90} (1995), no.~1-3, 1--66.

\bibitem[FK06]{fukaya-kato-AMS}
Takako Fukaya and Kazuya Kato, \emph{A formulation of conjectures on {$p$}-adic
  zeta functions in noncommutative {I}wasawa theory}, Proceedings of the {S}t.
  {P}etersburg {M}athematical {S}ociety. {V}ol. {XII}, Amer. Math. Soc. Transl.
  Ser. 2, vol. 219, Amer. Math. Soc., Providence, RI, 2006, pp.~1--85.

\bibitem[FL82]{FL}
Jean-Marc Fontaine and Guy Laffaille, \emph{Construction de repr\'esentations
  {$p$}-adiques}, Ann. Sci. \'Ecole Norm. Sup. (4) \textbf{15} (1982), no.~4,
  547--608.

\bibitem[FW21]{fouquet-wan}
Olivier Fouquet and Xin Wan, \emph{The {I}wasawa {M}ain {C}onjecture for
  universal families of modular motives}, preprint,
  \href{https://arxiv.org/abs/2107.13726}{arXiv:2107.13726}.

\bibitem[GLL26]{gajek-leonard-lei}
Rylan Gajek-Leonard and Antonio Lei, \emph{Mazur–{T}ate elements of
  non-ordinary modular forms with {S}erre weight larger than two}, Advances in
  Mathematics \textbf{502} (2026), 111141.

\bibitem[GPRJ25]{GPRJ}
Andrew Graham, Vincent Pilloni, and Joaqu\'in Rodrigues~Jacinto,
  \emph{{$p$}-adic interpolation of {G}auss-{M}anin connections on nearly
  overconvergent modular forms and {$p$}-adic {$L$}-functions}, Compos. Math.
  \textbf{161} (2025), no.~9, 2380--2441.

\bibitem[Gre89]{greenberg-iwasawa}
Ralph Greenberg, \emph{Iwasawa theory for {$p$}-adic representations},
  Algebraic number theory, Adv. Stud. Pure Math., vol.~17, Academic Press,
  Boston, MA, 1989, pp.~97--137.

\bibitem[Gre99]{greenberg-cetraro}
\bysame, \emph{Iwasawa theory for elliptic curves}, Arithmetic theory of
  elliptic curves ({C}etraro, 1997), Lecture Notes in Math., vol. 1716,
  Springer, Berlin, 1999, pp.~51--144.

\bibitem[GV00]{GV}
Ralph Greenberg and Vinayak Vatsal, \emph{On the {I}wasawa invariants of
  elliptic curves}, Invent. Math. \textbf{142} (2000), no.~1, 17--63.

\bibitem[Hsi14]{hsieh}
Ming-Lun Hsieh, \emph{Special values of anticyclotomic {R}ankin-{S}elberg
  {$L$}-functions}, Doc. Math. \textbf{19} (2014), 709--767.

\bibitem[HT93]{HT-ENS}
H.~Hida and J.~Tilouine, \emph{Anti-cyclotomic {K}atz {$p$}-adic
  {$L$}-functions and congruence modules}, Ann. Sci. \'Ecole Norm. Sup. (4)
  \textbf{26} (1993), no.~2, 189--259.

\bibitem[IS24]{sprung-IMC}
Florian Ito~Sprung, \emph{On {I}wasawa main conjectures for elliptic curves at
  supersingular primes: beyond the case {$a_p = 0$}}, Adv. Math. \textbf{449}
  (2024), Paper No. 109741, 47.

\bibitem[JSW17]{JSW}
Dimitar Jetchev, Christopher Skinner, and Xin Wan, \emph{The {B}irch and
  {S}winnerton-{D}yer formula for elliptic curves of analytic rank one}, Camb.
  J. Math. \textbf{5} (2017), no.~3, 369--434.

\bibitem[Kat78]{Katz49}
Nicholas~M. Katz, \emph{{$p$}-adic {$L$}-functions for {CM} fields}, Invent.
  Math. \textbf{49} (1978), no.~3, 199--297.

\bibitem[Kat93]{kato-kodai}
Kazuya Kato, \emph{Iwasawa theory and {$p$}-adic {H}odge theory}, Kodai Math.
  J. \textbf{16} (1993), no.~1, 1--31.

\bibitem[Kat04]{Kato295}
\bysame, \emph{{$p$}-adic {H}odge theory and values of zeta functions of
  modular forms}, Ast\'erisque (2004), no.~295, ix, 117--290, Cohomologies
  $p$-adiques et applications arithm{\'e}tiques. III.

\bibitem[Kim25]{kim-refined-GL2}
Chan-Ho Kim, \emph{The refined {T}amagawa {N}umber {C}onjectures of {${\rm
  GL}_2$}}, preprint,
  \href{https://arxiv.org/abs/2505.09121}{arXiv:2505.09121}, with an appendix
  in collaboration with Robert Pollack.

\bibitem[Kim26]{kim-AJM}
\bysame, \emph{The structure of {S}elmer groups and the {I}wasawa main
  conjecture for elliptic curves}, Amer. J. Math. \textbf{148} (2026), no.~1,
  79--129.

\bibitem[Kis03]{kisin-FM}
Mark Kisin, \emph{Overconvergent modular forms and the {F}ontaine-{M}azur
  conjecture}, Invent. Math. \textbf{153} (2003), no.~2, 373--454.

\bibitem[KLZ17]{KLZ2}
Guido Kings, David Loeffler, and Sarah~Livia Zerbes, \emph{Rankin-{E}isenstein
  classes and explicit reciprocity laws}, Camb. J. Math. \textbf{5} (2017),
  no.~1, 1--122.

\bibitem[KLZ20]{KLZ1}
\bysame, \emph{Rankin-{E}isenstein classes for modular forms}, Amer. J. Math.
  \textbf{142} (2020), no.~1, 79--138.

\bibitem[KO23]{kim-ota}
Chan-Ho Kim and Kazuto Ota, \emph{On the quantitative variation of congruence
  ideals and integral periods of modular forms}, Res. Math. Sci. \textbf{10}
  (2023), no.~2, Paper No. 22, 34.

\bibitem[Kob03]{kobayashi-ss}
Shin-ichi Kobayashi, \emph{Iwasawa theory for elliptic curves at supersingular
  primes}, Invent. Math. \textbf{152} (2003), no.~1, 1--36.

\bibitem[Kol91]{kolyvagin-Sel}
V.~A. Kolyvagin, \emph{On the structure of {S}elmer groups}, Math. Ann.
  \textbf{291} (1991), no.~2, 253--259.

\bibitem[KS24]{kataoka-sano}
Takenori Kataoka and Takamichi Sano, \emph{On {E}uler systems for motives and
  {H}eegner points}, J. Assoc. Math. Res. \textbf{2} (2024), no.~2, 154--208.

\bibitem[Kur02]{kurihara-invmath}
Masato Kurihara, \emph{On the {T}ate {S}hafarevich groups over cyclotomic
  fields of an elliptic curve with supersingular reduction. {I}}, Invent. Math.
  \textbf{149} (2002), no.~1, 195--224.

\bibitem[Kur14]{kurihara-TATA}
\bysame, \emph{The structure of {S}elmer groups of elliptic curves and modular
  symbols}, Iwasawa theory 2012, Contrib. Math. Comput. Sci., vol.~7, Springer,
  Heidelberg, 2014, pp.~317--356.

\bibitem[Lei11]{Lei-PhD}
Antonio Lei, \emph{Iwasawa theory for modular forms at supersingular primes},
  Compos. Math. \textbf{147} (2011), no.~3, 803--838.

\bibitem[Lei25]{lei-amburgh}
\bysame, \emph{Artin formalism for {$p$}-adic {$L$}-functions of weight-two
  modular forms at non-ordinary primes}, Abh. Math. Semin. Univ. Hambg.
  \textbf{95} (2025), no.~2, 149--160.

\bibitem[Liu15]{liu-triangulation}
R.~Liu, \emph{Triangulation of refined families}, Comment. Math. Helv.
  \textbf{90} (2015), no.~4, 831--904.

\bibitem[LLZ10]{LLZ-AJM}
Antonio Lei, David Loeffler, and Sarah~Livia Zerbes, \emph{Wach modules and
  {I}wasawa theory for modular forms}, Asian J. Math. \textbf{14} (2010),
  no.~4, 475--528.

\bibitem[LLZ11]{LLZ-ANT}
\bysame, \emph{Coleman maps and the {$p$}-adic regulator}, Algebra Number
  Theory \textbf{5} (2011), no.~8, 1095--1131.

\bibitem[LLZ14]{LLZ}
\bysame, \emph{Euler systems for {R}ankin-{S}elberg convolutions of modular
  forms}, Ann. of Math. (2) \textbf{180} (2014), no.~2, 653--771.

\bibitem[LLZ15]{LLZ-K}
\bysame, \emph{Euler systems for modular forms over imaginary quadratic
  fields}, Compos. Math. \textbf{151} (2015), no.~9, 1585--1625.

\bibitem[LLZ17]{LLZ-Sha}
\bysame, \emph{On the asymptotic growth of {B}loch-{K}ato-{S}hafarevich-{T}ate
  groups of modular forms over cyclotomic extensions}, Canad. J. Math.
  \textbf{69} (2017), no.~4, 826--850.

\bibitem[Loe18]{loeffler-note}
David Loeffler, \emph{A note on {$p$}-adic {R}ankin-{S}elberg {$L$}-functions},
  Canad. Math. Bull. \textbf{61} (2018), no.~3, 608--621.

\bibitem[LZ14]{LZ2}
David Loeffler and Sarah~Livia Zerbes, \emph{Iwasawa theory and {$p$}-adic
  {$L$}-functions over {$\Bbb{Z}_p^2$}-extensions}, Int. J. Number Theory
  \textbf{10} (2014), no.~8, 2045--2095.

\bibitem[LZ16]{LZ-Coleman}
\bysame, \emph{Rankin-{E}isenstein classes in {C}oleman families}, Res. Math.
  Sci. \textbf{3} (2016), Paper No. 29, 53.

\bibitem[Mak26]{maksoud}
Alexandre Maksoud, \emph{A canonical generator for congruence ideals of {H}ida
  families}, Res. Math. Sci. \textbf{13} (2026), no.~1, Paper No. 18, 30.

\bibitem[MTT86]{mtt}
B.~Mazur, J.~Tate, and J.~Teitelbaum, \emph{On {$p$}-adic analogues of the
  conjectures of {B}irch and {S}winnerton-{D}yer}, Invent. Math. \textbf{84}
  (1986), no.~1, 1--48.

\bibitem[Oht00]{OhtaII}
Masami Ohta, \emph{Ordinary {$p$}-adic \'etale cohomology groups attached to
  towers of elliptic modular curves. {II}}, Math. Ann. \textbf{318} (2000),
  no.~3, 557--583.

\bibitem[Pol03]{pollack}
Robert Pollack, \emph{On the {$p$}-adic {$L$}-function of a modular form at a
  supersingular prime}, Duke Math. J. \textbf{118} (2003), no.~3, 523--558.

\bibitem[PR94]{PR115}
Bernadette Perrin-Riou, \emph{Th\'eorie d'{I}wasawa des repr\'esentations
  {$p$}-adiques sur un corps local}, Invent. Math. \textbf{115} (1994), no.~1,
  81--161, With an appendix by Jean-Marc Fontaine.

\bibitem[PW11]{pollack-weston}
Robert Pollack and Tom Weston, \emph{On anticyclotomic {$\mu$}-invariants of
  modular forms}, Compos. Math. \textbf{147} (2011), no.~5, 1353--1381.

\bibitem[Rib90]{ribet-eps}
K.~A. Ribet, \emph{On modular representations of {${\rm Gal}(\overline{\bf
  Q}/{\bf Q})$} arising from modular forms}, Invent. Math. \textbf{100} (1990),
  no.~2, 431--476.

\bibitem[Roh88]{rohrlich-division}
David~E. Rohrlich, \emph{{$L$}-functions and division towers}, Math. Ann.
  \textbf{281} (1988), no.~4, 611--632.

\bibitem[Rub00]{Rubin-ES}
Karl Rubin, \emph{Euler systems}, Annals of Mathematics Studies, vol. 147,
  Princeton University Press, Princeton, NJ, 2000, Hermann Weyl Lectures. The
  Institute for Advanced Study.

\bibitem[Shi76]{shimuraCPAM}
Goro Shimura, \emph{The special values of the zeta functions associated with
  cusp forms}, Comm. Pure Appl. Math. \textbf{29} (1976), no.~6, 783--804.

\bibitem[Ski16]{skinner-mult}
Christopher Skinner, \emph{Multiplicative reduction and the cyclotomic main
  conjecture for {${\rm GL}_2$}}, Pacific J. Math. \textbf{283} (2016), no.~1,
  171--200.

\bibitem[Ski20]{skinner}
\bysame, \emph{A converse to a theorem of {G}ross, {Z}agier, and {K}olyvagin},
  Ann. of Math. (2) \textbf{191} (2020), no.~2, 329--354.

\bibitem[Spr12]{sprung-JNT}
Florian E.~Ito Sprung, \emph{Iwasawa theory for elliptic curves at
  supersingular primes: a pair of main conjectures}, J. Number Theory
  \textbf{132} (2012), no.~7, 1483--1506.

\bibitem[Ste89]{stevens-invmath}
Glenn Stevens, \emph{Stickelberger elements and modular parametrizations of
  elliptic curves}, Invent. Math. \textbf{98} (1989), no.~1, 75--106.

\bibitem[SU14]{SU}
Christopher Skinner and Eric Urban, \emph{The {I}wasawa {M}ain {C}onjectures
  for {$GL\sb 2$}}, Invent. Math. \textbf{195} (2014), no.~1, 1--277.

\bibitem[Swe21]{sweeting}
Naomi Sweeting, \emph{Kolyvagin's conjecture, bipartite {E}uler systems, and
  higher congruences of modular forms}, preprint,
  \href{https://arxiv.org/abs/2012.11771}{arxiv.org/abs/2012.11771}.

\bibitem[SZ14]{skinner-zhang}
Christopher Skinner and Wei Zhang, \emph{Indivisibility of {H}eegner points in
  the multiplicative case}, preprint,
  \href{https://arxiv.org/abs/1407.1099}{arXiv:1407.1099}.

\bibitem[Urb14]{urban-rankin}
Eric Urban, \emph{Nearly overconvergent modular forms}, Iwasawa theory 2012,
  Contrib. Math. Comput. Sci., vol.~7, Springer, Heidelberg, 2014,
  pp.~401--441.

\bibitem[Wan14a]{wan-nonord-GU31}
Xin Wan, \emph{{I}wasawa {M}ain {C}onjecture for {R}ankin-{S}elberg $p$-adic
  {$L$}-functions: {N}on-{O}rdinary case}, preprint,
  \href{https://arxiv.org/abs/1412.1767v1}{arXiv:1412.1767v1}.

\bibitem[Wan14b]{wan-ss}
\bysame, \emph{{I}wasawa {M}ain {C}onjecture for {S}upersingular {E}lliptic
  {C}urves}, preprint,
  \href{https://arxiv.org/abs/1411.6352v1}{arXiv:1411.6352v1}.

\bibitem[Wan15]{wan-HMF}
\bysame, \emph{The {I}wasawa main conjecture for {H}ilbert modular forms},
  Forum Math. Sigma \textbf{3} (2015), Paper No. e18, 95.

\bibitem[Wan16]{wan-nonord}
\bysame, \emph{Iwasawa {M}ain {C}onjecture for {N}on-{O}rdinary {M}odular
  {Forms}}, preprint,
  \href{https://arxiv.org/abs/1607.07729}{arXiv:1607.07729}.

\bibitem[Zha14]{zhang-Kolyvagin}
Wei Zhang, \emph{Selmer groups and the indivisibility of {H}eegner points},
  Camb. J. Math. \textbf{2} (2014), no.~2, 191--253.

\end{thebibliography}

\end{document}